\documentclass{article}
\usepackage{amsmath}
\usepackage{amsfonts,amssymb,mathrsfs}
\usepackage{theorem}
\usepackage{ulem}
\usepackage{hyperref}

\usepackage{tikz}
\usepackage{tikz-cd}
\usepgflibrary{arrows}
\usetikzlibrary{matrix}

\newcommand{\Ob}{\mathrm{Ob}}

\newcommand{\cl}{\mathrm{cl}}

\input amssym.def

\numberwithin{equation}{section}

\numberwithin{figure}{section}

\makeatletter
\newcommand\qedsymbol{\hbox{$\Box$}}
\newcommand\qed{\relax\ifmmode\Box\else
  {\unskip\nobreak\hfil\penalty50\hskip1em\null\nobreak\hfil\qedsymbol
  \parfillskip=\z@\finalhyphendemerits=0\endgraf}\fi}

\newenvironment{proof-of}[1][{}]{\par\noindent \textbf{Proof of} {#1}. }{\qed}

\newenvironment{proof}[0]{\par\noindent \textbf{Proof}.}{\qed}

\newcommand{\gray}[1]{\textcolor{gray}{#1}}

\newcommand{\blue}[1]{\textcolor{blue}{#1}}

\DeclareMathOperator*{\Gal}{Gal}

\newcommand{\GT}{\mathsf{GT}}
\newcommand{\GTSh}{\mathsf{GTSh}}
\newcommand{\GTh}{\widehat{\mathsf{GT}}}
\newcommand{\Zhat}{\widehat{\mathbb{Z}}}

\newcommand{\F}{\mathsf{F}}
\newcommand{\Fh}{\widehat{\mathsf{F}}}

\newcommand{\hcP}{\hat{\mathcal{P}}}

\newcommand{\PaB}{\mathsf{PaB}}
\newcommand{\NFI}{\mathsf{NFI}}

\newcommand{\PB}{\mathrm{PB}} 
\newcommand{\B}{\mathrm{B}} 
\newcommand{\N}{\mathsf{N}} 
\newcommand{\K}{\mathsf{K}} 

\newcommand{\sT}{\mathsf{T}}

\newcommand{\paren}{(\hspace{-0.05cm}..\hspace{-0.05cm}(}

\newcommand{\Conf}{{\mathsf{Conf}}}
\newcommand{\Ch}{{\mathsf{Ch}}}

\newcommand{\Hom}{\mathrm{Hom}}
\newcommand{\Isom}{\mathrm{Isom}}
\newcommand{\Aut}{\mathrm{Aut}}

\newcommand{\ord}{\mathrm{ord}}
\newcommand{\isom}{\mathrm{isom}}
\newcommand{\pr}{\mathrm{pr}}

\newcommand{\conn}{\mathrm {conn}}

\newcommand{\id}{\mathrm{id}}

\newcommand{\ti}[1]{{\tilde{#1}}}
\newcommand{\wt}[1]{{\widetilde{#1}}}
\newcommand{\wh}[1]{{\widehat{#1}}}
\newcommand{\ol}[1]{{\overline{#1}}}

\newcommand{\sh}{\sharp}

\newcommand{\hs}{\heartsuit}

\newcommand{\al}{{\alpha}}
\newcommand{\la}{{\lambda}}

\newcommand{\mm}{{\mathfrak{m}}}
\newcommand{\ou}{\mathfrak{ou}}

\newcommand{\mmp}{{\mathfrak{p}}}
\newcommand{\ms}{\mathfrak{s}}

\newcommand{\om}{{\omega}}

\newcommand{\si}{{\sigma}}
\newcommand{\ga}{{\gamma}}
\newcommand{\vf}{{\varphi}}

\newcommand{\ka}{{\kappa}}

\newcommand{\ML}{{\mathcal{ML}}}

\newcommand{\cI}{\mathcal{I}}

\newcommand{\cP}{\mathcal{P}}

\newcommand{\cL}{{\mathcal{L}}}

\newcommand{\cG}{\mathcal{G}}

\newcommand{\cZ}{{\mathcal{Z}}}
\newcommand{\cS}{{\mathcal{S}}}
\newcommand{\cT}{{\mathcal{T}}}

\newcommand{\cO}{\mathcal{O}}

\newcommand{\bbP}{{\mathbb P}}
\newcommand{\bbC}{{\mathbb C}}

\newcommand{\bbZ}{{\mathbb Z}}
\newcommand{\bbQ}{{\mathbb Q}}

\newcommand{\Qbar}{\overline{\mathbb Q}}

\newcommand{\te}{\theta}

\renewcommand{\mod}{\,\,\mathrm{mod}\,\,}

\newcommand{\lan}{\langle\,}
\newcommand{\ran}{\,\rangle}

\date{}
\newtheorem{thm}{Theorem}[section]
\newtheorem{defi}[thm]{Definition}

\newtheorem{cor}[thm]{Corollary}

\newtheorem{prop}[thm]{Proposition}

\newtheorem{quest}[thm]{Question}

\newtheorem{pty}[thm]{Property}

\theorembodyfont{\rm}
\newtheorem{remark}[thm]{Remark}

\title{What are $\GT$-shadows?}

\author{V. A. Dolgushev,  K. Q. Le and A. A. Lorenz}

\date{}

\begin{document}

\large

\maketitle

\begin{abstract}
Let $\B_4$ (resp. $\PB_4$) be the braid group (resp. the pure braid group) on $4$ strands
and $\NFI_{\PB_4}(\B_4)$ be the poset whose objects are finite index normal subgroups $\N$ of $\B_4$
that are contained in $\PB_4$. In this paper, we introduce $\GT$-shadows which 
may be thought of as ``approximations'' to elements of the profinite version $\GTh$ of  
the Grothendieck-Teichmueller group \cite[Section 4]{Drinfeld}. We prove that $\GT$-shadows 
form a groupoid whose objects are elements of $\NFI_{\PB_4}(\B_4)$. We show that $\GT$-shadows 
coming from elements of $\GTh$ satisfy various additional properties and we investigate these
properties. We establish an explicit link between $\GT$-shadows and the group $\GTh$ 
(see Theorem \ref{thm:GTh-ML}). We also present selected results of computer experiments 
on $\GT$-shadows. 
In the appendix of this paper, we give a complete description of $\GT$-shadows in the Abelian 
setting. We also prove that, in the Abelian setting, every $\GT$-shadow comes from an element 
of $\GTh$.  Objects very similar to $\GT$-shadows  
were introduced in paper \cite{HS} by D. Harbater and L. Schneps. 
A variation of the concept of $\GT$-shadows for the coarse version of $\GTh$ was studied in papers 
\cite{Guillot} and \cite{Guillot1} by P. Guillot. 
\end{abstract}

\tableofcontents

\section{Introduction}
\label{sec:intro}
The absolute Galois group $G_{\bbQ}$ of the field $\bbQ$ of rational numbers and the 
Grothendieck-Teichmueller group $\GTh$ introduced by V. Drinfeld in \cite{Drinfeld}
are among the most mysterious objects in mathematics\footnote{This list of references is
far from complete.} \cite{Jordan}, \cite{Fresse1}, 
\cite{GT-outer}, \cite{Ihara-ICM}, \cite{OpenGTLeila}, \cite{Nakamura-Leila}, \cite{IOM}, 
\cite{Leila-lect}. 

Using the outer action of $G_{\bbQ}$ on the algebraic fundamental group of 
$\bbP^1_{\Qbar} \setminus \{0,1,\infty\}$, one can produce a natural group homomorphism 
\begin{equation}
\label{G-Q-to-GTh}
G_{\bbQ} \to \GTh
\end{equation}
and, due to Belyi's theorem \cite{Belyi}, this homomorphism is injective.  
Although both $G_{\bbQ}$ and $\GTh$ are uncountable, it is very hard to produce explicit 
examples of elements in $G_{\bbQ}$ and in $\GTh$. In particular, the famous question on 
surjectivity of \eqref{G-Q-to-GTh} posed by Ihara at his ICM address \cite{Ihara-ICM} is still open. 

The group $G_{\bbQ}$ can be studied by investigating finite degree field extensions of $\bbQ$. 
In fact $G_{\bbQ}$ coincides with the limit of the functor 
that sends a finite degree 
Galois extension $K$ of $\bbQ$ to the Galois group $\Gal(K/\bbQ)$. 
The goal of this paper is to propose a loose analog of such a functor for $\GTh$.  

The most elegant definition of the group $\GTh$ involves (the profinite completion $\wh{\PaB}$ of) 
the operad $\PaB$ of parenthesized braids \cite{BNGT},  \cite[Chapter 6]{Fresse1},  
\cite{Tamarkin}. $\PaB$ is an operad in the category of groupoids that 
is ``assembled from'' braid groups  $\B_n$ for all $n \ge 1$.  
The objects of $\PaB(n)$ are words of the free magma generated by symbols
$1,2, \dots, n$ in which each generator appears exactly once. 
For example, $\PaB(3)$ has exactly $12$ objects: 
$
(12)3,~ (21)3,~ (23)1,~ (32)1,~ (31)2,~ (13)2,$ 
$ 1(23),~ 2(13),~ 2(31),~ 3(21),~ 3(12),~ 1(32).$ 
For every $n \ge 2$ and every object $\tau$ of $\PaB$, we have 
$$
\Aut_{\PaB(n)}(\tau) = \PB_n\,,
$$  
where $\PB_n$ is the pure braid group on $n$ strands. 

As an operad in the category of groupoids, $\PaB$ is generated by 
these two morphisms:
\begin{equation}
\label{beta-al}
\begin{tikzpicture}[scale=1.5, > = stealth]
\tikzstyle{v} = [circle, draw, fill, minimum size=0, inner sep=1]
\draw (-0.7,0.5) node[anchor=center] {{$\beta ~ : = $}};
\node[v] (v1) at (0, 0) {};
\draw (0,-0.2) node[anchor=center] {{\small $1$}};
\draw (0,1.2) node[anchor=center] {{\small $2$}};
\node[v] (v2) at (1, 0) {};
\draw (1,-0.2) node[anchor=center] {{\small $2$}};
\draw (1,1.2) node[anchor=center] {{\small $1$}};
\node[v] (vv2) at (0, 1) {};
\node[v] (vv1) at (1, 1) {};
\draw [->] (v1) -- (vv1);
\draw (v2) -- (0.6, 0.4); 
\draw [->] (0.4, 0.6) -- (vv2);
\begin{scope}[shift={(3.5,0)}]
\draw (-0.7,0.5) node[anchor=center] {{$\al ~ : = $}};
\node[v] (v1) at (0, 0) {};
\draw (0,-0.2) node[anchor=center] {{\small $(1$}};
\node[v] (v2) at (0.5, 0) {};
\draw (0.5,-0.2) node[anchor=center] {{\small $2)$}};
\node[v] (v3) at (1.5, 0) {};
\draw (1.5,-0.2) node[anchor=center] {{\small $3$}};
\node[v] (vv1) at (0, 1) {};
\draw (0,1.2) node[anchor=center] {{\small $1$}};
\node[v] (vv2) at (1, 1) {};
\draw (1,1.2) node[anchor=center] {{\small $(2$}};
\node[v] (vv3) at (1.5, 1) {};
\draw (1.5,1.2) node[anchor=center] {{\small $3)$}};
\draw [->] (v1) -- (vv1);  \draw [->] (v2) -- (vv2);  \draw [->] (v3) -- (vv3); 
\end{scope}
\end{tikzpicture}
\end{equation}
Moreover, any relation on $\beta$ and $\al$ in $\PaB$ is a consequence of
the pentagon relation and the two hexagon relations 
(see \eqref{pentagon}, \eqref{hexagon1} and \eqref{hexagon11} in 
Appendix \ref{app:oper-PaB}). The hexagon relations 
come from two ways of connecting $(12)3$ to $3(12)$ and two ways of connecting 
$1(23)$ to $(23)1$ in $\PaB(3)$. Similarly, the pentagon relation comes from two ways 
of connecting $((12)3)4$ to $1(2(34))$ in $\PaB(4)$. 
For more details about the operad $\PaB$ and its profinite completion $\wh{\PaB}$, 
see Appendix \ref{app:PaB}. 

By definition, $\GTh$ is the group $\Aut(\wh{\PaB})$ of (continuous) 
automorphisms\footnote{We tacitly assume that our automorphisms act as identity on objects.} 
of the profinite completion $\wh{\PaB}$ of $\PaB$.

Since the morphisms $\beta$ and $\al$ from \eqref{beta-al} are topological 
generators of $\wh{\PaB}$, every $\hat{T}\in \GTh$ is uniquely determined by 
its values
\begin{equation}
\label{hat-T-beta-alpha}
\hat{T}(\beta) \in \Hom_{\wh{\PaB}}((1,2), (2,1)), \qquad 
\hat{T}(\al) \in \Hom_{\wh{\PaB}}((1,2)3, 1(2,3)).
\end{equation}
Moreover, since $\Aut_{\wh{\PaB}}((1,2)3) = \wh{\PB}_3$, 
$\Aut_{\wh{\PaB}}((1,2)) = \wh{\PB}_2$ and  $\wh{\PB}_2 \cong \Zhat$, 
the underlying set of $\GTh$ can be identified with the subset of 
pairs $(\hat{m}, \hat{f}) \in \wh{\bbZ} \times \wh{\PB}_3$
satisfying some relations and technical conditions. 

Recall that $\PB_3$ is isomorphic to the direct product $\F_2\times \bbZ$ of 
the free group $\F_2$ on two generators and the infinite cyclic group. 
The $\F_2$-factor is generated by the two standard generators $x_{12}$, $x_{23}$
and the $\bbZ$-factor is generated by the element 
$c:= x_{23} x_{12} x_{13}$. In this paper, we tacitly identify $\F_2$ (resp. its profinite completion $\wh{\F}_2$) 
with the subgroup $\lan x_{12}, x_{23} \ran \le \PB_3$ (resp. the topological closure 
of $\lan x_{12}, x_{23} \ran$ in $\wh{\PB}_3$). Occasionally, we denote the standard generators 
of $\F_2$ by $x$ and $y$. 

One can show\footnote{This statement can also be found in many introductory papers on $\GTh$.}
(see, for example, Corollary \ref{cor:comm-F2} in Section \ref{sec:pairs-shadows}
of this paper) that, for every $\hat{T}\in \GTh$, the corresponding element $\hat{f} \in \wh{\PB}_3$ belongs to 
the topological closure $([\Fh_2, \Fh_2])^{\cl}$ of the commutator subgroup of $\wh{\F}_2$. 

\begin{remark}  
\label{rem:GTh-different}
Due to Proposition \ref{prop:T-hat-F2}, the restriction of every (continuous) automorphism 
$\hat{T} \in \Aut(\wh{\PaB})$ to  $\wh{\F}_2 \le \wh{\PB_3} = \Aut_{\wh{\PaB}}((1,2)3)$
gives us an automorphism of $\wh{\F}_2$. 
In fact, many authors introduce $\GTh$ as the subgroup of (continuous) 
automorphisms of $\wh{\F}_2$ of the form 
$$
x \mapsto x^{\hat{\la}}, \qquad 
y \mapsto \hat{f}^{-1} y^{\hat{\la}} \hat{f}\,,
$$
where the pair $(\hat{\la}, \hat{f}) \in \Zhat^{\times} \times ([\Fh_2, \Fh_2])^{\cl}$ satisfies certain cocycle relations 
and the ``invertibility condition.'' Another equivalent definition of $\GTh$ is based on the use 
of the outer automorphisms of the profinite completions of the pure mapping class groups. 
For more details about this definition, we refer the reader to \cite[Main Theorem]{GT-outer}. 
\end{remark}
\begin{remark}  
\label{rem:subtle-prop}
It is known (see \cite[Theorem 2]{Leila-Sasha-Coh}) that, for every $(\hat{m}, \hat{f}) \in \GTh$, 
the element $\hat{f}$ satisfies further rather subtle properties. It would be interesting to investigate 
whether $\GT$-shadows satisfy consequences of these properties.
\end{remark}

\subsection{The link between $G_{\bbQ}$ and $\GTh$}
\label{sec:GQ-GTh}

For completeness, we briefly recall here the link between the absolute Galois group $G_{\bbQ}$ of 
rationals and the Grothendieck-Teichmueller group $\GTh$. 

Applying the basic theory of the algebraic fundamental group \cite{SGA1}, \cite[Section 5.6]{Szamuely} to
$$
\bbP^1_{\Qbar} \setminus \{0,1,\infty\},
$$ 
we get an outer action of the absolute Galois group $G_{\bbQ}$ on $\wh{\F}_2$. 
Using the fact that this action preserves the 
inertia subgroups, we can lift this outer action to an honest action of the form
\begin{equation}
\label{G-Q-acts}
g(x)  = x^{\chi(g)}, \qquad g(y)  = \hat{f}_g(x,y)^{-1} y^{\chi(g)} \hat{f}_g(x,y), \qquad g \in G_{\bbQ}\,,
\end{equation}
where $\chi: G_{\bbQ} \to  \wh{\bbZ}^{\times}$ is the cyclotomic character and 
$\hat{f}_g(x,y)$ is an element of $([\Fh_2, \Fh_2])^{\cl}$ that depends only on $g$. 

It is known \cite[Section 4]{Drinfeld}, \cite[Section 3]{Ihara-ICM},  
\cite[Theorem 4.7.7]{Szamuely}, \cite[Fact 4.7.8]{Szamuely} that, 
\begin{itemize}

\item $\forall$ $g \in G_{\bbQ}$, the pair 
$\big(\, (\chi(g)-1)/2, \hat{f}_g(x,y) \big) \in \wh{\bbZ} \times \Fh_2 $ 
defines an element of $\GTh$;
 
\item the assignment 
$
g \in G_{\bbQ} ~\mapsto~ \big(\, (\chi(g)-1)/2, \hat{f}_g(x,y) \big)  \in \wh{\bbZ} \times \wh{F}_2 
$
defines the group homomorphism \eqref{G-Q-to-GTh}. 

\item finally, using Belyi's theorem \cite{Belyi}, one can prove that the homomorphism \eqref{G-Q-to-GTh} is injective.
 
\end{itemize}
For more details, we refer the reader to \cite{Ihara-embedding}.

\subsection{The groupoid $\GTSh$ of $\GT$-shadows and its link to $\GTh$}
\label{sec:GTSh-intro}

Let us denote by $\PaB^{\le 4}$ the truncation of the operad $\PaB$ up to arity $4$, i.e. 
$$
\PaB^{\le 4} : = \PaB(1) \sqcup  \PaB(2) \sqcup  \PaB(3) \sqcup  \PaB(4). 
$$
Moreover, let $\NFI_{\PB_4}(\B_4)$ be the poset of finite index normal subgroups
$\N \lhd \B_4$ such that $\N \le \PB_4$. 

To every $\N \in \NFI_{\PB_4}(\B_4)$, we assign an equivalence relation $\sim_{\N}$ 
on $\PaB^{\le 4}$ that is compatible with the structure of the truncated operad 
and the composition of morphisms. For every $\N \in \NFI_{\PB_4}(\B_4)$, the quotient 
$$
\PaB^{\le 4}/\sim_{\N}
$$
is a truncated operad in the category of {\it finite} groupoids. 

In this paper, we introduce a groupoid $\GTSh$ whose objects are elements 
of $\NFI_{\PB_4}(\B_4)$. Morphisms from $\ti{\N}$ to $\N$ are isomorphisms 
of truncated operads
\begin{equation}
\label{morph}
\PaB^{\le 4} /\sim_{\ti{\N}}  ~\stackrel{\cong}{\longrightarrow}~ \PaB^{\le 4} /\sim_{\N}\,. 
\end{equation}
We call these isomorphisms \emph{$\GT$-shadows}.

Just as $\PaB$, the truncated operad $\PaB^{\le 4}$ is generated by the braiding $\beta \in \PaB(2)$ 
and the associator $\al \in \PaB(3)$. Hence morphisms of $\GTSh$ to $\N \in \NFI_{\PB_4}(\B_4)$ 
are in bijection with pairs 
\begin{equation}
\label{pairs-intro}
(m + N_{\ord} \bbZ, f \N_{\PB_3}) ~\in~ \bbZ / N_{\ord} \bbZ \times \PB_3 / \N_{\PB_3}\,,
\end{equation}
that satisfy appropriate versions of the hexagon relations, the pentagon 
relation and some technical conditions.  
Here, the integer $N^{\ord}$ and the (finite index) normal subgroup $\N_{\PB_3} \unlhd \PB_3$ are obtained 
from $\N$ via a precise procedure described in Subsection \ref{sec:fromN-in-PB4}.

We denote by $\GT(\N)$ the set of such pairs \eqref{pairs-intro} and identify them 
with $\GT$-shadows whose target is $\N$. From now on, we denote by $[(m,f)]$ the 
$\GT$-shadow represented by a pair $(m,f) \in \bbZ \times \PB_3$. 
  
A $\GT$-shadow $[(m,f)] \in \GT(\N)$ is called \emph{genuine} if there exists 
an element $\hat{T} \in \GTh$ such that the diagram 
\begin{equation}
\label{T-hat-diag}
\begin{tikzpicture}
\matrix (m) [matrix of math nodes, row sep=1.5em, column sep=1.5em]
{\wh{\PaB}^{\le 4} & \wh{\PaB}^{\le 4} \\
 \PaB^{\le 4}/ \sim_{\ti{\N}}  &  \PaB^{\le 4}/ \sim_{\N}\,, \\};
\path[->, font=\scriptsize]
(m-1-1) edge node[above] {$\hat{T}$} (m-1-2) edge (m-2-1)
(m-2-1) edge node[above] {$\cong$} (m-2-2) (m-1-2) edge (m-2-2);
\end{tikzpicture}
\end{equation}
commutes. In \eqref{T-hat-diag}, 
the lower horizontal arrow is the isomorphism corresponding to 
$[(m,f)]$ and the vertical arrows are the canonical projections.  
If such $\hat{T}$ does not exist, we say that the $\GT$-shadow $[(m,f)]$ is 
\emph{fake}\footnote{This name was suggested to the authors by David Harbater.}. 

In this paper, we show that genuine $\GT$-shadows satisfy additional conditions. 
For example, every genuine $\GT$-shadow in $\GT(\N)$ 
can be represented by a pair $(m,f)$ 
with\footnote{It should be mentioned that, in the computer implementation \cite{package-GT}, 
we only considered $\GT$-shadows of the form $[(m,f)]$ with $f \in \F_2 \le \PB_3$.} 
\begin{equation}
\label{f-in-comm-F2}
f \in [\F_2, \F_2], 
\end{equation}
where $[\F_2, \F_2]$ is the commutator subgroup of $\F_2 \le \PB_3$. 

A $\GT$-shadow $[(m,f)]$ satisfying all these additional conditions (see Definition \ref{dfn:charm}) 
is called \emph{charming}. In this paper, we show that charming $\GT$-shadows form a subgroupoid 
of $\GTSh$ and we denote this subgroupoid by $\GTSh^{\hs}$. 

The groupoid $\GTSh^{\hs}$ is highly disconnected. However, it is easy to see that, 
for every $\N \in \NFI_{\PB_4}(\B_4)$, the connected component $\GTSh^{\hs}_{\conn}(\N)$ is 
a finite groupoid (see Proposition \ref{prop:GTSh-conn-finite}).
In all examples we have considered so far (see \cite{package-GT} and Section \ref{sec:comp-exp} 
of this paper),  $\GTSh^{\hs}_{\conn}(\N)$ has at most two objects and, for many of examples of 
$\N \in \NFI_{\PB_4}(\B_4)$ the groupoid $\GTSh_{\conn}(\N)$ has exactly one object 
(i.e. $\GT(\N)$ is a group). Such elements of $\NFI_{\PB_4}(\B_4)$ play a special role and
we call them \emph{isolated}. We denote by $\NFI^{isolated}_{\PB_4}(\B_4)$ the subposet 
of isolated elements of $\NFI_{\PB_4}(\B_4)$.

In this paper, we show that the subposet $\NFI^{isolated}_{\PB_4}(\B_4)$ is cofinal 
(i.e., for every $\N \in \NFI_{\PB_4}(\B_4)$, there exists $\K \in \NFI^{isolated}_{\PB_4}(\B_4)$ such
that $\K \le \N$). 
We show that the assignment $\N \mapsto \GT(\N)$ upgrades to a functor $\ML$
from the poset $\NFI^{isolated}_{\PB_4}(\B_4)$ to the category of finite groups and we prove that 
the limit of this functor is precisely the Grothendieck-Teichmueller group $\GTh$
(see Theorem \ref{thm:GTh-ML}).

\begin{remark}  
\label{rem:Guillot}
Recall \cite{GT-outer} that, omitting the pentagon relation from the definition of $\GTh$, we get    
the coarse version $\GTh_0$ of the Grothendieck-Teichmueller group. 
It is not hard to show that $\GTh_0$ is the group 
of continuous automorphisms of the truncated operad $\wh{\PaB}^{\le 3}$
and $\GTh$ is a subgroup of $\GTh_0$.
In papers \cite{Guillot} and \cite{Guillot1}, P. Guillot studies a variant of $\GT$-shadows for 
this coarse version $\GTh_0$ of the Grothendieck-Teichmueller group. 
\end{remark}

\subsection{Organization of the paper}
\label{sec:org}
In Section \ref{sec:pairs-shadows}, we introduce the poset of compatible equivalence 
relations on the truncated operad $\PaB^{\le 4}$, and we show that  $\NFI_{\PB_4}(\B_4)$
can be identified with the subposet of this poset. We introduce the concept 
of $\GT$-pair and show that $\GT$-pairs coming from elements of $\GTh$ 
satisfy certain conditions. This consideration motivates the 
concept of $\GT$-shadow (see Definition \ref{dfn:GT-shadows}). 
We prove that $\GT$-shadows form a groupoid $\GTSh$: objects 
of this groupoid are elements of  $\NFI_{\PB_4}(\B_4)$ and morphisms are 
$\GT$-shadows. 

In Section \ref{sec:pairs-shadows}, we also investigate further conditions on $\GT$-shadows 
coming from elements of $\GTh$, introduce charming $\GT$-shadows and prove that 
charming $\GT$-shadows form a subgroupoid of $\GTSh$. In this section, we introduce 
a natural functor $\Ch_{cyclot}$ from $\GTSh$ to the category of finite cyclic groups. 
We call this functor \emph{the virtual cyclotomic character}.

In Section \ref{sec:ML}, we introduce an important subposet $\NFI^{isolated}_{\PB_4}(\B_4)$
of $\NFI_{\PB_4}(\B_4)$ and construct a functor $\ML$ from  $\NFI^{isolated}_{\PB_4}(\B_4)$
to the category of finite groups. In this section, we prove that the limit of the functor $\ML$ is 
precisely the Grothendieck-Teichmueller group $\GTh$. 

In Section \ref{sec:comp-exp}, we present selected results of computer experiments.
We outline the basic information about 35 selected elements of $\NFI_{\PB_4}(\B_4)$
and the corresponding connected components of the groupoid $\GTSh$. We say a few words 
about selected remarkable examples. 
Finally, we discuss two versions of the Furusho property (see Properties \ref{P:Furusho-strong}
and \ref{P:Furusho-weak}) and list selected open questions.   

In Appendix \ref{app:PaB}, we give a brief reminder of (pure) braid groups, 
the operad $\PaB$ and its completion.  

In Appendix \ref{app:Abelian}, we give a complete description of charming $\GT$-shadows in the 
Abelian setting and we prove that, in this setting, every charming $\GT$-shadow is genuine
(see Theorem \ref{thm:Abelian}).

\subsection{Notational conventions}
For a set $X$ with an equivalence relation and $a \in X$ we will denote by $[a]$
the equivalence class that contains the element $a$. For a groupoid $\cG$, the notation 
$\ga \in \cG$ means that $\ga$ is a {\it morphism} of this groupoid. 
 
Every finite group is tacitly considered with the discrete topology. For a group $G$, $\hat{G}$ 
denotes the profinite completion \cite{RZ-profinite} of $G$.
The notation $[G, G]$ is reserved for the commutator subgroup of $G$. 
For a normal subgroup $H\unlhd G$ of finite index, we denote by $\NFI_{H}(G)$ the poset 
of finite index normal subgroups $N$ in $G$ such that $N \le H$. Moreover, 
$\NFI(G) := \NFI_G(G)$, i.e. $\NFI(G)$ is the poset of normal finite index subgroups of a group $G$.

For a group $G$ and elements $\K \le \N$ of the poset $\NFI(G)$, 
the notation $\cP_{\N}$ (resp. $\cP_{\K, \N}$)  is reserved for the reduction
homomorphism $G \to G/\N$ (resp.  $G/\K \to G/\N$). The notation $\hat{\cP}_{\N}$
is reserved for the canonical (continuous) homomorphism from $\hat{G}$ to $G/\N$. 
Similar notation is used for the canonical functors to finite quotients of a groupoid. 

The notation $\B_n$ (resp. $\PB_n$) is reserved for the Artin braid group 
on $n$ strands (resp. the pure braid group on $n$ strands). $S_n$ denotes the 
symmetric group on $n$ letters. The standard generators of $\B_n$ are 
denoted by $\si_1, \dots, \si_{n-1}$ and the standard generators of $\PB_n$
are denoted by $x_{ij}$ (for $1 \le i < j \le n$). We will tacitly identify the 
free group $\F_2$ on two generators with the subgroup 
$\lan x_{12}, x_{23} \ran$ of $\PB_3$.

We will freely use the language of operads \cite[Section 3]{notes}, \cite[Chapter 1]{Fresse1}, 
\cite{LV-operads}, \cite{Markl-Stasheff}, \cite{Jim-operad}. In this paper, we work with operads 
in the category of sets and in the category of (topological) groupoids. The category of topological 
groupoids is understood in the ``strict sense.'' For example, the associativity axioms for 
the {\it elementary insertions}\footnote{In the literature, elementary insertions are sometimes called {\it partial 
compositions.}} $\circ_i$ (for operads in the category of groupoids)
are satisfied ``on the nose.''   

For an integer $q \ge 1$, a {\it $q$-truncated operad} in the category of groupoids is 
a collection of groupoids $\{ \cG(n) \}_{1 \le n \le q}$ such that 
\begin{itemize}

\item For every $1 \le n \le q$, the groupoid $\cG(n)$ is equipped with an action of $S_n$.

\item For every triple of integers $i,n,m$ such that $1 \le i  \le n$, $n,m, n+m -1 \le q$
we have functors 
\begin{equation}
\label{circ-i-fun}
\circ_i :  \cG(n) \times  \cG(m) \to  \cG(n+m -1).
\end{equation}

\item The axioms of the operad for $\{ \cG(n) \}_{1 \le n \le q}$ are satisfied 
in the cases where all the arities are $\le q$. 

\end{itemize}

For every operad $\cO$ and every integer $q \ge 1$, the disjoint union 
$\displaystyle \cO^{\le q} : = \bigsqcup_{n=0}^q \cO(n)$
is clearly a $q$-truncated operad. In this paper, we only consider $4$-truncated operads. 
So we will simply call them {\it truncated operads}. 
  
The operad $\PaB$ of parenthesized braids, its truncation $\PaB^{\le 4}$ and 
its completion $\wh{\PaB}^{\le 4}$ play the central role in this paper. See Appendix \ref{app:PaB} 
for more details.

\bigskip
\noindent
{\bf Acknowledgement.} We are thankful to Benjamin Collas, David Harbater, Julia Hartmann, 
Florian Pop, Leila Schneps, Dmitry Vaintrob and 
John Voight for useful discussions. V.A.D. is thankful to Pavol Severa for showing him the works 
of Pierre Guillot. V.A.D. discussed vague ideas of the construction presented in this paper 
during a walk near Z\"urich with Thomas Willwacher in October of 2016. 
V.A.D. is thankful to Thomas for leaving him with the question ``What are $\GT$-shadows?'' and for giving 
the title to this paper! V.A.D. is especially thankful to Leila Schneps for her unbounded enthusiasm about this
project and its possible continuations. V.A.D. is also thankful to Leila Schneps for her suggestion 
to modify the original definition of charming $\GT$-shadows, her comments about the introduction, 
and her input concerning hypothetical versions of Furusho's property. 
A.A.L. acknowledges both the Temple University Honors Program and 
the Undergraduate Research Program for their active support of undergraduate researchers.
V.A.D. and K.Q.L. acknowledge a partial support from NSF grant DMS-1501001.

\section{$\GT$-pairs and $\GT$-shadows}
\label{sec:pairs-shadows}

\subsection{The poset of compatible equivalence relations on $\PaB^{\le 4}$}
\label{sec:poset-equiv}

An equivalence relation $\sim$ on the disjoint union of groupoids\footnote{Recall that $\PaB(0)$ is the 
empty groupoid.} 
$$
\PaB^{\le 4} = \PaB(1) \sqcup  \PaB(2) \sqcup \PaB(3) \sqcup \PaB(4)
$$
is an equivalence relation on the set of morphisms of $\PaB^{\le 4}$ such that, if 
$\ga \sim \ti{\ga}$, then the source (resp. the target) of $\ga$ coincides with the source 
(resp. the target) of $\ti{\ga}$. 
In particular, $\ga \sim \ti{\ga}$ implies that $\ga$ and $\ti{\ga}$ have the same arity. 

\begin{defi}  
\label{dfn:compat-equiv}
An equivalence relation $\sim$ on $\PaB^{\le 4}$ is called \emph{compatible} if
\begin{itemize}

\item for every pair of composable morphisms $\ga, \ti{\ga} \in \PaB(n)$ the equivalence class of 
the composition $\ga \cdot \ti{\ga} $ depends only on the equivalence classes of 
$\ga$ and $\ti{\ga}$;

\item for every $\ga, \ti{\ga} \in \PaB(n)$ and every $\te \in S_n$ 
$$
\ga \sim \ti{\ga} ~~~\Leftrightarrow~~~ \te(\ga) \sim \te(\ti{\ga});
$$ 

\item for every tuple of integers $i,n,m$, $1\le i \le n$, $n,m, n+m-1 \le 4$, and 
every $\ga_1, \ti{\ga}_1 \in \PaB(n)$,  $\ga_2, \ti{\ga}_2 \in \PaB(m)$ we have 
$$
\ga_1 \sim \ti{\ga}_1 ~~\Rightarrow~~ \ga_1 \circ_i \ga_2 \sim \ti{\ga}_1 \circ_i \ga_2\,,
\qquad
\ga_2 \sim \ti{\ga}_2 ~~\Rightarrow~~ \ga_1 \circ_i \ga_2 \sim \ga_1 \circ_i \ti{\ga}_2\,.
$$

\end{itemize}
\end{defi}  

It is clear that, for every compatible equivalence relation $\sim$ on $\PaB^{\le 4}$, the set 
\begin{equation}
\label{PaB-quotient}
\PaB^{\le 4} /\sim 
\end{equation}
of equivalence classes of morphisms in $\PaB^{\le 4}$ is a truncated operad in 
the category of groupoids. The set of objects of \eqref{PaB-quotient}
coincides with the set of objects of $\PaB^{\le 4}$. The action of symmetric groups and the
elementary insertions are defined by the formulas
$$
\te(\,  [\ga]\, ) : = [\te (\ga)]\,, \qquad \te \in S_n, ~~ \ga \in \PaB(n), 
$$
$$
[\ga_1] \circ_i [\ga_2] : = [\ga_1 \circ_i \ga_2], \qquad \ga_1 \in \PaB(n), ~~\ga_2 \in \PaB(m). 
$$

The conditions of Definition \ref{dfn:compat-equiv} guarantee that the composition of morphisms, 
the action of 
the symmetric groups on $\PaB(n)/\sim$ and the elementary operadic insertions are 
well defined. The axioms of the (truncated) operad follow directly from their 
counterparts for $\PaB^{\le 4}$. 

Compatible equivalence relations on $\PaB^{\le 4}$ form a poset with the following 
obvious partial order: we say that $\sim_1 \le \sim_2$ if $\sim_1$ is finer than 
$\sim_2$, i.e.
$$
\ga \sim_1 \ti{\ga} ~\Rightarrow~ \ga \sim_2 \ti{\ga}.  
$$ 

It is clear that, for every pair of compatible equivalence relations $\sim_1, \sim_2$ on $\PaB^{\le 4}$
such that $\sim_1 \le \sim_2$, we have a natural onto morphism of truncated operads 
\begin{equation}
\label{map-of-quotients}
\cP_{\sim_1, \sim_2} : \PaB^{\le 4}/\sim_1 ~\to~ \PaB^{\le 4}/\sim_2\,.
\end{equation}
Moreover, the assignment $\sim ~\mapsto~ \PaB^{\le 4}/\sim$ upgrades to a functor from the poset of 
compatible equivalence relations to the category of truncated operads. 

For every compatible equivalence relation $\sim$ on $\PaB^{\le 4}$, we denote by 
$\cP_{\sim}$ the natural (onto) morphism of truncated operads:
\begin{equation}
\label{to-quotient}
\cP_{\sim} : \PaB^{\le 4} ~\to~ \PaB^{\le 4}/\sim\,.
\end{equation}

\subsection{From $\NFI_{\PB_4}(\B_4)$ to the poset of compatible equivalence relations}
\label{sec:fromN-in-PB4}

In this paper, we mostly consider compatible equivalence relations for which the 
set of morphisms of \eqref{PaB-quotient} is finite and a large supply of such compatible 
equivalence relations come from elements of the poset $\NFI_{\PB_4}(\B_4)$. 

For $\N \in \NFI_{\PB_4}(\B_4)$, we set 
\begin{equation}
\label{N-PB-3}
\N_{\PB_3} : = \vf_{123}^{-1} (\N) \cap \vf_{12,3,4}^{-1}(\N) \cap
\vf_{1,23,4}^{-1}(\N)  \cap \vf_{1,2,34}^{-1}(\N) \cap \vf_{234}^{-1}(\N)
\end{equation}
and 
\begin{equation}
\label{N-PB-2}
\N_{\PB_2} : = \vf_{12}^{-1} (\N_{\PB_3}) \cap \vf_{12,3}^{-1}(\N_{\PB_3}) \cap
\vf_{1,23}^{-1}(\N_{\PB_3})  \cap \vf_{23}^{-1}(\N_{\PB_3}),
\end{equation}
where $\vf_{123}, ~\vf_{12,3,4},~ \vf_{1,23,4}, ~\vf_{1,2,34}, ~\vf_{234}$ are the homomorphisms 
defined in \eqref{vfs-PB} and $\vf_{12}, ~\vf_{12,3},~ \vf_{1,23}, ~\vf_{23}$ are the homomorphisms 
defined in \eqref{vfs-PB-2-3} (see also the explicit formulas in \eqref{vfs-3-4-on-gen} and 
\eqref{vfs-2-3-on-gen}).  
 
We claim that 
\begin{prop}
\label{prop:N-PB3-B3}
For every $\N \in \NFI_{\PB_4}(\B_4)$, the subgroup $\N_{\PB_3}$ 
(resp. $\N_{\PB_2}$) is an element of the poset $\NFI_{\PB_3}(\B_3)$ 
(resp. the poset $\NFI_{\PB_2}(\B_2)$). 
\end{prop}
\begin{proof}
Since every subgroup of $\B_2$ is normal and $\N_{\PB_2}$ has a finite index in $\PB_2$,
the statement about  $\N_{\PB_2}$ is obvious. 

It is also easy to see that $\N_{\PB_3}$ is a subgroup of finite index in $\PB_3$. 
So it suffices to prove that 
$$
g\, \N_{\PB_3} \, g^{-1} \le \N_{\PB_3} \qquad \forall~g \in B_3. 
$$

Let $h \in \N_{\PB_3}$ and $g \in \B_3$. Then 
\begin{equation}
\label{vf-to-conj}
\vf_{1,23,4} ( g \cdot h \cdot g^{-1} ) = \ou ( \mm( g \cdot h \cdot g^{-1}) \circ_2 \id_{12} ). 
\end{equation}

Using identity \eqref{mm-comp}, we get
$$
\mm( g \cdot h \cdot g^{-1}) = \te(\mm(g)) \cdot \te(\chi) \cdot \mm(g^{-1})\,,
$$
where $\te = \rho(g)$ and $\chi : = \mm(h)$. 

Therefore
\begin{equation}
\label{mm-unfold}
\mm( g \cdot h \cdot g^{-1}) \circ_2 \id_{12} = 
 (\te(\mm(g))  \circ_2 \id_{12}) \cdot (\te(\chi) \circ_2 \id_{12}) \cdot 
 (\mm(g^{-1}) \circ_2 \id_{12}).
\end{equation}

Combining \eqref{vf-to-conj} with \eqref{mm-unfold}, we get 
\begin{equation}
\label{vf-to-conj1}
\vf_{1,23,4} ( g \cdot h \cdot g^{-1} ) = 
\ou(\te(\mm(g))  \circ_2 \id_{12}) \cdot \ou(\te(\chi) \circ_2 \id_{12}) \cdot 
 \ou(\mm(g^{-1}) \circ_2 \id_{12}).
\end{equation}

Since 
$$
\ou(\te(\mm(g))  \circ_2 \id_{12})
\cdot \ou(\mm(g^{-1}) \circ_2 \id_{12})
= 
\ou\big( 
\te(\mm(g))  \circ_2 \id_{12} 
\, \cdot\,  
\mm(g^{-1}) \circ_2 \id_{12} 
\big)=
$$
$$
\ou\big( 
( \te(\mm(g)) 
\, \cdot\,  
\mm(g^{-1}) ) \circ_2 \id_{12} 
\big) = 
\ou\big( 
\mm(g \cdot g^{-1}) \, \circ_2\, \id_{12} 
\big) =\ou(\id_{(1(23))4}) = 1_{\B_4}\,, 
$$
the element $\vf_{1,23,4} ( g \cdot h \cdot g^{-1} ) \in \B_4$ can be rewritten as
$$
\vf_{1,23,4} ( g \cdot h \cdot g^{-1} ) =  \ti{g} \cdot \ou(\te(\chi) \circ_2 \id_{12}) \cdot  \ti{g}^{-1},
$$
where 
$$
\ti{g}  : = \ou(\te(\mm(g))  \circ_2 \id_{12}). 
$$
Thus it remains to prove that 
\begin{equation}
\label{stuff-chi-in-N}
\ou(\te(\chi) \circ_2 \id_{12})\in \N.
\end{equation}

For this purpose, we consider the three possible cases: $\te(1)= 2$, $\te(2)= 2$ and $\te(3)=2$:
\begin{itemize}

\item If $\te(1)= 2$ then $\ou(\te(\chi) \circ_2 \id_{12}) = \vf_{12,3,4}(h)$ and 
\eqref{stuff-chi-in-N} is a consequence of $h \in  \vf_{12,3,4}^{-1} (\N)$. 

\item If $\te(2)= 2$ then $\ou(\te(\chi) \circ_2 \id_{12}) = \vf_{1,23,4}(h)$ and
\eqref{stuff-chi-in-N} is a consequence of $h \in  \vf_{1,23,4}^{-1} (\N)$.

\item If $\te(3)= 2$ then $\ou(\te(\chi) \circ_2 \id_{12}) = \vf_{1,2,34}(h)$ and
\eqref{stuff-chi-in-N} is a consequence of $h \in  \vf_{1,2,34}^{-1} (\N)$. 

\end{itemize}

We proved that the element $g h g^{-1}$ belongs to $\vf_{1,23,4}^{-1}(\N)\subset \PB_3$. 
The proofs for the remaining $4$ homomorphisms $\vf_{123}, \vf_{12,3,4}, \vf_{1,2,34}$ and $\vf_{234}$ are similar 
and we omit them.
\end{proof}

It is clear that $\N_{\PB_2}  = \lan x_{12}^{N_{\ord}} \ran$, where $N_{\ord}$ is the index of $\N_{\PB_2}$ in $\PB_2$, i.e. 
$N_{\ord}$ is the least common multiple of orders of $x_{12} \N_{\PB_3}$, $x_{23} \N_{\PB_3}$, 
$x_{12} x_{13} \N_{\PB_3}$ and $x_{13} x_{23} \N_{\PB_3}$ in  $\PB_3 / \N_{\PB_3}$.

Using the identities $x_{12} x_{13} =  x_{23}^{-1} c$, $x_{13} x_{23} = x_{12}^{-1} c$ involving 
the generator $c$ (see \eqref{c-PB3-B3}) of the center of $\PB_3$, it is easy to prove the
following statement: 
\begin{prop}
\label{prop:N-ord}
Let $\N_{\PB_2} = \lan x_{12}^{N_{\ord}} \ran $ be the subgroup of $\PB_2$ defined in \eqref{N-PB-2}. 
Then $N_{\ord}$ coincides with 
\begin{enumerate}

\item the  least common multiple of orders of elements 
$x_{12} \N_{\PB_3}$, $x_{23} \N_{\PB_3}$ and  $x_{12} x_{13} \N_{\PB_3}$; 

\item the  least common multiple of orders of elements 
$x_{12} \N_{\PB_3}$, $x_{23} \N_{\PB_3}$ and  $x_{13} x_{23} \N_{\PB_3}$; and 

\item the  least common multiple of orders of elements 
$x_{12} \N_{\PB_3}$, $x_{23} \N_{\PB_3}$ and  $c\, \N_{\PB_3}$ 

\end{enumerate}
\end{prop}
\qed

\bigskip

Given $\N \in \NFI_{\PB_4}(\B_4)$ and the corresponding normal subgroups $\N_{\PB_3}$ and $\N_{\PB_2}$, 
we will now define an equivalence relation $\sim_{\N}$ on the set of morphisms in $\PaB^{\le 4}$.  
 
The groupoid $\PaB(1)$ has exactly one object and exactly one (identity) morphism. 
So $\PaB(1)$ has only one equivalence relation. 

For two isomorphisms $(2 \le n \le 4)$
$$
\ga, \ti{\ga} \in \Hom_{\PaB(n)}(\tau_1, \tau_2),
$$
we declare that $\ga \sim_{\N} \ti{\ga}$ if and only if 
\begin{equation}
\label{ou-loop}
\ou (\ga^{-1} \cdot \ti{\ga}) \in  \N_{\PB_n}\,,
\end{equation}
where $\N_{\PB_4} : =\N$.  In other words, $\ga \sim_{\N} \ti{\ga}$ if and only if 
$$
\ti{\ga} = \ga \cdot \eta,
$$
where $\ou(\eta) \in \N_{\PB_n}$ and the source of $\eta$ coincides with the target of $\eta$.

\bigskip

We claim that 
\begin{prop}  
\label{prop:sim-OK}
For every $\N \in \NFI_{\PB_4}(\B_4)$, $\sim_{\N}$ is a compatible equivalence relation 
on $\PaB^{\le 4}$ in the sense of Definition \ref{dfn:compat-equiv}. Moreover, the assignment 
$$
\N ~~\mapsto~~ \sim_{\N}
$$
upgrades to a functor from the poset $\NFI_{\PB_4}(\B_4)$ to the poset of 
compatible equivalence relations on $\PaB^{\le 4}$.
\end{prop}  
\begin{proof}
The first property of $\sim_{\N}$ follows from the fact that
$\N_{\PB_4} : = \N$ (resp. $\N_{\PB_3}$,  $\N_{\PB_2}$) is normal in $\B_4$ (resp. $\B_3$, $\B_2$).

The second property of $\sim_{\N}$ follows from the obvious identity:
$$
\ou(\ga) = \ou (\te(\ga)) \qquad \forall~~\ga \in \PaB(n),~\te \in S_n\,.
$$

The proof of the last property is based on the observation that elementary operadic insertions for $\PaB$
can be expressed in terms of the operations $? \circ_i \id_{\tau}$, $\id_{\tau} \circ_i ?$, 
and the composition of morphisms in the groupoids $\PaB(3)$ and $\PaB(4)$. 

Let $n \in \{2,3\}$ and $\eta$ be a morphism in $\PaB(n)$ whose target coincides with its source. 
In particular, $\ou (\eta) \in \PB_n$. Let us prove that, if $\ou (\eta) \in \N_{\PB_n}$,
then, for every $\tau \in \Ob(\PaB(m))$ with $n+m-1 \le 4$, we have 
\begin{equation}
\label{eta-circ-id}
\ou (\eta \circ_i  \id_{\tau}) \in \PB_{n+m-1}, \qquad \forall~~ 1 \le i \le n 
\end{equation}
and 
\begin{equation}
\label{id-circ-eta}
\ou (\id_{\tau} \circ_i \eta) \in \PB_{n+m-1}, \qquad \forall~~ 1 \le i \le m. 
\end{equation} 

Let $h = \ou (\eta)$. If $m = 2$ then there exists $1 \le j \le n$ 
(resp. $j \in \{1,2\}$) such that $\ou(\mm(h) \circ_j \id_{12}) = \ou (\eta \circ_i  \id_{\tau})$
(resp. $\ou(\id_{12}) \circ_j \mm(h) = \ou (\id_{\tau} \circ_i \eta)$).

Thus, if $m=2$, \eqref{eta-circ-id} and \eqref{id-circ-eta} follow 
directly from the definitions of $\N_{\PB_3}$, $\N_{\PB_2}$ \eqref{N-PB-3}, \eqref{N-PB-2} 
and the definitions of the homomorphisms $\vf_{123}, \dots$, $\vf_{12}, \dots$
(see \eqref{vfs-PB} and \eqref{vfs-PB-2-3}).
 
If $m=3$, then there exist $j,k \in \{1,2\}$ such that
$$
\ou (\eta \circ_i  \id_{\tau}) = \ou \big( (\eta \circ_j \id_{12}) \circ_k \id_{12} \big). 
$$
For example, if $\eta \in \Hom_{\PaB}((2,1),(2,1))$, then 
$\ou (\eta \circ_2  \id_{2(1,3)}) = \ou \big( (\eta \circ_2 \id_{12}) \circ_3 \id_{12} \big)$. 

Thus \eqref{eta-circ-id} for $m=3$ follows from \eqref{eta-circ-id} for $m=2$. 
Similarly,  \eqref{id-circ-eta} for $m=3$ follows from \eqref{id-circ-eta} for $m=2$. 

We will now use \eqref{eta-circ-id} and  \eqref{id-circ-eta} to prove the last 
property of $\sim_{\N}$. 

Consider $\gamma_{1}, \tilde{\gamma}_{1} \in \text{Hom}_{\PaB(n)}(\tau_{1}, \tau_{2})$
and $\gamma_{2}, \tilde{\gamma}_{2} \in \text{Hom}_{\PaB(m)}(\omega_{1}, \omega_{2})$. 
First suppose $\gamma_{1} \sim_{\N} \tilde{\gamma}_{1}$, so 
$\tilde{\gamma}_{1} = \gamma_{1} \cdot \eta$ for some $\eta \in \Hom_{\PaB(n)}(\tau_1, \tau_1)$ 
such that $\ou(\eta) \in \N_{\PB_{n}}$. It follows that
\begin{align*}
	\tilde{\gamma}_{1} \circ_{i} \gamma_{2} &= (\gamma_{1} \cdot \eta) \circ_{i} (\gamma_{2} \cdot \id_{\omega_{1}})\\
	&= (\gamma_{1} \circ_{i} \gamma_{2}) \cdot (\eta \circ_{i} \id_{\omega_{1}})
\end{align*}
Due to  \eqref{eta-circ-id}, $\ou(\eta \circ_{i} \id_{\omega_{1}}) \in \N_{\PB_{n+m-1}}$ and 
hence $\tilde{\gamma}_{1} \circ_{i} \gamma_{2} \sim \gamma_{1} \circ_{i} \gamma_{2}$.
 
Now suppose $\gamma_{2} \sim_{\N} \tilde{\gamma}_{2}$, so $\tilde{\gamma}_{2} = \gamma_{2} \cdot \eta'$ 
for some $\eta' \in  \Hom_{\PaB(m)}(\om_1, \om_1)$ such that $\ou(\eta') \in \N_{\PB_{m}}$. It follows that
\begin{align*}
	\gamma_{1} \circ_{i} \tilde{\gamma}_{2} &= (\gamma_{1} \cdot \id_{\tau_{1}}) \circ_{i} (\gamma_{2} \cdot \eta')\\
	&= (\gamma_{1} \circ_{i} \gamma_{2}) \cdot (\id_{\tau_{1}} \circ_{i} \eta')
\end{align*}
Due to \eqref{id-circ-eta}, $\ou(\id_{\tau_{1}} \circ_{i} \eta') \in \N_{\PB_{n+m-1}}$ and 
hence $\gamma_{1} \circ_{i} \tilde{\gamma}_{2} \sim \gamma_{1} \circ_{i} \gamma_{2}$.

This completes the proof of the fact that $\sim_{\N}$ is indeed a compatible equivalence relation on $\PaB^{\le 4}$. 

It is clear that, if $\ti{\N}, \N \in \NFI_{\PB_4}(\B_4)$ and 
$\ti{\N} \le \N$, then $\ti{\N}_{\PB_3} \le \N_{\PB_3}$ and  $\ti{\N}_{\PB_2} \le \N_{\PB_2}$. 

Thus the assignment  $\N ~\mapsto~ \sim_{\N}$ upgrades to a functor from 
the poset $\NFI_{\PB_4}(\B_4)$ to the poset of compatible equivalence relations. 
\end{proof}

\bigskip 

Later, we will need the following technical statement about $\NFI_{\PB_4}(\B_4)$:
\begin{prop}  
\label{prop:subgroups-in-PB3-2}
~~~
\begin{itemize}
\item[{\bf A)}] For every $\N \in \NFI(\PB_3)$, there exists $\K \in  \NFI_{\PB_4}(\B_4)$ 
satisfying the property 
$$
\K_{\PB_3} \le \N.
$$

\item[{\bf B)}] For every $\N \in \NFI(\PB_2)$ there exists 
$\K \in \NFI_{\PB_4}(\B_4)$ such that $\K_{\PB_2} \le \N$. 
 
\end{itemize}

\end{prop}  
\begin{proof}
Stronger versions of these statements are proved in Subsection \ref{sec:settled-isolated}
(see Proposition \ref{prop:cofinal}).
So we omit the proof of this proposition. 
\end{proof}

\subsection{The set of $\GT$-pairs $\GT_{\pr}(\N)$}
\label{sec:GT-pr-N}

Let $\N \in \NFI_{\PB_4}(\B_4)$ and $\sim_{\N}$ be the corresponding compatible equivalence relation 
on $\PaB^{\le 4}$. Let $\N_{\PB_3}$ (resp. $\N_{\PB_2}$) be the corresponding normal 
subgroup of $\PB_3$ (resp. $\PB_2$) and $N_{\ord}$ be the index of  $\N_{\PB_2}$ in $\PB_2$. 

Since the groupoid $\PaB(0)$ is empty,  Theorem \ref{thm:PaB-gener} implies that
the truncated operad $\PaB^{\le 4}$ is generated by morphisms 
$\al$ and $\beta$ shown in figure
\ref{fig:beta-alpha-here}. 
\begin{figure}[htp] 
\centering 
\begin{tikzpicture}[scale=1.5, > = stealth]
\tikzstyle{v} = [circle, draw, fill, minimum size=0, inner sep=1]
\draw (-0.7,0.5) node[anchor=center] {{$\beta ~ : = $}};
\node[v] (v1) at (0, 0) {};
\draw (0,-0.2) node[anchor=center] {{\small $1$}};
\draw (0,1.2) node[anchor=center] {{\small $2$}};
\node[v] (v2) at (1, 0) {};
\draw (1,-0.2) node[anchor=center] {{\small $2$}};
\draw (1,1.2) node[anchor=center] {{\small $1$}};
\node[v] (vv2) at (0, 1) {};
\node[v] (vv1) at (1, 1) {};
\draw [->] (v1) -- (vv1);
\draw (v2) -- (0.6, 0.4); 
\draw [->] (0.4, 0.6) -- (vv2);
\begin{scope}[shift={(3.5,0)}]
\draw (-0.7,0.5) node[anchor=center] {{$\al ~ : = $}};
\node[v] (v1) at (0, 0) {};
\draw (-0.15,-0.2) node[anchor=center] {{\small $($}};
\draw (0,-0.2) node[anchor=center] {{\small $1$}};
\node[v] (v2) at (0.5, 0) {};
\draw (0.5,-0.2) node[anchor=center] {{\small $2$}};
\draw (0.65,-0.2) node[anchor=center] {{\small $)$}};
\node[v] (v3) at (1.5, 0) {};
\draw (1.5,-0.2) node[anchor=center] {{\small $3$}};
\node[v] (vv1) at (0, 1) {};
\draw (0,1.2) node[anchor=center] {{\small $1$}};
\node[v] (vv2) at (1, 1) {};
\draw (0.85,1.2) node[anchor=center] {{\small $($}};
\draw (1,1.2) node[anchor=center] {{\small $2$}};
\node[v] (vv3) at (1.5, 1) {};
\draw (1.5,1.2) node[anchor=center] {{\small $3$}};
\draw (1.65,1.2) node[anchor=center] {{\small $)$}};
\draw [->] (v1) -- (vv1);  \draw [->] (v2) -- (vv2);  \draw [->] (v3) -- (vv3); 
\end{scope}
\end{tikzpicture}
\caption{The isomorphisms $\al$ and $\beta$} \label{fig:beta-alpha-here}
\end{figure}
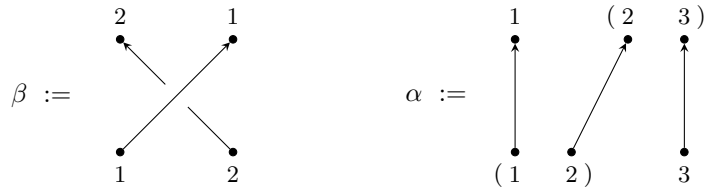

Moreover any relation on $\al$ and $\beta$ in $\PaB^{\le 4}$ is a consequence of 
the pentagon relation
\begin{equation}
\label{pentagon-here}
\begin{tikzpicture}
\matrix (m) [matrix of math nodes, row sep=1.5em, column sep=1.5em]
{~&  (1(23))4  &  ~ & 1((23)4) &  ~ \\
((12)3)4 & ~ & ~ & ~  & 1(2(34)) \\ 
 ~ & ~~ & (12)(34)  & ~ &  ~ \\};
\path[->, font=\scriptsize]
(m-2-1) edge node[above] {$\id_{12} \circ_1 \al~~~~~~~~$} (m-1-2)
(m-1-2) edge node[above] {$\al \circ_2 \id_{12}$} (m-1-4)
(m-1-4) edge node[above] {$~~~~~~~~\id_{12} \circ_2 \al$} (m-2-5)
(m-2-1) edge node[above] {$~~~~\al \circ_1 \id_{12}$} (m-3-3)
 (m-3-3)  edge node[above] {$\al \circ_3 \id_{12}~~~~$} (m-2-5);
\end{tikzpicture}
\end{equation}
and the hexagon relations
 \begin{equation}
\label{hexagon1-here}
\begin{tikzpicture}
\matrix (m) [matrix of math nodes, row sep=1.8em, column sep=3.2em]
{(12) 3 &  3(12)  &  (31)2  \\
1(23)  & 1(32) &   (13)2  \\ };
\path[->, font=\scriptsize]
(m-1-1) edge node[above] {$\beta \circ_1 \id_{12}$} (m-1-2)  edge node[left] {$\al~$} (m-2-1)
(m-2-1) edge node[above] {$\id_{12} \circ_2 \beta $} (m-2-2)  (m-2-2)  edge node[above] {$ (2,3)\, \al^{-1}$} (m-2-3) 
(m-2-3)  edge node[right] {$(2,3)\, (\id_{12} \circ_1 \beta)$} (m-1-3) (m-1-3)  edge node[above] {$(1,3,2)\, \al$} (m-1-2) ;
\end{tikzpicture}
\end{equation}
\begin{equation}
\label{hexagon11-here}
\begin{tikzpicture}
\matrix (m) [matrix of math nodes, row sep=1.8em, column sep=3.2em]
{1(23) &  (23)1  &  2(31)  \\
(12)3  & (21)3 &   2(13)  \\ };
\path[->, font=\scriptsize]
(m-1-1) edge node[above] {$\beta \circ_2 \id_{12}$} (m-1-2)  (m-1-2) edge node[above] {$(1,2,3)\, \al$}  (m-1-3)
(m-1-1) edge node[left] {$\al^{-1}$} (m-2-1)  (m-2-1) edge node[above] {$\id_{12} \circ_1 \beta$} (m-2-2) 
(m-2-2)  edge node[above] {$(1,2)\, \al $} (m-2-3)  (m-2-3) edge node[right] {$(1,2) \, (\id_{12} \circ_2 \beta)$}  (m-1-3);
\end{tikzpicture}
\end{equation}

Thus morphisms of truncated operads 
\begin{equation}
\label{from-truncated}
T : \PaB^{\le 4} \to \PaB^{\le 4} / \sim_{\N}
\end{equation}
are in bijection with pairs
\begin{equation}
\label{the-pair}
(m + N_{\ord} \bbZ ,f \N_{\PB_3}) \in \bbZ/ N_{\ord} \bbZ \times \PB_3/ \N_{\PB_3}
\end{equation}
satisfying the relations 
\begin{equation}
\label{hexa1}
\si_1 x_{12}^m \, f^{-1} \si_2 x_{23}^m f \, \N_{\PB_3} ~ = ~ 
f^{-1} \si_1 \si_2 (x_{13} x_{23})^m \, \N_{\PB_3}\,,
\end{equation} 
\begin{equation}
\label{hexa11}
f^{-1} \si_2 x_{23}^m f \, \si_1 x_{12}^m \,  \N_{\PB_3}  ~=~ \si_2 \si_1 (x_{12} x_{13})^m \, f \,  \N_{\PB_3}
\end{equation}
in $\B_3 / \N_{\PB_3}$ and  
\begin{equation}
\label{GT-penta}
\vf_{234}(f)\, \vf_{1,23,4} (f)\, \vf_{123}(f) \, \N    ~ = ~  \vf_{1,2,34}(f)   \vf_{12,3,4}(f)\, \N
\end{equation}
in $\PB_4 /\N$.

More precisely, this bijection sends a pair \eqref{the-pair} to the morphism 
of truncated operads $T_{m,f} : \PaB^{\le 4} \to \PaB^{\le 4} / \sim_{\N}$
defined by the formulas:
$$
T_{m,f}(\al): = [ \al \cdot \mm(f) ], \qquad
T_{m,f}(\beta) : =  [\beta \cdot \mm(x_{12}^m)],
$$ 
where $\mm$ is the map from $\B_n$ to $\PaB(n)$ defined in Appendix \ref{app:groupoid-PaB}.

This observation motivates our definition of a $\GT$-pair:
\begin{defi}  
\label{dfn:GT-pair}
For $\N \in \NFI_{\PB_4}(\B_4)$, the set $\GT_{\pr}(\N)$ consists of pairs 
$$
(m + N_{\ord} \bbZ ,f \N_{\PB_3}) ~\in~ \bbZ/ N_{\ord} \bbZ \times \PB_3/ \N_{\PB_3}
$$
satisfying \eqref{hexa1}, \eqref{hexa11} and \eqref{GT-penta}.
Elements of $\GT_{\pr}(\N)$ are called \emph{$\GT$-pairs}.
\end{defi}  
We will represent $\GT$-pairs by tuples $(m,f) \in \bbZ \times \PB_3$.
It is straightforward to see that, if relations \eqref{hexa1}, \eqref{hexa11} and \eqref{GT-penta}
are satisfied for a pair $(m,f)$, then they are also satisfied for $(m+ q N_{\ord}, f h)$, where $q$ 
is an arbitrary integer and $h$ is an arbitrary element in $\N_{\PB_3}$.  
A $\GT$-pair in $\bbZ/ N_{\ord} \bbZ \times \PB_3/ \N_{\PB_3}$ represented by a tuple 
$(m,f) \in \bbZ \times \PB_3$ will be often denoted by 
$$
[(m,f)].
$$

For a tuple $(m,f)$ representing a $\GT$-pair in $\GT_{\pr}(\N)$, we denote by  
$T_{m,f}$ the corresponding morphism of truncated operads:
\begin{equation}
\label{T-m-f}
T_{m,f} : \PaB^{\le 4} ~\to~  \PaB^{\le 4} /\sim_{\N}\,.
\end{equation}

It is clear that the assignment $\ou$ from Appendix \ref{app:groupoid-PaB} 
induces the obvious map 
\begin{equation}
\label{ou-quotient}
\PaB(n) / \sim_{\N} ~\to~ \B_n  /\N_{\PB_n}\,,
\end{equation}
for every $2 \le n \le 4$ and, by abuse of notation, we will use the same symbol 
$\ou$ for the map \eqref{ou-quotient}.

Using this map together with the $\mm: \B_n \to \PaB(n)$ from 
Appendix \ref{app:groupoid-PaB} and morphism \eqref{T-m-f}, we define group homomorphisms 
$\B_2 \to \B_2 / \N_{\PB_2}$,  $\B_3 \to \B_3 / \N_{\PB_3}$,  $\B_4 \to \B_4 / \N$. 
Restricting these homomorphisms to $\PB_2$, $\PB_3$ and $\PB_4$, we get 
group homomorphisms $\PB_2 \to \PB_2 / \N_{\PB_2}$,  $\PB_3 \to \PB_3 / \N_{\PB_3}$,  
$\PB_4 \to \PB_4 / \N$, respectively. More precisely,  
\begin{cor}
\label{cor:for-braid-groups}
For every pair $(m + N_{\ord} \bbZ , ~ f \N_{\PB_3}) \in \GT_{\pr}(\N)$, and every 
$2 \le n \le 4$, the assignment
\begin{equation}
\label{T-B-n}
T^{\B_n}_{m,f}(g) : = \ou \circ T_{m,f} \circ \mm(g)
\end{equation}
defines a group homomorphism from $\B_n \to \B_n / \N_{\PB_n}$. 
The restriction of $T^{\B_n}_{m,f}$ to $\PB_n$ gives us a group homomorphism 
\begin{equation}
\label{T-PB-n}
T^{\PB_n}_{m,f} (g) : \PB_n \to  \PB_n / \N_{\PB_n}\,.
\end{equation}

The action of $T^{\B_4}_{m,f}$ on the generators of $\B_4$ is given by the formulas:
\begin{equation}
\label{for-B4}
T^{\B_4}_{m,f} (\si_1) : = \si_1 x_{12}^m \, \N, \quad
T^{\B_4}_{m,f} (\si_2) : = \vf_{123}(f)^{-1} (\si_2 x_{23}^m) \vf_{123}(f) \, \N, 
\end{equation}
$$
T^{\B_4}_{m,f} (\si_3) : =  \vf_{12,3,4}(f)^{-1} \, (\si_3 x_{34}^m) \, \vf_{12,3,4}(f)\, \N. 
$$
The action of $T^{\B_3}_{m,f}$ on the generators of $\B_3$ are given by the formulas:
\begin{equation}
\label{for-B3}
T^{\B_3}_{m,f} (\si_1) : = \si_1 x_{12}^m \, \N_{\PB_3} \,, \quad
T^{\B_3}_{m,f} (\si_2) : = f^{-1} (\si_2 x_{23}^m) f \, \N_{\PB_3}\,.
\end{equation}
Finally, $T^{\B_2}_{m,f}$ sends $\si_1$ to $\si_1 x_{12}^m \, \N_{\PB_2}$. 
\end{cor}
\begin{proof} It is clear that, for every two composable morphisms 
$\ga_1, \ga_2 \in \PaB(n)/\sim_{\N}$, we have 
\begin{equation}
\label{ou-composition}
\ou(\ga_1 \cdot \ga_2) = \ou(\ga_1) \cdot \ou(\ga_2).
\end{equation}
Then using \eqref{mm-comp} from Appendix \ref{app:groupoid-PaB} and the compatibility 
of $T_{m,f}$ with the structure of the truncated operad we get
\begin{align*}
	T^{\B_n}_{m,f}(g_{1} \cdot g_{2}) &= \ou \Big(T_{m,f}\big(\mm(g_{1} \cdot g_{2})\big)\Big) \\
	&= \ou \Big(T_{m,f} \big( \rho(g_{2})^{-1} \big(\mm(g_{1}) \big) \cdot \mm(g_{2}) \big) \Big) \\
	&= \ou \Big(T_{m,f} \big(\rho(g_{2})^{-1} (\mm(g_{1})) \big) \cdot T_{m,f} \big( \mm(g_{2}) \big) \Big) \\
	&= \ou \Big( \rho(g_{2})^{-1} T_{m,f} (\mm(g_{1})) \cdot T_{m,f} (\mm(g_{2})) \Big) \\
	&= \ou \Big(\rho(g_{2})^{-1} T_{m,f} (\mm(g_{1})) \Big) \cdot \ou \Big( T_{m,f} (\mm(g_{2})) \Big) \\
	&= \ou \Big(  T_{m,f}\big(\mm(g_{1})\big) \Big) \cdot \ou \Big(T_{m,f}\big(\mm(g_{2})\big) \Big) \\
	&= T^{\B_n}_{m,f}(g_{1}) \cdot T^{\B_n}_{m,f}(g_{2})
\end{align*}
whence $T^{\B_n}_{m,f}$ is a group homomorphism.

The second statement of the corollary follows immediately from the fact $\mm$ is a right 
inverse of $\ou$ and $T_{m,f}$ acts trivially on objects of $\PaB$.

We will now prove \eqref{for-B4}. The easier cases of $T_{m,f}^{\B_3}$ and $T_{m,f}^{\B_2}$
are left for the reader.

For the generator $\si_1$, we have:
\begin{align*}
	T_{m,f}^{\B_4}(\si_1) &= \ou \Big(T_{m,f} \big(\mm(\si_1) \big) \Big) \\ 
	&= \ou \Big(T_{m,f} \big(\id_{(12)3} \circ_1 \beta \big) \Big) \\
	&= \ou \Big( \id_{(12)3} \circ_1 [\beta \cdot \mm(x_{12}^{m})] \big) \\
	&= \si_{1} x_{12}^{m}\, \N.
\end{align*}

For the generator $\si_2$, we have:
\begin{align*}
T_{m,f}^{\B_4}(\si_2) &= \ou\Big(T_{m,f}\big(\mm(\si_2)\big)\Big) \\
&= \ou\Big(T_{m,f}\big( (2,3)(\id_{12} \circ_1 \alpha^{-1}) \cdot (\id_{(12)3} \circ_{2} \beta) 
\cdot (\id_{12} \circ_1 \alpha) \big) \Big) \\
&= \ou\Big( (2,3) \big(\id_{12} \circ_1 [\mm(f^{-1}) \cdot \al^{-1}]\big) 
\cdot \big(\id_{(12)3} \circ_{2} [\beta \cdot \mm(x_{12}^m)]\big) \cdot 
	\big(\id_{12} \circ_1 [\alpha \cdot \mm(f)] \big)\Big) \\
	&= \vf_{123}(f)^{-1} (\si_2 x_{23}^m) \vf_{123}(f)\, \N.
\end{align*}

Finally, for the generator $\si_3$, we have:
\begin{align*}
T_{m,f}^{\B_4}(\si_3) &= \ou\Big(T_{m,f}\big(\mm(\si_3)\big)\Big) \\
&= \ou\Big(T_{m,f}\big( (3,4)(\al^{-1} \circ_1 \id_{12}) \cdot (\id_{(12)3} \circ_3 \beta) 
\cdot (\al \circ_1 \id_{12}) \big)\Big) \\
&= \ou\Big( (3,4) \big([ \mm(f)^{-1} \cdot \al^{-1} ] \circ_1 \id_{12}\big) \cdot 
\big(\id_{(12)3} \circ_3 [ \beta\cdot \mm(x_{12}^{m})]\big) \cdot \big([\alpha \cdot \mm(f)] \circ_{1} \id_{12}\big)\Big) \\
&= \vf_{12,3,4}(f)^{-1}(\si_{3}x_{34}^{m})\vf_{12,3,4}(f)\, \N.
\end{align*}
This completes the proof of Corollary \ref{cor:for-braid-groups}.
\end{proof}

Let us now use Corollary \ref{cor:for-braid-groups} to prove the following statement:
\begin{cor}
\label{cor:act-on-PB3}
Let $\N \in \NFI_{\PB_4}(\B_4)$, $[(m,f)] \in \GT_{\pr}(\N)$ and $c$ be the generator of the center 
of $\PB_3$ (see \eqref{c-PB3-B3}). Then 
\begin{equation}
\label{act-on-xy}
T^{\PB_3}_{m,f}(x_{12}) = x_{12}^{2m+1} \, \N_{\PB_3}\,,
\qquad
T^{\PB_3}_{m,f}(x_{23}) = f^{-1} x_{23}^{2m+1} f \, \N_{\PB_3}\,,
\end{equation}
\begin{equation}
\label{act-on-x-c}
T^{\PB_3}_{m,f}(x_{13}) = x_{12}^{-m} \si_1^{-1}\,   f^{-1} x_{23}^{2m+1} f \, \si_1 x_{12}^m \, \N_{\PB_3},  \qquad
T^{\PB_3}_{m,f}(c) = c^{2m+1} \, \N_{\PB_3}\,. 
\end{equation}
\end{cor}
\begin{proof}
The first equation in \eqref{act-on-xy} is a simple consequence of the first 
equation in \eqref{for-B3}. 

Using the second equation in \eqref{for-B3}, we get
$$
T^{\PB_3}_{m,f}(x_{23}) = T^{\PB_3}_{m,f}(\si_2^2) =  \big( f^{-1} (\si_2 x_{23}^m) f \big)^2\, \N_{\PB_3} =
 f^{-1} (\si^2_2 x_{23}^{2m}) f\,  \N_{\PB_3} =  f^{-1} \, x_{23}^{2m+1}\, f \, \N_{\PB_3}\,.
$$
Thus the second equation in \eqref{act-on-xy} is proved. 

Using the definition of $x_{13}:= \si_1^{-1} \si^2_2 \si_1 = \si_1^{-1} x_{23} \si_1$, 
the first equation in \eqref{for-B3} and the second equation in \eqref{act-on-xy}, we get
$$
T^{\PB_3}_{m,f}(x_{13}) = T^{\PB_3}_{m,f}( \si_1^{-1} x_{23} \si_1) = 
x_{12}^{-m} \si_1^{-1}\,   f^{-1} x_{23}^{2m+1} f \, \si_1 x_{12}^m \, \N_{\PB_3}\,.
$$
Thus the first equation in \eqref{act-on-x-c} is also satisfied. 

To prove the second equation in \eqref{act-on-x-c}, 
we use the formulas \eqref{hexa1}, \eqref{hexa11}, \eqref{for-B3}, and 
the identities $x_{1 3} x_{2 3} = x_{12}^{-1} c$,  $x_{12} x_{13} = x_{23}^{-1} c$ extensively.
\begin{align*}
	T_{m,f}^{\PB_3}(c) &= T_{m,f}^{\PB_3}((\si_1\si_2)^3) \\
	&= \big( T_{m,f}^{\B_3}(\si_1\si_2) \big) ^3 \\
	&= \si_1 x_{12}^m f^{-1} \si_2 x_{23}^m f \si_1 x_{12}^m f^{-1} \si_2 x_{23}^m f \si_1 x_{12}^m f^{-1} \si_2 x_{23}^m f \, \N_{\PB_3} \\
	&= f^{-1} \si_1 \si_2 (x_{13}x_{23})^m ~ \si_1 x_{12}^m ~ \si_2 \si_1 (x_{12}x_{13})^m f ~ f^{-1} \si_2 x_{23}^m f \, \N_{\PB_3} \\
	&= f^{-1} \si_1 \si_2 ( x_{12}^{-1} c)^m \si_1 x_{12}^m \si_2 \si_1 (x_{23}^{-1} c)^m \si_2 x_{23}^m f \,  \N_{\PB_3} \\
	&= f^{-1} \si_1 \si_2  x_{12}^{-m} c^m \si_1 x_{12}^m \si_2 \si_1 x_{23}^{-m} c^m \si_2 x_{23}^m f  \, \N_{\PB_3} \\
	&= c^{2m} f^{-1}  (\si_1 \si_2 \si_1 \si_2 \si_1 \si_2)  f \, \N_{\PB_3} \\
	&= c^{2m+1} \, \N_{\PB_3}\,.
\end{align*}
\end{proof}

\bigskip

\subsection{$\GT$-pairs coming from automorphisms of $\wh{\PaB}$}
\label{sec:GT-shadows}

Let $\N \in \NFI_{\PB_4}(\B_4)$ and $\sim_{\N}$ be the corresponding compatible equivalence 
relation. Since the groupoids $\PaB(n)/\sim_{\N}$ (for $1\le n \le 4$) are finite, 
we have a canonical continuous onto morphism of truncated operads 
\begin{equation}
\label{to-quotient-PaB}
\hcP_{\N} : \wh{\PaB}^{\le 4} ~\to~  \PaB^{\le 4} / \sim_{\N}\,.
\end{equation}

Thus, given any continuous automorphism $\hat{T} : \wh{\PaB} \to \wh{\PaB}$, 
we can produce the morphism of truncated operads
$$
T_{\N} : \PaB^{\le 4} ~\to~ \PaB^{\le 4} / \sim_{\N}
$$
by setting 
\begin{equation}
\label{T-N}
T_{\N} := \hcP_{\N} \circ \hat{T} \circ \cI,
\end{equation}
where $\cI$ is the natural embedding of truncated operads
\begin{equation}
\label{to-completion}
\cI :  \PaB^{\le 4} \to  \wh{\PaB}^{\le 4}\,.
\end{equation}

In other words, for every continuous automorphism of $\wh{\PaB}$ and every  $\N \in \NFI_{\PB_4}(\B_4)$,
we get a $\GT$-pair $[(m,f)]$ corresponding to $T_{\N}$. In this situation, we say that the $\GT$-pair 
$[(m,f)]$ {\it comes from} the automorphism $\hat{T}$.

$\GT$-pairs coming from automorphisms of $\wh{\PaB}$ satisfy additional properties. 
Indeed, since $\cI(\PaB^{\le 4})$ is dense in $\wh{\PaB}^{\le 4}$ and the morphism 
$$
\cP_{\N} \circ \hat{T} : \wh{\PaB}^{\le 4} \to  \PaB^{\le 4} / \sim_{\N}
$$
is continuous and onto, the morphism $T_{\N}$ is also onto.

Thus, if a $\GT$-pair $[(m,f)]$ comes from a (continuous) automorphism of $\wh{\PaB}$ 
then the group homomorphisms
\begin{equation}
\label{T-PB-4}
T^{\PB_4}_{m,f} : \PB_4 \to \PB_4/\N,
\end{equation}
\begin{equation}
\label{T-PB-3}
T^{\PB_3}_{m,f} : \PB_3 \to \PB_3/\N_{\PB_3},
\end{equation}
\begin{equation}
\label{T-PB-2}
T^{\PB_2}_{m,f} : \PB_2 \to \PB_2/\N_{\PB_2},
\end{equation}
are onto.

$\GT$-pairs satisfying these properties are called {\it $\GT$-shadows}.
More precisely, 
\begin{defi}  
\label{dfn:GT-shadows}
Let $\N$ be a finite index normal subgroup of $\B_4$ such that $\N \le \PB_4$.
Furthermore, let $\N_{\PB_3}$, $\N_{\PB_2}$ be the corresponding normal subgroups of 
$\B_3$ and $\B_2$, respectively and let $N_{\ord}$ be the index of $\N_{\PB_2}$ in $\PB_2$. 
The set $\GT(\N)$ consists of $\GT$-pairs $[(m,f)] \in \GT_{\pr}(\N)$ for which 
group homomorphisms \eqref{T-PB-4}, \eqref{T-PB-3}, \eqref{T-PB-2} are onto.
Elements of  $\GT(\N)$ are called \emph{$\GT$-shadows}.
\end{defi}  

It is easy to see that homomorphism \eqref{T-PB-2} is onto if and only if
\begin{equation}
\label{friendly}
(2 m+1) ~+~ N_{\ord} \bbZ  ~\textrm{ is a unit in the ring } ~\bbZ/ N_{\ord}\bbZ\,.
\end{equation}
We say that a $\GT$-pair $[(m,f)]$ is {\it friendly} if $m$ satisfies condition \eqref{friendly}. 

Due to the following proposition, only homomorphisms \eqref{T-PB-3} and \eqref{T-PB-2} matter:
\begin{prop}  
\label{prop:onto}
Let $\N \in \NFI_{\PB_4}(\B_4)$ and $[(m,f)] \in \GT_{\pr}(\N)$. 
The following statements are equivalent: 
\begin{enumerate}

\item $[(m,f)]$ is a $\GT$-shadow;

\item group homomorphisms  \eqref{T-PB-3} and  \eqref{T-PB-2} are onto;

\item the map of truncated operads $T_{m,f} : \PaB^{\le 4} \to  \PaB^{\le 4}/\sim_{\N}$ is onto.

\end{enumerate}

\end{prop}
\bigskip
\begin{proof} The implication {\it 1.} $\Rightarrow$  {\it 2.} is obvious. It is also clear that, 
if  $T_{m,f} : \PaB^{\le 4} \to  \PaB^{\le 4}/\sim_{\N}$ is onto then
group homomorphisms \eqref{T-PB-4}, \eqref{T-PB-3}, \eqref{T-PB-2} are onto.
Thus the implication {\it 3.} $\Rightarrow$ {\it 1.} is also obvious.   
 
It remains to prove the implication {\it 2.} $\Rightarrow$ {\it 3.}

Since group homomorphism \eqref{T-PB-2} is onto,
there exists $\ga_2 \in \Hom_{\PaB}(12,12)$ such that 
$$
T_{m,f} (\ga_2) = \big[ \mm( x_{12}^{-m}) \big].
$$
Therefore, 
$$
T_{m,f} (\beta \cdot \ga_2) =  T_{m,f} (\beta) \cdot T_{m,f}(\ga_2) = [\beta]. 
$$

Since homomorphism \eqref{T-PB-2} is onto, there exists 
$\ga_3 \in \Hom_{\PaB}((12)3,(12)3)$ such that 
$$
T_{m,f} (\ga_3) = \big[ \mm( f^{-1} ) \big].
$$
Therefore, 
$$
T_{m,f} (\al \cdot \ga_3) =  T_{m,f} (\al) \cdot T_{m,f}(\ga_3) = 
T_{m,f} (\al) \cdot  \big[ \mm( f^{-1} ) \big] = [\al]. 
$$

Since, as a truncated operad in the category of groupoids, $\PaB^{\le 4}$ is 
generated by $\beta$ and $\al$, the truncated operad  $\PaB^{\le 4} /\sim_{\N}$ is 
generated by the equivalence classes $[\beta] \in \PaB(2) /\sim_{\N}$ and 
$[\al]\in \PaB(3) / \sim_{\N}$. 

Using the fact that $[\beta]$ and $[\al]$ belong to the image of $T_{m,f}$, we conclude 
that the morphism of truncated operads $T_{m,f}$  is indeed onto. 

Since the implication {\it 2.} $\Rightarrow$ {\it 3.} is established, the proposition is proved.
\end{proof}

\subsection{The groupoid $\GTSh$}
\label{sec:GTSh}

Let $\N \in \NFI_{\PB_4}(\B_4)$ and $[(m,f)] \in \GT(\N)$. 
The morphism of truncated operads 
$$
T_{m,f} : \PaB^{\le 4} \to \PaB^{\le 4} /\sim_{\N} 
$$
gives us the obvious compatible equivalence relation $\sim_{\ms}$: 
\begin{equation}
\label{sim-source}
\ga_1 \sim_{\ms} \ga_2  ~\Leftrightarrow~ T_{m,f}(\ga_1) ~ = ~ T_{m,f}(\ga_2). 
\end{equation}

\begin{prop}  
\label{prop:sim-s-OK}
Let $\N \in \NFI_{\PB_4}(\B_4)$, $[(m,f)] \in \GT(\N)$ and 
$$
\N^{\ms} : = \ker(T^{\PB_4}_{m,f}) ~\unlhd ~ \PB_4\,.
$$
Then $\N^{\ms} \in \NFI_{\PB_4}(\B_4)$ and the compatible equivalence 
relation $\sim_{\ms}$ coincides with $\sim_{\N^{\ms}}$. 
\end{prop}  
\begin{proof} To prove the first statement, we observe that, since $\N \in \PB_4$, 
the standard homomorphism $\rho : B_4 \to S_4$ induces a group homomorphism 
$\ti{\rho}: \B_4/\N \to S_4$.
Furthermore, using equations \eqref{for-B4}, it is easy to see that the composition  
$$
\psi : =  \ti{\rho} \circ T^{\B_4}_{m,f} : \B_4 \to S_4
$$
coincides with $\rho$. Thus $\N^{\ms}$ is the kernel of a group homomorphism 
$T^{\B_4}_{m,f}$ from $\B_4$ to a finite group $\B_4/\N$. Hence $\N^{\ms}$ is a 
finite index normal subgroup of $\B_4$. Since we also have $\N^{\ms} \le \PB_4$, 
we conclude that $\N^{\ms} \in \NFI_{\PB_4}(\B_4)$.  

Although the proof of the second statement is rather technical, the main idea is to 
show that group homomorphisms $T^{\PB_n}_{m,f}$ (for $n =2,3,4$) are, in some sense, compatible with 
the homomorphisms $\vf_{123}, \vf_{12,3,4}, \vf_{1,23,4}, \vf_{1,2,34}, 
\vf_{234}, \vf_{12}, \vf_{12,3}, \vf_{1,23}, \vf_{23}$ (see equations \eqref{vfs-3-4-on-gen} 
and \eqref{vfs-2-3-on-gen}). This fact is deduced from the compatibility of $T_{m,f}$ 
with the structures of truncated operads. 
Then the desired second statement of Proposition \ref{prop:sim-s-OK}
is a simple consequence of this compatibility property
of homomorphisms $T^{\PB_n}_{m,f}$ (for $n =2,3,4$). 

Let us consider $h \in \PB_n$ (for $n \in \{2,3\}$) and denote by $\ti{h}$
any representative of the coset $T^{\PB_n}_{m,f}(h)$ in $\PB_n/\N_{\PB_n}$. 
Our first goal is to prove that, for every
$$
\vf \in 
\begin{cases}
\{ \vf_{123}, \vf_{12,3,4}, \vf_{1,23,4}, \vf_{1,2,34}, \vf_{234} \}    \qquad \textrm{if} ~~ n=3,     \\
\{\vf_{12}, \vf_{12,3}, \vf_{1,23}, \vf_{23}\}    \qquad \qquad ~~ \textrm{if} ~~ n=2,
\end{cases} 
$$
there exists $g \in \PB_{n+1}/ \N_{\PB_{n+1}} $ such that 
\begin{equation}
\label{T-PB-vf}
g^{-1} T^{\PB_{n+1}}_{m,f} (\vf(h)) g = \vf (\ti{h})\, \N_{\PB_{n+1}}\,.
\end{equation}

Indeed, let $n=3$ and $\vf =  \vf_{1,23,4}$. 
Setting $\eta : = \mm(h)$ and using the compatibility of $T_{m,f}$ with 
operadic insertions and compositions we get 
\begin{equation}
\label{T-circ-2}
T_{m,f} (\eta) \circ_2 \id_{12} = T_{m,f} (\eta  \circ_2 \id_{12} ).
\end{equation}

Applying $\ou$ to the left hand side of \eqref{T-circ-2}, we get 
\begin{equation}
\label{to-LHS}
\ou \big(T_{m,f} (\eta) \circ_2 \id_{12} \big)  = \vf_{1,23,4} (\ti{h})\, \N_{\PB_4}\,,
\end{equation}
where $\ti{h}$ is an element of the coset $T^{\PB_3}_{m,f}(h)$ in $\PB_3/\N_{\PB_3}$.  

As for the right hand side of \eqref{T-circ-2}, we have 
$$
T_{m,f} (\eta  \circ_2 \id_{12} ) = 
T_{m,f} \big(\al_{((1,2)3)4}^{(1(2,3))4}  \cdot \mm(\vf_{1,23,4}(h)) \cdot \al_{(1(2,3))4}^{((1,2)3)4}  \big) =
$$
$$
T_{m,f}(\al_{((1,2)3)4}^{(1(2,3))4} ) \cdot T_{m,f}(\mm(\vf_{1,23,4}(h)))
\cdot T_{m,f} (\al_{(1(2,3))4}^{((1,2)3)4}).
$$

Thus 
\begin{equation}
\label{to-RHS}
\ou\big(T_{m,f} (\eta  \circ_2 \id_{12} )  \big) = g^{-1} T^{\PB_4}_{m,f} (\vf_{1,23,4}(h)) g,
\end{equation}
where $g = \ou \big( T_{m,f} (\al_{(1(2,3))4}^{((1,2)3)4}) \big)$.

Combining \eqref{to-LHS} with \eqref{to-RHS}, we conclude that \eqref{T-PB-vf}
holds for $n= 3$ and $\vf = \vf_{1,23,4}$. 

Let us now consider the case when $n=2$ and $\vf =  \vf_{12}$.
 
As above, setting $\eta : = \mm(h)$ and using the compatibility of $T_{m,f}$ with 
operadic insertions and compositions we get 
\begin{equation}
\label{T-circ-more}
\id_{12} \circ_1 T_{m,f} (\eta) = T_{m,f} (\id_{12}  \circ_1 \eta ).
\end{equation}

Applying $\ou$ to the left hand side of \eqref{T-circ-more}, we get 
\begin{equation}
\label{to-LHS-more}
\ou \big( \id_{12} \circ_1 T_{m,f} (\eta) \big) = \vf_{12}(\ti{h})\, \N_{\PB_3}\,,
\end{equation}
where $\ti{h}$ is an element of the coset $T^{\PB_2}_{m,f}(h)$ in $\PB_2/\N_{\PB_2}$.  

The right hand side of \eqref{T-circ-more} can be rewritten as follows: 
$$
T_{m,f} (\id_{12}  \circ_1 \eta ) = T_{m,f} \big( \mm( \vf_{12}(h)) \big).
$$
Hence 
\begin{equation}
\label{to-RHS-more}
\ou \big( T_{m,f} (\id_{12}  \circ_1 \eta )  \big) =
\ou \big( T_{m,f} \big( \mm( \vf_{12}(h)) \big) \big) = T^{\PB_3}_{m,f}( \vf_{12}(h)).
\end{equation}

Combining \eqref{to-LHS-more} with \eqref{to-RHS-more}, we conclude that 
\eqref{T-PB-vf} holds for $n= 2$ and $\vf = \vf_{12}$ with $g = 1_{\PB_3/\N_{\PB_3}}$.  

The proof of \eqref{T-PB-vf} for the remaining case proceeds in the similar way. 

Let us now prove that, for every $n \in \{2,3,4\}$, 
\begin{equation}
\label{mm-h-OK}
h \in \N^{\ms}_{\PB_n} ~\Rightarrow~ 
\mm(h) \sim_{\ms} \id_{((1,2)..}\,,
\end{equation}
where $((1,2)..$ denotes $12$ (resp. $(1,2)3$, $((1,2)3)4$) if $n =2$ (resp. $n=3$, $n=4$). 

For $n=4$, \eqref{mm-h-OK} is a straightforward consequence of the definition of $\N^{\ms}$. 
So let $n =3$ and $\ti{h}$ be an element of the coset $T^{\PB_3}_{m,f}(h)$ in $\PB_3/\N_{\PB_3}$

Since  $T^{\PB_4}_{m,f}(\vf(h)) = 1$ in $\PB_4/\N$ for every 
$\vf \in \{ \vf_{123}, \vf_{12,3,4}, \vf_{1,23,4}, \vf_{1,2,34}, \vf_{234} \}$, 
equation \eqref{T-PB-vf} implies that 
$$
\ti{h} ~\in~  \vf_{123}^{-1} (\N) \cap  \vf_{12,3,4}^{-1}(\N) \cap
\vf_{1,23,4}^{-1}(\N) \cap \vf_{1,2,34}^{-1} (\N) 
\cap  \vf_{234}^{-1} (\N).  
$$
In other words, $\ti{h} \in \N_{\PB_3}$ and hence $T^{\PB_3}_{m,f}(h) = 1$ in  $\PB_3/\N_{\PB_3}$
Thus \eqref{mm-h-OK} holds for $n=3$. 

Let us now consider the case $n =2$ and denote by
$\ti{h}$ an element of the coset $T^{\PB_2}_{m,f}(h)$ in $\PB_2/\N_{\PB_2}$. 

Since $\vf(h) \in \N^{\ms}_{\PB_3}$ for every 
$\vf \in \{\vf_{12}, \vf_{12,3}, \vf_{1,23}, \vf_{23}\}$ and 
implication \eqref{mm-h-OK} is proved for $n=3$, we conclude that 
$$
T^{\PB_3}_{m,f}(\vf(h)) = 1 \qquad \forall~ \vf \in \{\vf_{12}, \vf_{12,3}, \vf_{1,23}, \vf_{23}\}.
$$
Therefore, equation \eqref{T-PB-vf} implies that 
$$
\ti{h} ~\in~  \vf_{12}^{-1} (\N_{\PB_3}) \cap  
 \vf_{12,3}^{-1} (\N_{\PB_3}) \cap  \vf_{1,23}^{-1} (\N_{\PB_3})
\cap  \vf_{23}^{-1} (\N_{\PB_3}).  
$$
In other words, $\ti{h} \in \N_{\PB_2}$ and hence $T^{\PB_2}_{m,f}(h) =1$ 
in  $\PB_2/\N_{\PB_2}$. Thus implication \eqref{mm-h-OK} holds for $n=2$ as well.

Let us now prove that, for every $n \in \{2,3,4\}$ and $h \in \PB_n$
\begin{equation}
\label{if-mm-h-then}
\mm(h) \sim_{\ms} \id_{((1,2)..}  ~\Rightarrow~ h \in \N^{\ms}_{\PB_n}\,.
\end{equation}

Again, for $n = 4$, \eqref{if-mm-h-then} is a straightforward consequence of 
the definition of $\N^{\ms}$. So let $h \in \PB_3$. 
 
Since $\mm(h) \sim_{\ms} \id_{(1,2)3}$, $T^{\PB_3}_{m,f}(h)$ is the identity element 
of $\PB_3/ \N_{\PB_3}$.  Hence, equation \eqref{T-PB-vf} implies that 
$\vf(h) \in \N^{\ms}$ for every $\vf \in \{ \vf_{123}, \vf_{12,3,4}, \vf_{1,23,4}, \vf_{1,2,34}, \vf_{234} \} $ or 
equivalently $h \in  \N^{\ms}_{\PB_3}$. 

Similarly, if $h \in \PB_2$ and $\mm(h) \sim_{\ms} \id_{1 2}$ then $T^{\PB_2}_{m,f}(h)$ is the identity element 
of $\PB_2/ \N_{\PB_2}$. Hence, equation \eqref{T-PB-vf} implies that 
$T^{\PB_3}_{m,f}(\vf(h)) =1$ in $\PB_3/\N_{\PB_3}$ for every 
$$
\vf \in \{\vf_{12}, \vf_{12,3}, \vf_{1,23}, \vf_{23}\},
$$
or equivalently 
$$
\mm(\vf(h)) \sim_{\ms} \id_{((1,2)3} \qquad \forall~~\vf \in \{\vf_{12}, \vf_{12,3}, \vf_{1,23}, \vf_{23}\}.
$$

Since implication \eqref{if-mm-h-then} is already proved for $n =3$, we 
conclude that 
$$
\vf(h) \in \N^{\ms}_{\PB_3} \qquad \forall~~\vf \in \{\vf_{12}, \vf_{12,3}, \vf_{1,23}, \vf_{23}\}.
$$
Thus $h \in  \N^{\ms}_{\PB_2}$ and \eqref{if-mm-h-then} is proved for $n =2$. 

Let $n \in \{2,3,4\}$, $\tau \in \Ob(\PaB(n))$, 
$\eta \in \Aut_{\PaB}(\tau)$ and $h : = \ou (\eta) \in \PB_n$. 
Our next goal is to prove that 
\begin{equation}
\label{the-key-sim-ms}
h \in \N^{\ms}_{\PB_n} ~~\Leftrightarrow~~ \eta \sim_{\ms} \id_{\tau}\,.
\end{equation}

Since $T_{m,f}$ is compatible with the action of the symmetric groups, 
we may assume, without loss of generality, that the underlying permutation of 
$\tau$ is the identity permutation in $S_n$. 

Therefore
\begin{equation}
\label{eta-h}
\eta =  \al_{((1,2)..}^{\tau} \mm(h) \al^{((1,2)..}_{\tau} 
\end{equation}
and hence $T_{m,f}(\eta)= \id_{\tau}$ if and only if $T_{m,f}(\mm(h)) = \id_{((1,2)..}$ 
and the latter is equivalent to $\mm(h) \sim_{\ms}  \id_{((1,2)..}$.

Thus \eqref{the-key-sim-ms} is a consequence of implications \eqref{mm-h-OK} and 
\eqref{if-mm-h-then}.

Finally, let us use \eqref{the-key-sim-ms} to prove the statement of the proposition. 

Let $\ga, \ti{\ga} \in \PaB(n)$ (with $n \in \{2,3,4\}$) and $\tau$ be the source of 
both morphisms.  Clearly, $\ga \sim_{\ms} \ti{\ga}$ if and only 
if $\eta \sim_{\ms} \id_{\tau}$, where $\eta = \ga^{-1} \cdot \ti{\ga}$.

Thus, due to \eqref{the-key-sim-ms},  $\ga \sim_{\ms} \ti{\ga}$ if and 
only if $\ou ( \ga^{-1} \cdot \ti{\ga}) \in \N^{\ms}_{\PB_n}$. 

Proposition \ref{prop:sim-s-OK} is proved. 
\end{proof}

\bigskip

Proposition \ref{prop:sim-s-OK} has the following important consequences: 
\begin{cor}  
\label{cor:source}
For every $\GT$-shadow $[(m,f)] \in \GT(\N)$

\begin{itemize}


\item $|\PB_4 : \N^{\ms}| = |\PB_4: \N|$,

\item  $|\PB_3 : \N^{\ms}_{\PB_3}| = |\PB_3: \N_{\PB_3}|$, and

\item $N^{\ms}_{\ord} = N_{\ord}$ or equivalently $\N^{\ms}_{\PB_2} = \N_{\PB_2}$.

\end{itemize}
\qed
\end{cor}  
\begin{cor}  
\label{cor:isom-quotients}
For every $\GT$-shadow $[(m,f)] \in \GT(\N)$ 
the morphism of truncated operads $T_{m,f} :  \PaB^{\le 4}  \to  \PaB^{\le 4} /\sim_{\N} $
factors as follows
$$
\begin{tikzpicture}
\matrix (m) [matrix of math nodes, row sep=1.8em, column sep=2.6em]
{\PaB^{\le 4}  &   ~ \\ 
\PaB^{\le 4} /\sim_{\N^{\ms}} & \PaB^{\le 4} /\sim_{\N}\,, 
\\ };
\path[->, font=\scriptsize]
(m-1-1) edge node[left] {$\cP_{\N^{\ms}}$} (m-2-1) edge node[above] {$~~T_{m,f}$} (m-2-2)
(m-2-1) edge node[above] {$T_{m,f}^{\isom}$} (m-2-2); 
\end{tikzpicture}
$$
where $\cP_{\N^{\ms}}$ is the canonical projection and $T_{m,f}^{\isom}$ is 
an isomorphism of truncated operads. 

The assignment $[(m,f)] \mapsto T_{m,f}^{\isom}$ gives us a bijection from the set 
$$
\{ [(m,f)] \in \GT(\N) ~|~ \N^{\ms} = \ker(T^{\PB_4}_{m,f}) \}
$$
to the set  $\Isom(\PaB^{\le 4} /\sim_{\N^{\ms}},\PaB^{\le 4} /\sim_{\N} )$ of isomorphisms 
of truncated operads (in the category of groupoids). 
\end{cor}  
\begin{proof} Due to Proposition \ref{prop:onto} and the definition of the equivalence relation $\sim_{\ms}$, 
we have the following commutative diagram of morphisms of truncated operads:
$$
\begin{tikzpicture}
\matrix (m) [matrix of math nodes, row sep=1.8em, column sep=2.6em]
{\PaB^{\le 4}  &   ~ \\ 
\PaB^{\le 4} /\sim_{\ms} & \PaB^{\le 4} /\sim_{\N}\,, 
\\ };
\path[->, font=\scriptsize]
(m-1-1) edge (m-2-1) edge node[above] {$~~T_{m,f}$} (m-2-2)
(m-2-1) edge node[above] {$T_{m,f}^{\isom}$} (m-2-2); 
\end{tikzpicture}
$$
with $T_{m,f}^{\isom}$ being a bijection\footnote{We tacitly assume that 
$T_{m,f}^{\isom}$ acts as the identity on the level of objects.} 
on the level of morphisms. 

Thanks to Proposition \ref{prop:sim-s-OK}, the equivalence relation $\sim_{\ms}$ coincides
with $\sim_{\N^{\ms}}$. Hence $T_{m,f}^{\isom}$ is a morphism of truncated operads
\begin{equation}
\label{T-isom}
T_{m,f}^{\isom} :   \PaB^{\le 4} /\sim_{\N^{\ms}}  \stackrel{\cong}{\,\longrightarrow\,} \PaB^{\le 4} /\sim_{\N}\,.
\end{equation}

Let us denote by $S_{m,f}^\isom: \PaB^{\le 4} /\sim_{\N} \to  \PaB^{\le 4} /\sim_{\N^{\ms}}$ 
the inverse of $T_{m,f}^{\isom}$ (viewed as a map of morphisms) and show 
that $S_{m,f}^\isom$ is compatible with the composition of morphisms and with 
the operadic insertions.

As for the compatibility with operadic insertions, we have 
$$
S_{m,f}^\isom([\ga_1]\circ_i[\ga_2]) = S_{m,f}^\isom(T_{m,f}^\isom([\ti{\ga}_1])\circ_iT_{m,f}^\isom([\ti{\ga}_2])) 
$$
$$
= S_{m,f}^\isom(T_{m,f}^\isom([\ti{\ga}_1]\circ_i[\ti{\ga}_2])) = [\ti{\ga}_1]\circ_i[\ti{\ga}_2] = S_{m,f}^\isom([\ga_1]) 
\circ_i S_{m,f}^\isom([\ga_2]), 
$$
for any $[\ga_1] \in \PaB(n) /\sim_{\N}, ~ [\ga_2] \in  \PaB(k) /\sim_{\N} $ and  $k \le 4$, $1 \le i \le n \le 4$.

The compatibility of $S_{m,f}^\isom$ with the composition of morphisms is proved in a similar fashion. 

Let us now consider an isomorphism of truncated operads 
$$
T^{\isom} : \PaB^{\le 4} /\sim_{\N^{\ms}}  \stackrel{\cong}{\,\longrightarrow\,} \PaB^{\le 4} /\sim_{\N}\,.
$$

Pre-composing $T^{\isom}$
with the canonical projection $\cP_{\N^{\ms}} :  \PaB^{\le 4} \to \PaB^{\le 4} /\sim_{\N^{\ms}} $
we get an onto morphism $T : = T^{\isom} \circ \cP_{\N^{\ms}}$ of truncated operads.
Since $T$ is uniquely determined by a $\GT$-shadow $[(m,f)]  \in \GT(\N)$ such 
that $\ker(T^{\PB_4}_{m,f}) =  \N^{\ms}$, we conclude that the assignment 
$[(m,f)] \mapsto T^{\isom}_{m,f}$ is indeed a bijection 
\begin{equation}
\label{mf-to-isom}
\{ [(m,f)] \in \GT(\N) ~|~ \N^{\ms} = \ker(T^{\PB_4}_{m,f}) \} 
\stackrel{\cong}{~\longrightarrow~} 
\Isom(\PaB^{\le 4} /\sim_{\N^{\ms}},\PaB^{\le 4} /\sim_{\N} ).
\end{equation}

Corollary \ref{cor:isom-quotients} is proved.
\end{proof}
\bigskip

Let us now observe that the assignment
\begin{equation}
\label{isom-truncated}
\Hom(\ti{\N}, \N) : = \Isom(\PaB^{\le 4} /\sim_{\ti{\N}},\PaB^{\le 4} /\sim_{\N}  ) 
\end{equation}
upgrades the set $\NFI_{\PB_4}(\B_4)$ to a groupoid. The set of objects 
of this groupoid is $\NFI_{\PB_4}(\B_4)$ and the set of morphisms 
from $\ti{\N}$ to $\N$ is the set  $\Isom(\PaB^{\le 4} /\sim_{\ti{\N}},\PaB^{\le 4} /\sim_{\N} )$
of isomorphisms of truncated operads (in the category of groupoids). 
Morphisms of this groupoid are composed in the standard way. 

The second statement of Corollary \ref{cor:isom-quotients} 
allows us to tacitly identify \eqref{isom-truncated} with the set 
$$
\{ [(m,f)] \in \GT(\N) ~|~ \ker(T^{\PB_4}_{m,f}) = \ti{\N} \}. 
$$

We will use the identification in the remainder of this paper and we denote 
by $\GTSh$ the resulting \emph{groupoid of $\GT$-shadows.}

The following proposition gives us an explicit formula for the composition of morphisms in $\GTSh$: 
\begin{prop}
\label{prop:GTSh}
Let $\N^{(1)}$, $\N^{(2)}$ and $\N^{(3)}$ be elements of $\NFI_{\PB_4}(\B_4)$ and 
$$
[(m_1,f_1)] \in \Hom_{\GTSh}(\N^{(1)}, \N^{(2)}), \qquad
[(m_2,f_2)] \in \Hom_{\GTSh}(\N^{(2)}, \N^{(3)}). 
$$
Then their composition $[(m_2,f_2)] \circ  [(m_1,f_1)]$ is represented by the pair $(m,f)$ where 
\begin{equation}
\label{sh-composition}
m : = 2 m_1 m_2 + m_1 + m_2\,, \qquad
f  \N^{(3)}_{\PB_3} : = f_2 \N^{(3)}_{\PB_3} \cdot T^{\PB_3}_{m_2, f_2} (f_1). 
\end{equation}
\end{prop}
\bigskip
\begin{proof} 
Let $[(m_2,f_2)]\in \GT(\N^{(3)})$ and $[(m_1,f_1)]\in \GT(\N^{(2)})$, where 
$$
\N^{(2)} : = \ker(T^{\PB_4}_{m_2,f_2}) \qquad \textrm{and} \qquad \N^{(1)} : = \ker(T^{\PB_4}_{m_1,f_1}).
$$
In other words, the $\GT$-shadow $[(m_1,f_1)]$ (resp.  $[(m_2,f_2)]$)
is a morphism from $\N^{(1)}$ to $\N^{(2)}$ (resp. a morphism from  $\N^{(2)}$ to $\N^{(3)}$) in $\GTSh$. 

By Corollary \ref{cor:isom-quotients}, we have the following diagram of morphisms of truncated operads
\begin{equation}
\begin{tikzpicture}
\matrix (m) [matrix of math nodes, row sep=1.8em, column sep=2.1em]
{\PaB^{\le 4}  & &   \\
	~ &  &  \\
	\PaB^{\le 4}  / \sim_{\N^{(1)}}   & ~ & \PaB^{\le 4} / \sim_{\N^{(2)}} & ~ & \PaB^{\le 4} / \sim_{\N^{(3)}}   \,, \\ };
\path[->, font=\scriptsize]
(m-1-1) edge node[left] {$\cP_{\N^{(1)}}$} (m-3-1) edge node[above] {$T_{m,f}$} (m-3-5)
(m-3-1) edge node[above] {$T^{\isom}_{m_1,f_1}$} (m-3-3) 
(m-3-3) edge node[above] {$T^{\isom}_{m_2,f_2}$} (m-3-5); 
\end{tikzpicture}
\end{equation}
where the vertical arrow is the canonical projection.

Formula \eqref{sh-composition} is obtained by looking at the image of 
the associator $[\alpha] \in \PaB^{\leq 4}/\sim_{\N^{(1)}}$ 
(resp. the braiding $[\beta] \in \PaB^{\leq 4}/\sim_{\N^{(1)}}$) under 
$T^{\isom}_{m_2,f_2}\circ T^{\isom}_{m_1,f_1}$. For $[\alpha]$, we have  
$$
T_{m,f}(\alpha) = T^{\isom}_{m_2,f_2}(T^{\isom}_{m_1,f_1}[\alpha]) =  T^{\isom}_{m_2,f_2}(T_{m_1,f_1}(\alpha))  
$$
$$
 =T^{\isom}_{m_2,f_2}([\alpha \cdot \mm(f_1)])=T_{m_2,f_2}(\alpha\cdot \mm(f_1))  =
 T_{m_2,f_2}(\alpha) \cdot T_{m_2,f_2}\big( \mm(f_1)) \big) 
$$
$$
= [\alpha\cdot \mm(f_2)] \cdot \mm(T^{\PB_3}_{m_2,f_2}(f_1))=
[\al \cdot \mm(f)], 
$$
where $f$ is any representative of the coset $f_2 \N^{(3)}_{\PB_3}  \cdot T^{\PB_3}_{m_2, f_2} (f_1)$
in $\PB_3 / \N^{(3)}_{\PB_3}$. 

Similarly, computing $T_{m,f}(\beta)$, it is easy to see that 
$m \equiv 2 m_1 m_2 + m_1 + m_2 ~\mod~ N^{(3)}_{\ord}$.
\end{proof}
\begin{remark}  
\label{rem:practical}
Later we will see that it makes sense to focus only on $\GT$-shadows 
that can be represented by pairs $(m,f)$ with 
\begin{equation}
\label{f-in-F2}
f \in \F_2 \le \PB_3\,.
\end{equation}
Let us call such $\GT$-shadows \emph{practical}.

Using \eqref{act-on-xy} and \eqref{sh-composition}, we get the
following formula for the composition $[(m,f)] := [(m_2,f_2)] \circ [(m_1,f_1)]$
of practical $\GT$-shadows $[(m_2,f_2)] $ and $[(m_1,f_1)]$:
\begin{equation}
\label{compose-practical}
\begin{array}{c}
m : = 2 m_1 m_2 + m_1 + m_2\,, \\[0.3cm]
f(x,y) : = f_2(x,y)\, f_1(x^{2 m_2+1}, f_2(x,y)^{-1} y^{2 m_2+1} f_2(x,y)).
\end{array}
\end{equation}
Due to this observation, practical $\GT$-shadows form a subgroupoid of $\GTSh$. 

The authors {\it do not know} whether there exists 
$\N \in \NFI_{\PB_4}(\B_4)$ and an onto morphism of truncated operads 
$\PaB^{\le 4} \to \PaB^{\le 4}/\sim_{\N}$ that {\it cannot} be represented by a pair 
$(m,f) \in \bbZ \times \F_2$. 
\end{remark}

\subsubsection{The virtual cyclotomic character}
\label{sec:virt-cyclotomic}

Let us observe that to every $\N \in \NFI_{\PB_4}(\B_4)$ we assign the (finite) cyclic group 
$$
\PB_2 / \lan x_{12}^{N_{\ord}} \ran \cong \bbZ / N_{\ord} \bbZ\,,
$$ 
where $N_{\ord}$ is the index of $\N_{\PB_2}$ in $\PB_2$. Moreover, if $[(m,f)]$ is a morphism 
from $\N^{\ms}$ to $\N$ in the groupoid $\GTSh$, then both $\N$ and $\N^{\ms}$ correspond 
to the same quotient $\PB_2 / \lan x_{12}^{N_{\ord}} \ran$ of $\PB_2$.

Proposition \ref{prop:GTSh} implies that the assignment $\N \mapsto  \bbZ / N_{\ord} \bbZ$ 
upgrades to a functor $\Ch_{cyclot}$ from $\GTSh$ to the category of finite cyclic groups. 
More precisely, 
\begin{cor}
\label{cor:virt-cyclotomic}
Let $[(m,f)]$ be a morphism from $\N^{(1)}$ to $\N^{(2)}$ in the groupoid $\GTSh$. 
The assignments
\begin{equation}
\label{virt-cyclotomic}
\N \mapsto \PB_2 / \N_{\PB_2}\,, \qquad 
 [(m,f)] \mapsto \Ch_{cyclot}(m,f)  \in  \Aut(\PB_2 / \N^{(2)}_{\PB_2}), 
\end{equation}
$$
\Ch_{cyclot}(m,f) \big( x_{12} \N^{(2)}_{\PB_2} \big) : = x^{2m+1}_{12} \N^{(2)}_{\PB_2}
$$
define a functor  $\Ch_{cyclot}$ from the groupoid $\GTSh$ to the category of finite cyclic groups. 
\end{cor}
\begin{proof} Since, for every $\GT$ shadow $[(m,f)]$,  $2 m+1$ represents an invertible element of
the ring $\bbZ/ N^{(2)}_{\ord} \bbZ$, $\Ch_{cyclot}(m,f)$ is clearly an automorphism of $\PB_2 / \N^{(2)}_{\PB_2} = 
\PB_2 / \N^{(1)}_{\PB_2}$.

Thus it remains to show that $\Ch_{cyclot}$ is compatible with the composition of $\GT$-shadows. 

For this purpose, we consider two composable $\GT$-shadows: 
$[(m_1,f_1)] \in \Hom_{\GTSh}(\N^{(1)}, \N^{(2)})$ and $[(m_2,f_2)] \in  \Hom_{\GTSh}(\N^{(2)}, \N^{(3)})$. 

Since $\N^{(1)}$, $\N^{(2)}$ and  $\N^{(3)}$ belong to the same connected component of $\GTSh$, 
$\N^{(1)}_{\PB_2} = \N^{(2)}_{\PB_2} =  \N^{(3)}_{\PB_2}$ or equivalently 
$N^{(1)}_{\ord} = N^{(2)}_{\ord} = N^{(3)}_{\ord}$. So let us set 
$\N_{\PB_2} := \N^{(1)}_{\PB_2}$ and $N_{\ord} := N^{(1)}_{\ord}$. 
 
Let $[(m,f)] : = [(m_2,f_2)] \circ [(m_1,f_1)]$.  
 
Due to the first equation in \eqref{sh-composition}, $m \equiv 2 m_1 m_2 + m_1 + m_2 \,\mod\, N_{\ord}$. 
Hence 
$$
\Ch_{cyclot}(m,f) \big( x_{12} \N_{\PB_2} \big) =  x_{12}^{2(2 m_1 m_2 + m_1 + m_2)+1} \N_{\PB_2} 
$$
$$
= x_{12}^{4 m_1 m_2 + 2m_1 + 2m_2+1} \N_{\PB_2} = x_{12}^{(2 m_1+1)(2 m_2 + 1)} \N_{\PB_2}
$$
$$
= \big( x_{12}^{(2 m_1+1)} \N_{\PB_2} \big)^{2m_2+1} =  
\Ch_{cyclot}(m_2,f_2) \circ \Ch_{cyclot}(m_1,f_1)\, \big(  x_{12} \N_{\PB_2}\big). 
$$

Thus $\Ch_{cyclot}$ is indeed a functor from $\GTSh$ to the category of finite cyclic groups.
\end{proof}

We call the functor $\Ch_{cyclot}$ the {\it virtual cyclotomic character}. 
This name is justified by the following remark: 
\begin{remark}
\label{rem:virt-cyclotomic}
Let $\N \in \NFI_{\PB_4}(\B_4)$, $g \in G_{\bbQ}$ and $[(m,f)]$ 
be the $\GT$-shadow in $\GT(\N)$ induced by the element in $\GTh$ corresponding 
to $g$ then 
\begin{equation}
\label{cyclotomic_actually}
\Ch_{cyclot}(m,f) \big( x_{12} \N_{\PB_2} \big) = x_{12}^{ \chi(g)_{N_{\ord}} } \N_{\PB_2}\,,
\end{equation}
where $\chi : G_{\bbQ} \to \Zhat^{\times} \cong \Aut(\Zhat)$ is the cyclotomic character and 
$\chi(g)_{\N_{\ord}}$ represents the image of $\chi(g)$ in 
$\Aut(\bbZ/N_{\ord} \bbZ) \cong \big( \bbZ/N_{\ord} \bbZ \big)^{\times}$.   
Equation \eqref{cyclotomic_actually} follows from the discussion in \cite[Example 4.7.4]{Szamuely}
and \cite[Remark 4.7.5]{Szamuely}. See also \cite[Proposition 1.6]{Ihara-embedding}.
\end{remark}
 
\subsection{Charming $\GT$-shadows}
\label{sec:charm-GT}

Recall that $\PB_3$ is isomorphic to $\F_2 \times \bbZ$ where the $\F_2$-factor is freely generated 
by $x_{12}$ and $x_{23}$ and the $\bbZ$-factor is generated by the central element $c$ 
given in (\ref{c-PB3-B3}). This implies that $\wh{\PB}_3 \cong \Fh_2 \times \Zhat$. 
Due to the following proposition, the action of $\GTh$ on $\wh{\PB}_3$ (viewed as 
the automorphism group of $(12)3$ in $\wh{\PaB}$) respects this decomposition: 
\begin{prop}  
\label{prop:T-hat-F2}
For every (continuous) automorphism $\hat{T}$ of $\wh{\PaB}^{\le 4}$, its restriction 
to the subgroup $\Fh_2 \le \wh{\PB}_3$ gives us an 
automorphism\footnote{In fact, some specialists like to define $\GTh$ as a certain 
subgroup of continuous automorphisms of $\Fh_2$.} of $\Fh_2$
$$
\hat{T}\big|_{\Fh_2} : \Fh_2 \to \Fh_2
$$
defined by the formulas
\begin{equation}
\label{T-hat-F2}
\hat{T}(x):=x^{2\hat{m}+1}, \qquad 
\hat{T}(y) := \hat{f}^{-1} y^{2\hat{m}+1} \hat{f}.
\end{equation}
The restriction of $\hat{T}$ to the central factor $\Zhat$ of $\wh{\PB}_3$ gives us 
the continuous automorphism of $\Zhat$ defined by the formula 
\begin{equation}
\label{T-hat-Zhat}
\hat{T}(c) := c^{2\hat{m}+1}\,.
\end{equation}
\end{prop}   
\begin{proof} Due to Proposition \ref{prop:subgroups-in-PB3-2}, 
the action of $\hat{T}$ on $\wh{\PB}_3$ is determined by the group homomorphisms
$$
T^{\PB_3}_{m,f}: \PB_3 \to  \PB_3/\N_{\PB_3}\,, \qquad \N \in \NFI_{\PB_4}(\B_4)
$$
corresponding to $\GT$-shadows $[(m,f)]$ that come from $\hat{T}$.  

Combining this observation with equations \eqref{act-on-xy} and the 
second equation in \eqref{act-on-x-c} and using the fact that $c$ is a central element 
of $\PB_3$, we conclude that the restrictions of $\hat{T}$ to $\Fh_2$ and 
to $\Zhat$ give us group homomorphisms 
\begin{equation}
\label{F2-hat-and-Z-hat}
\hat{T}\big|_{\Fh_2} : \Fh_2 \to \Fh_2
\quad
\textrm{and}
\quad
\hat{T}\big|_{\Zhat} : \Zhat \to \Zhat,
\end{equation}
respectively.

Since the restrictions of the inverse of $\hat{T}$ to $\Fh_2$ and to $\Zhat$
give us inverses of the two homomorphisms in \eqref{F2-hat-and-Z-hat}, respectively, 
the homomorphisms in \eqref{F2-hat-and-Z-hat} are indeed automorphisms. 

Explicit formulas \eqref{T-hat-F2} and \eqref{T-hat-Zhat} are consequences 
of equations \eqref{act-on-xy} and the second equation in \eqref{act-on-x-c}.
\end{proof}

If a $\GT$-shadow $[(m,f)]$ comes from an automorphism of $\wh{\PaB}$ then
it satisfies further conditions. The following definition is motivated by these conditions. 
\begin{defi}  
\label{dfn:charm}
Let $\N \in \NFI_{\PB_4}(\B_4)$. A $\GT$-shadow $[(m,f)] \in \GT(\N)$ is called \emph{genuine}
if it comes from an automorphism of $\wh{\PaB}$. Otherwise, $[(m,f)]$ is called \emph{fake}.
Furthermore, a $\GT$-shadow $[(m,f)] \in \GT(\N)$ is called \emph{charming} if
\begin{itemize}

\item the coset $f \N_{\PB_3}$ can be represented by $f_1 \in [\F_2, \F_2]$ and

\item the group homomorphism 
\begin{equation}
\label{T-F-2}
T^{\F_2}_{m,f} \,:= \, T^{\PB_3}_{m,f}\, \big|_{\F_2} \,:\, \F_2 ~\to~ \F_2/(\N_{\PB_3} \cap \F_2)
\end{equation}
is onto.

\end{itemize}
\end{defi}  

Since the intersection $\N_{\PB_3} \cap \F_2$ plays an important role, 
we will denote it by $\N_{\F_2}$. 
\begin{equation}
\label{N-F-2}
\N_{\F_2} : = \N_{\PB_3} \cap \F_2\,.
\end{equation}
Clearly, the kernel of the homomorphism  $T^{\F_2}_{m,f} : \F_2 ~\to~ \F_2/\N_{\F_2} $
coincides with $\N^{\ms}_{\F_2}$ and $|\F_2 : \N^{\ms}_{\F_2}| = |\F_2 : \N_{\F_2}|$ for 
every charming $\GT$-shadow $[(m,f)]$.   

Let us prove that 
\begin{prop}  
\label{prop:gen-charm}
Every genuine $\GT$-shadow is charming. 
\end{prop}
\begin{proof}
Let $\N \in \NFI_{\PB_4}(\B_4)$ and $[(m,f)] \in \GT(\N)$ be a genuine $\GT$-shadow. 
The element $f \in \PB_3$ can be written uniquely as
$$
f = g\, c^k,
$$
where $g \in \F_2$, $k \in \bbZ$ and $c$ is defined in \eqref{c-PB3-B3}.   

Since $\N_{\PB_3}$ is a normal subgroup of finite index in $\PB_3$, 
the subgroup $\N_{\F_2} : = \N_{\PB_3} \cap \F_2$ is normal in $\F_2$ and it has a finite
index in $\F_2$. 
Similarly, the subgroup $\N_{\bbZ} := \N_{\PB_3} \cap \bbZ$ has a finite index in $\bbZ$. 
Therefore, the subgroup $\N_{\F_2} \times \N_{\bbZ}$ is normal and it has finite index in $\PB_3$. 

Due to Proposition \ref{prop:subgroups-in-PB3-2}, there exists
$\K \in \NFI_{\PB_4}(\B_4)$ such that $\K_{\PB_3}$ is contained in $\N_{\F_2} \times \N_{\bbZ}$. 
Since $[(m,f)]$ is a genuine $\GT$-shadow, there exists $(m_1, f_1) \in \bbZ \times \PB_3$
such that $(m_1, f_1)$ represents the same $\GT$-shadow $[(m,f)]$ in $\GT(\N)$ and
$[(m_1,f_1)] \in \GT(\K)$. 

Thus, without loss of generality, we may assume that $m = m_1$ and $f = f_1$, i.e. $[(m,f)] \in \GT(\K)$.

Using relation \eqref{hexa1}, we have 
$$
\si_1 x_{12}^m \, f^{-1} \si_2 x_{23}^m f \, \K_{\PB_3} ~ = ~ 
f^{-1} \si_1 \si_2 (x_{13} x_{23})^m \, \K_{\PB_3}\,.
$$
Next, using (\ref{c-PB3-B3}) and the fact that $c$ is a central element of $\B_3$,
we get that
$$
x_{12}^{m}\,\si_2^{-1}\,\si_1^{-1}g(x_{12},x_{23})
\si_1\, x_{12}^m g(x_{12},x_{23})^{-1} \si_2\, x_{23}^m g(x_{12},x_{23})\, c^{-m+k} \, \in \K_{\PB_3}\,.
$$

Using equations in \eqref{conj-xxx-PB3} from Appendix \ref{app:braids}, we deduce that 
\begin{equation*}
 x_{12}^{m}\,\,g(x_{23}^{-1}x_{12}^{-1}c, x_{12})\, 
 (x_{23}^{-1}x_{12}^{-1}c)^m g(x_{23}^{-1}x_{12}^{-1}c,x_{23})^{-1} \, 
 x_{23}^m g(x_{12},x_{23})\, c^{-m+k} \, \in \K_{\PB_3}\,, \text{ or}
\end{equation*}
\begin{equation*}
x_{12}^{m}\,\,g(x_{23}^{-1}x_{12}^{-1},x_{12})\, (x_{23}^{-1}x_{12}^{-1})^m g(x_{23}^{-1}x_{12}^{-1},x_{23})^{-1} \, x_{23}^m g(x_{12},x_{23})\, c^{k} \, \in \K_{\PB_3}.
\end{equation*}
Since $\K_{\PB_3}$ is a subgroup of $\N_{\F_2} \times \N_{\bbZ}$, we have $c^k \in \N_{\bbZ} \subset \N_{\PB_3}$. 
Hence $f c^{-k} \N_{\PB_3}= g(x_{12},x_{23})\N_{\PB_3}$, and so the $\GT$-shadow 
has a representative of the form $(m,f)$ where $f \in \F_2$. 

It remains to show that 
\begin{itemize}

\item $[(m,f)]$ can be represented by a pair $(m,f_1)$ with $f_1 \in [\F_2, \F_2]$ and

\item homomorphism \eqref{T-F-2} is onto. 

\end{itemize}

Since homomorphism \eqref{T-F-2} does not depend on the choice of the representative 
of the $\GT$-shadow $[(m,f)]$, we first prove that this homomorphism is indeed onto.

Due to Proposition \ref{prop:T-hat-F2}, we have the following commutative diagram:
\begin{center}
\begin{tikzcd}
\Fh_{2} \arrow[r, "\hat{T}|_{\Fh_{2}}"] & \Fh_{2} \arrow[d, "\hcP_{\N_{\F_{2}}}"] \\
\F_{2} \arrow[u, "i"] \arrow[r, "{T_{m,f}^{\F_{2}}}"] &  \F_2 / \N_{\F_2}
\end{tikzcd}
\end{center}
Since $\F_{2}$ is dense in $\Fh_{2}$, we get that the composition 
$\hcP_{\N_{\F_{2}}} \circ \hat{T}|_{\Fh_{2}} \circ i$ is surjective whence we conclude 
$T_{m,f}^{\F_{2}}$ is onto.

Let us now prove that $[(m,f)]$ can be represented by a pair $(m,\ti{f})$ with $\ti{f} \in [\F_2, \F_2]$. 

Let $q$ be the least common multiple of the orders of $x_{12}\N_{\F_2}$ and  $x_{23}\N_{\F_2}$ in $\F_2 / \N_{\F_2}$
and $\psi_x : \PB_4 \to S_{q}$, $\psi_y : \PB_4 \to S_{q}$  be the group homomorphisms defined by equations
$$
\psi_x(x_{12}) : = (1,2, \dots, q), \qquad  
\psi_x(x_{23}) = \psi_x(x_{13}) =  \psi_x(x_{14}) =  \psi_x(x_{24}) =  \psi_x(x_{34}) := \id_{S_q}
$$
and 
$$
\psi_y(x_{34}) : = (1,2, \dots, q), \qquad  
\psi_y(x_{12}) = \psi_y(x_{23}) =  \psi_y(x_{13}) =  \psi_y(x_{14}) =  \psi_y(x_{24}) := \id_{S_q}\,,
$$
respectively. 

Let $\K$ be an element of $\NFI_{\PB_4}(\B_4)$ such that 
\begin{equation}
\label{K-in-kernels}
\K \le \N \cap \ker(\psi_x) \cap \ker(\psi_y).
\end{equation}

Since $[(m,f)]$ is a genuine $\GT$-shadow, there exists a $\GT$-shadow 
$ [(m_1,f_1)] \in \GT(K)$ such that $(m_1, f_1)$ is also a representative of $[(m,f)]$. 
We can assume, without loss of generality, that $f_1 \in \F_2$. 

Applying equation \eqref{GT-penta} to $f_1$ we see that 
\begin{equation}
\label{penta-in-K}
f^{-1}_1(x_{13}x_{23}, x_{34}) f^{-1}_1(x_{12}, x_{23}x_{24})
f_1(x_{23}, x_{34}) f_1(x_{12}x_{13}, x_{24}x_{34}) f_1(x_{12}, x_{23}) \in \K. 
\end{equation}

Inclusions \eqref{K-in-kernels} and \eqref{penta-in-K} imply that
$$
\psi_x \big( f^{-1}_1(x_{13}x_{23}, x_{34}) f^{-1}_1(x_{12}, x_{23}x_{24})
f_1(x_{23}, x_{34}) f_1(x_{12}x_{13}, x_{24}x_{34}) f_1(x_{12}, x_{23}) \big) = \id_{S_q}\,,
$$
and 
$$
\psi_y \big( f^{-1}_1(x_{13}x_{23}, x_{34}) f^{-1}_1(x_{12}, x_{23}x_{24})
f_1(x_{23}, x_{34}) f_1(x_{12}x_{13}, x_{24}x_{34}) f_1(x_{12}, x_{23}) \big) = \id_{S_q}\,.
$$
Hence the sum $s_x$ of exponents of $x_{12}$ in $f_1$ and the sum $s_y$ of exponents of $x_{23}$ in $f_1$ are 
multiples of $q$, i.e. $x_{12}^{-s_x} \in \N_{\F_2}$ and $x_{23}^{-s_y} \in \N_{\F_2}$.  

Thus $(m, f_1 x_{12}^{-s_x} x_{23}^{-s_y})$ is yet another representative of the $\GT$-shadow $[(m,f)]$ in $\GT(\N)$ 
and, by construction, $ f_1 x_{12}^{-s_x} x_{23}^{-s_y} \in [\F_2, \F_2]$. 
\end{proof}

\bigskip
The following statement can be found in many introductory (and ``not so introductory'') papers
on the Grothendieck-Teichmueller group $\GTh$. Here, we deduce it from Proposition \ref{prop:gen-charm}.  
\begin{cor}  
\label{cor:comm-F2}
For every $(\hat{m}, \hat{f}) \in \GTh$, $\hat{f}$ belongs to the topological 
closure of commutator subgroup of $\Fh_2$.   
\end{cor}  
\begin{proof}
It suffices to show that, for every $\N \in \NFI(\F_2)$, the element 
$\hcP_{\N}(\hat{f}) \in  \F_2/\N$ can be represented by $f_1 \in [\F_2, \F_2]$. 
Let us observe that $\N \times \lan c \ran \in \NFI(\PB_3)$. 

Due to Proposition \ref{prop:subgroups-in-PB3-2}, there exists $\K \in \NFI_{\PB_4}(\B_4)$
such that $\K_{\PB_3} \le \N \times \lan c \ran$. Clearly, $\K_{\F_2} \le \N$. 

Since the pair $(\hcP_{K_{\ord}}(\hat{m}), \hcP_{\K_{\F_2} }(\hat{f}))$ is
a charming $\GT$-shadow in $\GT(\K)$, the element  
$\hcP_{\K_{\F_2} }(\hat{f})\in \F_2 / \K_{\F_2} $ can be represented by 
$f_1 \in [\F_2, \F_2]$. Since $\K_{\F_2} \le \N$, the same element $f_1 \in [\F_2, \F_2]$
represents the coset $\hcP_{\N}(\hat{f}) \in \F_2/\N$.  
\end{proof}

\bigskip
\bigskip

Let us denote by $\GT^{\hs}(\N)$ the subset of all charming $\GT$-shadows in $\GT(\N)$ 
and prove that  $\GT(\N)$ can be safely replaced by $\GT^{\hs}(\N)$ in all the 
constructions of Section \ref{sec:GTSh}. More precisely, 
\begin{prop}  
\label{prop:charm}
The assignment 
\begin{equation}
\label{GT-sh-morph-charm}
\Hom_{\GTSh^{\hs}}(\ti{\N}, \N) ~ : = ~ \{ [(m,f)] \in \GT^{\hs}(\N) ~|~ \ti{\N} = \ker(T^{\PB_4}_{m,f}) \}, 
\quad 
\ti{\N}, \N \in \NFI_{\PB_4}(\B_4)
\end{equation} 
upgrades the set  $\NFI_{\PB_4}(\B_4)$ to a groupoid. 
\end{prop}  
\begin{proof} Let $[(m_1,f_1)] \in \Hom_{\GTSh^{\hs}}(\N^{(1)}, \N^{(2)})$ and 
$[(m_2,f_2)] \in \Hom_{\GTSh^{\hs}}(\N^{(2)}, \N^{(3)})$.
Since the $\GT$-shadows $[(m_1,f_1)]$ and $[(m_2, f_2)]$
are charming, we may assume, without loss of generality, that $f_1, f_2 \in [\F_2, \F_2]$. 

Due to Remark \ref{rem:practical}, the composition $[(m_2,f_2)] \circ  [(m_1,f_1)]$  
is represented by a pair $(m,f)$ with
$$
f = f_2 f_1(x^{2m_2+1}, f_2^{-1}(x,y) y^{2m_2+1} f_2(x,y)).  
$$
Since $f_1, f_2 \in [\F_2, \F_2]$, it is clear that $f$ also belongs to $[\F_2, \F_2]$. 

Since $T^{\F_2}_{m,f}: \F_2 \to \F_2 / \N^{(3)}$ is the composition of 
the onto homomorphism $T^{\F_2}_{m_1,f_1}: \F_2 \to \F_2 / \N^{(2)}$ and 
the isomorphism $T^{\F_2, \isom}_{m_2,f_2}: \F_2/ \N^{(2)}  \to \F_2 / \N^{(3)}$, 
the homomorphism $T^{\F_2}_{m,f}$ is also onto. 

We proved that the subset of charming $\GT$-shadows is closed under composition. 

To prove that the subset of charming $\GT$-shadows is closed under taking inverses, 
we start with 
a charming $\GT$-shadow $[(m,f)] \in \Hom_{\GTSh^{\hs}}(\N^{\ms}, \N)$ and assume 
that $f \in [\F_2, \F_2]$. 
Let $[(\ti{m},\ti{f})] \in \Hom_{\GTSh}(\N, \N^{\ms})$ be the inverse of $[(m,f)]$ in $\GTSh$. 
In other words, 
$$
2 m\ti{m} + m + \ti{m} \equiv 0 ~~\mod~~ N_{\ord}
$$
and 
\begin{equation}
\label{ti-f-f}
f \, T^{\PB_3}_{m,f} (\ti{f}) = 1_{\PB_3/\N_{\PB_3}}\,.
\end{equation}
Our goal is to show that the coset $\ti{f} \N_{\PB_3}$ can be represented by $g \in [\F_2, \F_2]$. 

Since $f^{-1}$ belongs to $[\F_2, \F_2]$, we have 
$$
f^{-1} = [g_{11}, g_{12}] [g_{21}, g_{22}] \dots  [g_{r1}, g_{r2}],
$$
where each $g_{ij} \in \F_2$ and $[g_1,g_2] := g_1 g_2 g_1^{-1} g_2^{-1}$.

Since the homomorphism $T^{\F_2}_{m,f} : \F_2 \to \F_2 / \N_{\F_2}$ is onto, 
for every $g_{ij}$, there exists $\ti{g}_{ij} \in \F_2$ such that 
$T^{\F_2}_{m,f}(\ti{g}_{ij}) = g_{ij} \N_{\F_2}$. Hence, for 
$$
g := [\ti{g}_{11}, \ti{g}_{12}] [\ti{g}_{21}, \ti{g}_{22}] \dots  [\ti{g}_{r1}, \ti{g}_{r2}] ~\in~  [\F_2, \F_2]
$$
we have $T^{\PB_3}_{m,f} (g) = f^{-1} \N_{\PB_3}$ or equivalently 
\begin{equation}
\label{g-and-f-inv}
f  T^{\PB_3}_{m,f} (g) = 1_{\PB_3/ \N_{\PB_3}}\,.
\end{equation}

Combining \eqref{ti-f-f} with \eqref{g-and-f-inv} we conclude that the element 
$g^{-1} \ti{f}$ belongs to the kernel of $T^{\PB_3}_{m,f} : \PB_3 \to \PB_3/ \N_{\PB_3}$. 
Thus, due to Proposition \ref{prop:sim-s-OK}, $g$ also represents the coset $\ti{f} \N_{\PB_3}$.  

Since, by construction $g \in [\F_2, \F_2]$, the desired statement is proved.  
\end{proof}

\section{The Main Line functor $\ML$ and $\GTh$}
\label{sec:ML}
In this section, we use (charming) $\GT$-shadows to construct a functor $\ML$ from a certain subposet 
of $\NFI_{\PB_4}(\B_4)$ to the category of finite groups. We prove that the limit of the functor $\ML$
is isomorphic to the Grothendieck-Teichmueller group $\GTh$.

\subsection{Connected components of $\GTSh^{\hs}$, settled $\GT$-shadows and 
isolated elements of $\NFI_{\PB_4}(\B_4)$}
\label{sec:settled-isolated}

Since the set $\NFI_{\PB_4}(\B_4)$ is infinite, so is the groupoid $\GTSh^{\hs}$. Moreover, the groupoid $\GTSh^{\hs}$ is 
highly disconnected. Indeed, if $\ti{\N}$ and $\N$ are connected by a morphism in $\GTSh^{\hs}$, then they must have the 
same index in $\PB_4$.  

For $\N \in \NFI_{\PB_4}(\B_4)$ we denote by 
$$
\GTSh^{\hs}_{\conn}(\N)
$$
the connected component of $\N$ in the groupoid $\GTSh^{\hs}$. 
Clearly, an element $\ti{\N}$ of $\NFI_{\PB_4}(\B_4)$ is an object of $\GTSh^{\hs}_{\conn}(\N)$ 
if and only if there exists $[(m,f)] \in \GT^{\hs}(\N)$ such that 
$$
\ti{\N} = \ker(T^{\PB_4}_{m,f}).
$$
We call objects of the groupoid $\GTSh^{\hs}_{\conn}(\N)$ \emph{conjugates} of $\N$. 

Since $\GT^{\hs}(\N)$ is a finite set for every $\N \in \NFI_{\PB_4}(\B_4)$, it is easy to show that 
\begin{prop}
\label{prop:GTSh-conn-finite}
For every $\N \in \NFI_{\PB_4}(\B_4)$, the (connected) groupoid $\GTSh^{\hs}_{\conn}(\N)$ is finite. 
\end{prop}
\qed

To establish a more precise link between (charming) $\GT$-shadows and the group $\GTh$, we 
will be interested in a certain subposet of $\NFI_{\PB_4}(\B_4)$. Let us start with 
the following definition:  
\begin{defi}  
\label{dfn:settled}
Let $\N \in \NFI_{\PB_4}(\B_4)$ and $[(m,f)] \in \GT^{\hs}(\N)$. 
A charming $\GT$-shadow 
$[(m,f)]$ is called \emph{settled} if its source coincides with $\N$, i.e. 
$\ker(T^{\B_4}_{m,f}) = \N$. An element $\N$ of the poset $\NFI_{\PB_4}(\B_4)$ 
is called \emph{isolated} if every $\GT$-shadow in $\GT^{\hs}(\N)$ is settled. 
\end{defi}  
Clearly, a $\GT$-shadow $[(m,f)] \in \GT^{\hs}(\N)$ is settled if and only if $[(m,f)]$ is an automorphism 
of the object $\N$ in the groupoid $\GTSh^{\hs}$. Moreover, 
an element $\N \in  \NFI_{\PB_4}(\B_4)$ is isolated if and only if the groupoid 
$\GTSh^{\hs}_{\conn}(\N)$ has exactly one object. In this case, 
$\GT^{\hs}(\N)$ is the group of automorphisms of the object $\N$ in the groupoid $\GTSh^{\hs}$.

The following proposition gives us a simple way to produce many examples of 
isolated elements of $\NFI_{\PB_4}(\B_4)$:   
\begin{prop}  
\label{prop:isolated-practice}
For every $\N \in \NFI_{\PB_4}(\B_4)$, the normal subgroup
\begin{equation}
\label{N-sharp}
\N^{\sh} : = \bigcap_{\K \in \Ob(\GTSh^{\hs}_{\conn}(\N))} \K
\end{equation}
is an isolated element of $\NFI_{\PB_4}(\B_4)$. 
\end{prop}        
\begin{proof} Let $[(m,f)] \in \GT^{\hs}(\N^{\sh})$ and $\N^{\sh, \ms}$ be the source of the corresponding 
morphism in $\GTSh^{\hs}$, i.e. $\N^{\sh, \ms} : = \ker(T^{\PB_4}_{m,f})$.

Since $\N^{\sh} \le \K$, the same pair $(m,f)\in \bbZ \times \F_2$ represents a $\GT$-shadow 
in $\GT^{\hs}(\K)$. Moreover, the homomorphism from $\PB_4$ to $\PB_4/\K$ corresponding 
to $[(m,f)] \in \GT^{\hs}(\K)$ is the composition $\cP_{\N^{\sh}, \K} \circ T^{\PB_4}_{m,f}$ of $T^{\PB_4}_{m,f}$
with the canonical projection 
$$
\cP_{\N^{\sh}, \K} : \PB_4/\N^{\sh} \to  \PB_4/ \K\,. 
$$

Let $h \in \N^{\sh}$, $\ti{h} \in \PB_4$ be a representative of $T^{\PB_4}_{m,f}(h)$ 
and $\K^{\ms}$ be the source of the $\GT$-shadow  $[(m,f)] \in \GT^{\hs}(\K)$ 
(i.e. $\K^{\ms} : = \ker\big( \cP_{\N^{\sh}, \K} \circ T^{\PB_4}_{m,f}\big)$)

Since $\N^{\sh} \le \K^{\ms}$, we have  
\begin{equation}
\label{goes-to-1-mod-K}
\cP_{\N^{\sh}, \K} \big( T^{\PB_4}_{m,f}(h) \big) = 1_{\PB_4/ \K}\,.
\end{equation}

Identity \eqref{goes-to-1-mod-K} implies that
$\ti{h} \in \K$ for every $\K \in \Ob(\GTSh^{\hs}_{\conn}(\N))$ 
and hence $\ti{h} \in \N^{\sh}$. Therefore $T^{\PB_4}_{m,f}(h) = 1_{\PB_4/\N^{\sh}}$ 
or equivalently $h \in \N^{\sh, \ms}$. 
 
We proved that 
\begin{equation}
\label{source-in-sharp}
\N^{\sh} \le \N^{\sh, \ms}\,.
\end{equation}

Since these subgroups have the same index in $\PB_4$, inclusion 
\eqref{source-in-sharp} implies that $\N^{\sh, \ms} = \N^{\sh}$.

Since we started with an arbitrary $\GT$-shadow in $\GT^{\hs}(\N^{\sh})$, 
we proved that $\N^{\sh}$ is indeed an isolated element of $\NFI_{\PB_4}(\B_4)$. 
\end{proof}
\begin{remark}  
\label{rem:isolated-practice}
In all examples we have considered so far (see Section \ref{sec:comp-exp} on selected results of 
computer experiments),  $\GTSh^{\hs}_{\conn}(\N)$ has at most two objects. 
Hence equation \eqref{N-sharp} gives us a practical way to produce examples of isolated 
elements of $\NFI_{\PB_4}(\B_4)$.  
\end{remark}

Let us denote by 
\begin{equation}
\label{NFI-isolated}
\NFI^{isolated}_{\PB_4}(\B_4)
\end{equation}
the subposet of isolated elements of $\NFI_{\PB_4}(\B_4)$. 

Since $\N^{\sh} \le \N$ for every $\N \in \NFI_{\PB_4}(\B_4)$, 
Proposition \ref{prop:isolated-practice} implies that
\begin{cor}  
\label{cor:isolated-cofinal}
The subposet $\NFI^{isolated}_{\PB_4}(\B_4)$ of $\NFI_{\PB_4}(\B_4)$ is \emph{cofinal}. 
In other words, for every $\N \in \NFI_{\PB_4}(\B_4)$, there exists 
$\K \in \NFI^{isolated}_{\PB_4}(\B_4)$ such that $\K \le \N$. $\qed$
\end{cor}  

Although, Corollary \ref{cor:isolated-cofinal} implies that 
the poset $\NFI^{isolated}_{\PB_4}(\B_4)$ is directed (it is a cofinal subposet of 
a directed poset), it is still useful to know that the intersection of two isolated 
elements of $\NFI_{\PB_4}(\B_4)$ is an isolated element of $\NFI_{\PB_4}(\B_4)$:
\begin{prop}  
\label{prop:intersect-isolated}
For every $\N^{(1)}, \N^{(2)} \in \NFI^{isolated}_{\PB_4}(\B_4)$, 
$$
\N^{(1)} \cap \N^{(2)}
$$
is also an isolated element of $\NFI_{\PB_4}(\B_4)$. 
\end{prop}  
\begin{proof} $\K := \N^{(1)} \cap \N^{(2)}$ is clearly an element of $\NFI_{\PB_4}(\B_4)$.
So our goal is to prove that $\K$ is isolated.

Let $[(m,f)] \in \GT^{\hs}(\K)$ and $\K^{\ms}$ be the kernel of the homomorphism 
$T^{\PB_4}_{m,f} : \PB_4 \to \PB_4 / \K $.

Recall that $\cP_{\K, \N^{(1)} }$ (resp. $\cP_{\K, \N^{(2)} }$) is the canonical homomorphism 
from $\PB_4/\K$ to $\PB_4/\N^{(1)}$ (resp. to  $\PB_4/\N^{(2)}$). 
Since $\K \le \N^{(1)}$ and $\K \le \N^{(2)}$, the pair $(m,f)$ also represents 
a $\GT$-shadow in $\GT^{\hs}(\N^{(1)})$ and a $\GT$-shadow in $\GT^{\hs}(\N^{(2)})$. 
Moreover, the compositions $\cP_{\K, \N^{(1)} } \circ T^{\PB_4}_{m,f}$ and 
$\cP_{\K, \N^{(2)} } \circ T^{\PB_4}_{m,f}$ are the homomorphisms $\PB_4 \to \PB_4/\N^{(1)}$
and $\PB_4 \to \PB_4/\N^{(2)}$ corresponding to these $\GT$-shadows in 
$\GT^{\hs}(\N^{(1)})$ and  $\GT^{\hs}(\N^{(2)})$, respectively. 

Let us now consider $h \in \K^{\ms}$. 
Since $T^{\PB_4}_{m,f}(h) = 1_{\PB_4/\K}$, we have
\begin{equation}
\label{both-to-1}
\cP_{\K, \N^{(1)} } \circ T^{\PB_4}_{m,f} (h) = 1_{\PB_4/\N^{(1)}}\,, \qquad
\cP_{\K, \N^{(2)} } \circ T^{\PB_4}_{m,f} (h) = 1_{\PB_4/\N^{(2)}}. 
\end{equation}

Since $\N^{(1)}$, $\N^{(2)}$ are both isolated, identities \eqref{both-to-1} imply that
$h \in \N^{(1)}$ and $h \in \N^{(2)}$. Hence $h \in \K$. 

Since we showed that $\K^{\ms} \le \K$ and both subgroups have the same (finite) index in $\PB_4$, 
we have the desired equality $\K^{\ms} = \K$. 
\end{proof}
\bigskip

Recall that, for every isolated element $\N \in \NFI_{\PB_4}(\B_4)$, 
the set $\GT^{\hs}(\N)$ is a finite group. More precisely, $\GT^{\hs}(\N)$ is 
the (finite) group of automorphisms of $\N$ in the groupoid $\GTSh^{\hs}$. 
Let us denote this finite group by $\ML(\N)$ and prove that
\begin{prop}
\label{prop:ML-functor}
The assignment 
$$
\N \mapsto \ML(\N)
$$
upgrades to a functor $\ML$ from the poset $\NFI^{isolated}_{\PB_4}(\B_4)$ to the category of finite groups. 
\end{prop}
\begin{proof}
Let $\K \le \N$ be isolated elements of $\NFI_{\PB_4}(\B_4)$. Our goal is to define a group homomorphism 
\begin{equation}
\label{ML-for-morphisms}
\ML_{\K, \N} :  \ML(\K) \to  \ML(\N) 
\end{equation}
and show that, for every triple of nested elements $\N^{(1)} \le \N^{(2)} \le  \N^{(3)}$ of $\NFI^{isolated}_{\PB_4}(\B_4)$,
\begin{equation}
\label{ML-OK}
\ML_{\N^{(2)}, \N^{(3)}}  \circ \ML_{\N^{(1)}, \N^{(2)}}  = \ML_{\N^{(1)}, \N^{(3)}}\,. 
\end{equation}

For this proof, it is convenient to identify $\GT$-shadows $[(m,f)] \in \GT^{\hs}(\K)$ with the corresponding 
onto morphisms $T_{m,f}: \PaB^{\le 4} \to \PaB^{\le 4}/\sim_{\K}$ of truncated operads. 
So let $[(m,f)] \in \GT^{\hs}(\K)$ and $T_{m,f}$ be the corresponding morphism.  

Recall that $\cP_{\K, \N}$ denotes the canonical onto morphism of truncated operads
$$
\cP_{\K, \N} : \PaB^{\le 4}/\sim_{\K} ~\to~  \PaB^{\le 4}/\sim_{\N}\,.
$$
 
Composing $\cP_{\K, \N} $ with $T_{m,f}$ we get an onto morphism 
$$
\cP_{\K, \N}  \circ T_{m,f} : \PaB^{\le 4} \to \PaB^{\le 4}/\sim_{\N}
$$
and hence an element of $\GT^{\hs}(\N)$.  

We set 
\begin{equation}
\label{ML-K-N}
\ML_{\K, \N} (T_{m,f}) : = \cP_{\K, \N}  \circ T_{m,f}\,.
\end{equation}

To prove that $\ML_{\K, \N}$ is a group homomorphism from $\ML(\K)$ to $\ML(\N)$, 
we recall that, since $\K$ is isolated, every  onto morphism of truncated operads 
$T : \PaB^{\le 4} \to \PaB^{\le 4}/\sim_{\K}$ factors as follows 
$$
\begin{tikzpicture}
\matrix (m) [matrix of math nodes, row sep=1.8em, column sep=2.1em]
{\PaB^{\le 4}   &  \\
\PaB^{\le 4} / \sim_{\K}  & \PaB^{\le 4} / \sim_{\K} \\ };
\path[->, font=\scriptsize]
(m-1-1) edge node[above] {$T$} (m-2-2) edge node[left] {$\cP_{\K}$} (m-2-1)
(m-2-1) edge node[above] {$T^{\isom}$} (m-2-2); 
\end{tikzpicture}
$$

Let us now show that, for every onto morphism of truncated operads 
$T : \PaB^{\le 4} \to \PaB^{\le 4}/\sim_{\K}$, the diagram
\begin{equation}
\label{p-T-diagram}
\begin{tikzpicture}
\matrix (m) [matrix of math nodes, row sep=1.8em, column sep=3.2em]
{\PaB^{\le 4}   &  \\
\PaB^{\le 4} / \sim_{\K}  & \PaB^{\le 4} / \sim_{\K} \\ 
 \PaB^{\le 4} / \sim_{\N}  & \PaB^{\le 4} / \sim_{\N} \\  };
\path[->, font=\scriptsize]
(m-1-1) edge node[above] {$T$} (m-2-2) edge node[left] {$\cP_{\K}$} (m-2-1)
(m-2-1) edge node[above] {$T^{\isom}$} (m-2-2)
(m-2-2) edge node[left] {$\cP_{\K, \N}$} (m-3-2)
(m-2-1) edge node[left] {$\cP_{\K, \N}$} (m-3-1) 
(m-3-1) edge node[above] {$(\cP_{\K, \N} \circ T)^{\isom}$} (m-3-2) ; 
\end{tikzpicture}
\end{equation}
commutes.

Since the top triangle of \eqref{p-T-diagram} commutes by definition of $T^{\isom}$, we only need 
to prove the commutativity of the square. Let $\ga \in \PaB^{\le 4}$ and $[\ga]_{\K}$ (resp.  $[\ga]_{\N}$) 
be equivalence classes of $\ga$ in $\PaB^{\le 4}/\sim_{\K}$ (resp. in $\PaB^{\le 4}/\sim_{\N}$). 
Since $T^{\isom}([\ga]_{\K}) = T(\ga)$,  $(\cP_{\K, \N} \circ T)^{\isom}([\ga]_{\N}) = \cP_{\K, \N}  \circ T(\ga)$ and 
$\cP_{\K, \N} ([\ga]_{\K}) = [\ga]_{\N}$, we have 
$$
\cP_{\K, \N}  \circ T^{\isom}([\ga]_{\K}) = \cP_{\K, \N}  \circ T(\ga) = 
(\cP_{\K, \N}  \circ T)^{\isom}([\ga]_{\N}) = (\cP_{\K, \N}  \circ T)^{\isom} \circ \cP_{\K, \N} ([\ga]_{\K}).
$$  
Thus \eqref{p-T-diagram} indeed commutes. 

Now let $T_1$ and $T_2$ be onto morphisms (of truncated operads)
$$
T_1, T_2 :  \PaB^{\le 4} \to \PaB^{\le 4}/\sim_{\K}\,.
$$ 

Since  
$$
T^{\isom}_1 \circ T_2 : \PaB^{\le 4} \to \PaB^{\le 4}/\sim_{\K}
$$
is the composition of $T_1$ and $T_2$ in $\GTSh^{\hs}$ and 
$$
(\cP_{\K, \N} \circ T_1)^{\isom} \circ (\cP_{\K, \N}  \circ T_2) :  \PaB^{\le 4} \to \PaB^{\le 4}/\sim_{\N}
$$
is the composition of $\cP_{\K, \N}  \circ T_1$ and $\cP_{\K, \N}  \circ T_2$ in $\GTSh^{\hs}$, 
our goal is to prove that
\begin{equation}
\label{p-T-T}
\cP_{\K, \N}  \circ ( T^{\isom}_1 \circ T_2 ) =  (\cP_{\K, \N}  \circ T_1)^{\isom} \circ (\cP_{\K, \N}  \circ T_2).
\end{equation}

Due to commutativity of \eqref{p-T-diagram} for $T = T_1$ we have 
$$
\cP_{\K, \N}  \circ  T^{\isom}_1 \circ T_2 = (\cP_{\K, \N}  \circ T_1)^{\isom} \circ \cP_{\K, \N}  \circ T_2\,.
$$
Thus equation \eqref{p-T-T} indeed holds and we proved that $\ML_{\K, \N}$ is a group homomorphism. 

\bigskip

Let us now consider isolated elements $\N^{(1)} \le \N^{(2)} \le  \N^{(3)}$ of $\NFI_{\PB_4}(\B_4)$. 
Since 
$$
\cP_{ \N^{(1)}, \N^{(3)} } = \cP_{ \N^{(2)}, \N^{(3)} } \circ \cP_{ \N^{(1)}, \N^{(2)} }\,,
$$
we have
$$
\ML_{\N^{(2)}, \N^{(3)}}
\circ \ML_{\N^{(1)}, \N^{(2)}}(T_{m,f}) =  \cP_{ \N^{(2)}, \N^{(3)} } \circ  \cP_{ \N^{(1)}, \N^{(2)} } \circ T_{m,f}  
$$
$$
= \cP_{ \N^{(1)}, \N^{(3)} } \circ T_{m,f} = \ML_{\N^{(1)}, \N^{(3)}}(T_{m,f}),
$$
for every $[(m,f)] \in \GT^{\hs}(\N^{(1)})$.

Thus the desired identity \eqref{ML-OK} holds and the proposition is proved. 
\end{proof}
\bigskip

We call the functor\footnote{One of the authors of this paper is trying to live in the 
sequence of suburbs of Philadelphia called the Main Line. The functor $\ML$ is named after this beautiful sequence of suburbs.} 
$\ML$  the {\it Main Line functor}. 

In the next section, we will prove the following theorem:
\begin{thm}  
\label{thm:GTh-ML}
The (profinite version) $\GTh$ of the Grothendieck-Teichmueller group is isomorphic to 
$$
\lim(\ML). 
$$
\end{thm}  

\subsection{Proof of Theorem \ref{thm:GTh-ML}}
\label{sec:proof}

We will need the following auxiliary statements:
\begin{prop}  
\label{prop:cofinal}
~~~
\begin{itemize}
\item[{\bf A)}] For every $\N \in \NFI(\PB_3)$, there exists $\K \in  \NFI^{isolated}_{\PB_4}(\B_4)$ 
satisfying the property 
$$
\K_{\PB_3} \le \N.
$$

\item[{\bf B)}] For every $\N \in \NFI(\PB_2)$ there exists 
$\K \in \NFI^{isolated}_{\PB_4}(\B_4)$ such that $\K_{\PB_2} \le \N$. 
 
\end{itemize}

\end{prop}  
\begin{proof}
Let  $\N \in \NFI(\PB_3)$ and $\psi$ be a group homomorphism from $\PB_3$ to $S_n$ 
such that $\ker(\psi) = \N$. 

Using relations \eqref{PB-n-rel} on the generators of $\PB_4$, 
it is easy to show that the equations
$$
\ti{\psi}(x_{12}) : = \psi(x_{12}), \quad 
\ti{\psi}(x_{23}) : = \psi(x_{23}), \quad
\ti{\psi}(x_{13}) : = \psi(x_{13}),  
$$
$$
\ti{\psi}(x_{14}) = \ti{\psi}(x_{24}) = \ti{\psi}(x_{34}) : = \id_{S_n}
$$
define a group homomorphism $\ti{\psi} : \PB_4 \to S_n$. 

Moreover, the kernel of $\ti{\psi}$ satisfies the property 
$$
 \vf_{123}^{-1}(\ker(\ti{\psi})) = \N.
$$
Hence 
\begin{equation}
\label{N-ker-ti-psi}
\vf_{123}^{-1}(\ker(\ti{\psi})) \cap 
\vf_{12,3,4}^{-1}(\ker(\ti{\psi})) \cap
\vf_{1,23,4}^{-1}(\ker(\ti{\psi})) \cap
\vf_{1,2,34}^{-1}(\ker(\ti{\psi})) \cap 
\vf_{234}^{-1}(\ker(\ti{\psi})) ~\le~ \N.
\end{equation}

Let $\ti{\N}$ be the normal subgroup of $\PB_4$ obtained by intersecting 
all normal subgroups of $\PB_4$ of index $|\PB_4 : \ker(\ti{\psi})|$. 
Since $\ti{\N}$ is a characteristic subgroup of $\PB_4$ of finite index (in $\PB_4$), we have 
$$
\ti{\N} \in \NFI_{\PB_4}(\B_4). 
$$

Furthermore, due to Corollary \ref{cor:isolated-cofinal}, there exists 
an isolated element $\K$ of $\NFI_{\PB_4}(\B_4)$ satisfying the property 
$\K \le \ti{\N}$. Combining $\K \le \ti{\N}$ with $\ti{\N} \le \ker(\ti{\psi})$ and \eqref{N-ker-ti-psi}, 
we deduce that 
$$
\K_{\PB_3} \le \N. 
$$

Thus desired Statement {\bf A)} is proved. 

Just as for Statement {\bf A)}, we start with a group homomorphism 
$\ka : \PB_2 \to S_n$ whose kernel coincides with $\N$. 

It is easy to see that the equations
$$
\ti{\ka}(x_{12}) : = \ka(x_{12}), \quad 
\ti{\ka}(x_{23}) : =   \ka(x_{12})^{-1}\,, \quad
\ti{\ka}(x_{13}) : =  \id_{S_n}\,,
$$
$$
\ti{\ka}(x_{14}) = \ti{\ka}(x_{24})  = \ti{\ka}(x_{34}) : =  \id_{S_n}
$$
define a group homomorphism $\ti{\ka} : \PB_4 \to S_n$. 

The kernel of $\ti{\ka}$ satisfies the property 
\begin{equation}
\label{N-ker-ti-ka}
\vf_{12}^{-1}\big( \vf_{123}^{-1}(\ker(\ti{\ka})) \big) = \N.
\end{equation}

Let $\ti{\N}$ be the normal subgroup of $\PB_4$ obtained by intersecting 
all normal subgroups of $\PB_4$ of index $|\PB_4 : \ker(\ti{\ka})|$. 
Since $\ti{\N}$ is a characteristic subgroup of $\PB_4$ of finite index (in $\PB_4$), we have 
$$
\ti{\N} \in \NFI_{\PB_4}(\B_4). 
$$

As above, there exists an isolated element $\K$ of $\NFI_{\PB_4}(\B_4)$ satisfying the property 
$\K \le \ti{\N}$. Combining $\K \le \ti{\N}$ with $\ti{\N} \le \ker(\ti{\ka})$ and \eqref{N-ker-ti-ka}, 
we deduce that 
$$
\K_{\PB_2} \le \N. 
$$

Thus Statement  {\bf B)} is also proved.
\end{proof}

Proposition \ref{prop:cofinal} allows us to produce a more practical description of $\wh{\PaB}^{\le 4}$. 
To give this description, we note that the assignment $\K \mapsto \PaB^{\le 4}/ \K$ upgrades to 
a functor from the poset $\NFI^{isolated}_{\PB_4}(\B_4)$ to the category of truncated operads 
in finite groupoids. Indeed, for every pair $\K_1 \le \K_2$ of elements of $\NFI^{isolated}_{\PB_4}(\B_4)$
we have the obvious morphism of truncated operads 
$$
\cP_{\K_1, \K_2} :  \PaB^{\le 4} /\sim_{\K_1} \, \to \,  \PaB^{\le 4} /\sim_{\K_2}\,.
$$
Moreover, for every triple $\K_1 \le \K_2 \le \K_3$ of elements of $\NFI^{isolated}_{\PB_4}(\B_4)$, 
we have  $\cP_{\K_2, \K_3} \circ \cP_{\K_1, \K_2} = \cP_{\K_1, \K_3}$.
 
Let us denote by 
\begin{equation}
\label{wt-PaB}
\wt{\PaB}^{\le 4}
\end{equation}
the limit of this functor. 

More concretely, $\wt{\PaB}(n)$ consists of functions 
$$
\ga :  \NFI^{isolated}_{\PB_4}(\B_4) ~\to~ \bigsqcup_{\K \in \NFI^{isolated}_{\PB_4}(\B_4)} ~ \PaB(n)/ \sim_{\K}
$$
satisfying these two conditions: 
\begin{itemize}

\item for every $\K \in  \NFI^{isolated}_{\PB_4}(\B_4)$,  $\ga(\K) \in \PaB(n)/ \sim_{\K}$ and 

\item for every pair $\K_1 \le \K_2$ in  $\NFI^{isolated}_{\PB_4}(\B_4)$,
$\cP_{\K_1, \K_2} (\ga(\K_1)) = \ga(\K_2)$.

\end{itemize}

Since for every pair $\K_1 \le \K_2$ of elements of $\NFI^{isolated}_{\PB_4}(\B_4)$ we have 
$$
\cP_{\K_1, \K_2} \circ \hcP_{\K_1}  = \hcP_{\K_2},  
$$
the assignment 
$$
\Psi (\hat{\ga}) (\K) : = \hcP_{\K} (\hat{\ga}), \qquad \hat{\ga} \in \wh{\PaB}(n)
$$
defines a morphism of truncated operads
\begin{equation}
\label{Psi}
\Psi : \wh{\PaB}^{\le 4} \to \wt{\PaB}^{\le 4}\,.
\end{equation}

Let us prove that
\begin{cor}  
\label{cor:hat-tilde}
The morphism $\Psi$ in \eqref{Psi} is an isomorphism of truncated operads 
in the category of topological groupoids.  
\end{cor}  
\begin{proof} Since the compatibility with the structures of truncated 
operads and the composition of morphisms is obvious, it suffices to 
prove that $\Psi$ is a homeomorphism of topological spaces.

Let $\tau$, $\tau'$ be objects of $\PaB(n)$ and 
$\hat{\ga}_1, \hat{\ga}_2 \in \Hom_{\wh{\PaB}}(\tau, \tau')$ such that 
$\Psi(\hat{\ga}_1) = \Psi(\hat{\ga_2})$ or equivalently, for every 
$\K \in \NFI^{isolated}_{\PB_4}(\B_4)$
$$
\Psi \big(\hat{\ga}_2^{-1} \cdot \hat{\ga}_1 \big)(\K) 
$$
is the identity automorphism of $\tau$ in $\PaB(n)/\sim_{\K}$. 

Thus, due to Proposition \ref{prop:cofinal}, the image of  
$\hat{\ga}_2^{-1} \cdot \hat{\ga}_1$ in $\PB_n / \N$ is the identity element 
for every $\N \in \NFI(\PB_n)$.  Therefore $\hat{\ga}_2^{-1} \cdot \hat{\ga}_1$ is 
the identity element of $\wh{\PB}_n$ and hence 
$$
\hat{\ga}_1 = \hat{\ga}_2\,.
$$ 
We proved that $\Psi$ is one-to-one. 

Let $\ga \in  \wt{\PaB}(n)$, $\tau$ and $\tau'$ be the source and 
the target of $\ga$, respectively. Let $\la$ be any isomorphism 
from $\tau$ to $\tau'$ in $\PaB(n)$. By abuse of notation, we will use 
symbol $\la$ for its obvious image in $\wh{\PaB}(n)$ and in $\wt{\PaB}(n)$.
   
Due to Proposition \ref{prop:cofinal}, there exists an element 
$\hat{h} \in \wh{\PB}_n$ such that 
\begin{equation}
\label{Psi-onto}
\hcP_{\K}(\hat{h}) = (\la^{-1}\cdot \ga)(\K), \qquad \forall~~\K \in \NFI^{isolated}_{\PB_4}(\B_4).
\end{equation}

Equation \eqref{Psi-onto} implies that $\Psi(\la \cdot \hat{h}) = \ga$. 
Thus we proved that $\Psi$ is onto. 

Since, for every $\K \in \NFI^{isolated}_{\PB_4}(\B_4)$,
the composition of $\Psi$ with the canonical projection 
$$
\wt{\PaB}^{\le 4} \to \PaB^{\le 4}/\sim_{\K}
$$
coincides with the continuous map 
$$
\hcP_{\K} : \wh{\PaB}^{\le 4} \to \PaB^{\le 4}/\sim_{\K}\,,
$$
we conclude that $\Psi$ is continuous.

Since $\Psi$ is a continuous bijection from a compact space 
$\wh{\PaB}^{\le 4}$ to a Hausdorff space
$\wt{\PaB}^{\le 4}$\,, $\Psi$ is indeed a homeomorphism.
\end{proof}  

Due to Corollary \ref{cor:hat-tilde}, we can safely replace $\wh{\PaB}^{\le 4}$ by 
$\wt{\PaB}^{\le 4}$ in all further considerations. We will also use the same symbol
$\cI$ (resp. $\hcP_{\K}$ for $\K \in \NFI^{isolated}_{\PB_4}(\B_4)$) for the canonical 
embedding $\cI: \PaB^{\le 4} \to \wt{\PaB}^{\le 4}$ and the canonical projection 
$\hcP_{\K} :  \wt{\PaB}^{\le 4} \to  \PaB^{\le 4} /\sim_{\K}$.

\bigskip
\bigskip

Recall that, for every $\hat{T} \in \GTh$ and  $\K \in  \NFI^{isolated}_{\PB_4}(\B_4)$, the formula 
$T_{\K}  :=  \hcP_{\K} \circ \hat{T} \circ \cI$
defines an onto morphism of truncated operads $ \PaB^{\le 4} \to  \PaB^{\le 4}/ \K$.
Since $\K$ is an isolated element of  $\NFI_{\PB_4}(\B_4)$, Corollary \ref{cor:isom-quotients}
implies that the onto morphism $T_{\K}$ factors as follows: 
\begin{equation}
\label{T-K-isom}
T_{\K} =  T^{\isom}_{\K} \circ \cP_{\K}\,,
\end{equation}
where $T^{\isom}_{\K}$ is an isomorphism of truncated operads 
$T^{\isom}_{\K} :  \PaB^{\le 4}/\K  ~\stackrel{\cong}{\longrightarrow}~  \PaB^{\le 4}/ \K$
and $\cP_{\K}$ is the canonical projection $\PaB^{\le 4} \to  \PaB^{\le 4}/ \K$.


We claim that
\begin{prop}  
\label{prop:T-hat-T-K}
For every $\hat{T}\in \GTh$ and for every $\K \in  \NFI^{isolated}_{\PB_4}(\B_4)$ the diagram
\begin{equation}
\label{T-hat-T-K-isom}
\begin{tikzpicture}
\matrix (m) [matrix of math nodes, row sep=1.8em, column sep=2.1em]
{\wt{\PaB}^{\le 4}   & \wt{\PaB}^{\le 4} \\
 \PaB^{\le 4} / \sim_{\K} & \PaB^{\le 4} / \sim_{\K} \\ };
\path[->, font=\scriptsize]
(m-1-1) edge node[above] {$\hat{T}$} (m-1-2)
edge node[left] {$\hcP_{\K}$} (m-2-1) 
(m-2-1) edge node[above] {$T_{\K}^{\isom}$} (m-2-2)
(m-1-2) edge node[left] {$\hcP_{\K}$} (m-2-2); 
\end{tikzpicture}
\end{equation}
commutes.
\end{prop}  
\begin{proof}  
By definition of $T_{\K}^{\isom}$, 
$\hcP_{\K} \circ \hat{T} \circ \cI (\ga) = T_{\K}^{\isom} \circ \cP_{\K} (\ga), $
for every $\ga\in \PaB^{\le 4}$.

Hence 
\begin{equation}
\label{T-hat-T-K-isom-eq}
\hcP_{\K} \circ \hat{T} (\cI (\ga)) = T_{\K}^{\isom} \circ \hcP_{\K}  (\cI(\ga)), \qquad \forall~~ \ga\in \PaB^{\le 4}\,. 
\end{equation}

Since the image $\cI(\PaB^{\le 4})$ of $\PaB^{\le 4}$ in $\wt{\PaB}^{\le 4}$ is dense in $\wt{\PaB}^{\le 4}$
and the target $\PaB^{\le 4} / \sim_{\K}$ of the compositions $\hcP_{\K} \circ \hat{T}$ and 
$T_{\K}^{\isom} \circ \cP_{\K}$ is Hausdorff, identity \eqref{T-hat-T-K-isom-eq}
implies that diagram \eqref{T-hat-T-K-isom} indeed commutes.
\end{proof}

\bigskip

\begin{proof-of}[{\bf Theorem \ref{thm:GTh-ML}}] 
Let $\K, \ti{\K}$ be elements of $\NFI^{isolated}_{\PB_4}(\B_4)$ such that $\ti{\K} \le \K$ and 
$\cP_{\ti{\K}, \K}$ be the canonical projection from $\PaB^{\le 4} / \sim_{\ti{\K}}$ 
to $\PaB^{\le 4} / \sim_{\K}$.
Furthermore, let $T_{\K}$ and $T_{\ti{\K}}$ be onto morphisms from $\PaB^{\le 4}$ to 
$\PaB^{\le 4}/ \K$  and $\PaB^{\le 4}/ \ti{\K}$, respectively, coming from $\hat{T}\in \GTh$. 

Since $\hcP_{\K} =  \hcP_{\ti{\K}, \K}  \circ  \cP_{\ti{\K}}\,,$ the diagram 
$$
\begin{tikzpicture}
\matrix (m) [matrix of math nodes, row sep=1.8em, column sep=2.1em]
{\PaB^{\le 4}   & \PaB^{\le 4} / \sim_{\ti{\K}} \\
~ & \PaB^{\le 4} / \sim_{\K} \\ };
\path[->, font=\scriptsize]
(m-1-1) edge node[above] {$T_{\ti{\K}}$} (m-1-2) edge node[below] {$T_{\K}~$} (m-2-2)
(m-1-2) edge node[right] {$\cP_{\ti{\K}, \K}$} (m-2-2); 
\end{tikzpicture}
$$
commutes. Hence the assignment $\hat{T} \mapsto \{T_{\K}\}_{\K \in \NFI^{isolated}_{\PB_4}(\B_4)}$ 
gives us a map from $\GTh$ to  $\lim(\ML)$
\begin{equation}
\label{GTh-to-ML}
\GTh \to \lim(\ML). 
\end{equation}

Let us show that the map \eqref{GTh-to-ML} is a group homomorphism.

Indeed, let $\hat{T}^{(1)}, \hat{T}^{(2)} \in \GTh$, $\hat{T} : = \hat{T}^{(1)} \circ \hat{T}^{(2)}$ and  $\K \in \NFI^{isolated}_{\PB_4}(\B_4)$.
Using  Proposition \ref{prop:T-hat-T-K}, we get  
$$
\hcP_{\K} \circ  \hat{T} =  \hcP_{\K} \circ \hat{T}^{(1)} \circ \hat{T}^{(2)} =  T^{(1),\, \isom}_{\K} \circ  \hcP_{\K} \circ  \hat{T}^{(2)} = 
T^{(1),\, \isom}_{\K} \circ  T^{(2),\, \isom}_{\K} \circ \hcP_{\K}. 
$$ 
On the other hand, $\hcP_{\K} \circ  \hat{T} =   T^{\isom}_{\K} \circ \hcP_{\K}$ and hence 
\begin{equation}
\label{T-T12-isom}
T^{\isom}_{\K} \circ \hcP_{\K} = T^{(1),\, \isom}_{\K} \circ  T^{(2),\, \isom}_{\K} \circ \hcP_{\K}\,.
\end{equation}

Since $\hcP_{\K} : \wt{\PaB}^{\le 4}  \to  \PaB^{\le 4} / \sim_{\K}$ is onto, identity \eqref{T-T12-isom} implies that 
$$
T^{\isom}_{\K} = T^{(1),\, \isom}_{\K} \circ  T^{(2),\, \isom}_{\K}\,.
$$  
Thus the map \eqref{GTh-to-ML} is indeed a group homomorphism.  

\bigskip
\bigskip

Our next goal is to show that homomorphism \eqref{GTh-to-ML} is one-to-one and onto.

To prove that \eqref{GTh-to-ML} is one-to-one, we consider $\hat{T} \in \GTh$ 
such that $T_{\K}$ coincides with the canonical projection 
$$
\PaB^{\le 4}   ~\to~ \PaB^{\le 4} / \sim_{\K}
$$
for every $\K \in  \NFI^{isolated}_{\PB_4}(\B_4)$.  

Hence, for every $\ga \in \PaB^{\le 4}$, we have 
$$
\hcP_{\K} \circ \hat{T}\big( \cI(\ga) \big) = \hcP_{\K} \circ \cI(\ga)\,, 
\qquad \forall~~\K \in \NFI^{isolated}_{\PB_4}(\B_4).
$$

This means that the restriction of $\hat{T}$ to the subset  $\cI(\PaB^{\le 4}) \subset  \wt{\PaB}^{\le 4}$
coincides with the restriction of the identity map $\id :  \wt{\PaB}^{\le 4} \to \wt{\PaB}^{\le 4}$
to the subset  $\cI(\PaB^{\le 4})$.
Since the subset $\cI(\PaB^{\le 4})$ is dense in $\wt{\PaB}^{\le 4}$ and the 
space $\wt{\PaB}^{\le 4}$ is Hausdorff, we conclude that $\hat{T}$ is the identity 
map $\id :  \wt{\PaB}^{\le 4} \to \wt{\PaB}^{\le 4}$. 
Thus the injectivity of \eqref{GTh-to-ML} is established. 

\bigskip

Note that an element of $\lim(\ML)$ is a family 
$\{ \cT^{\isom}_{\K} \}_{\K \in  \NFI^{isolated}_{\PB_4}(\B_4)}$ of isomorphisms of truncated operads 
$$
\cT^{\isom}_{\K} : \PaB^{\le 4} / \sim_{\K}  \stackrel{\cong}{\longrightarrow} \PaB^{\le 4} / \sim_{\K} 
$$
satisfying the following property: for every pair $\K \le \ti{\K}$ in $\NFI^{isolated}_{\PB_4}(\B_4)$, 
the diagram 
\begin{equation}
\label{K-ti-K}
\begin{tikzpicture}
\matrix (m) [matrix of math nodes, row sep=2.1em, column sep=2.1em]
{ \PaB^{\le 4} / \sim_{\K} &   \PaB^{\le 4} / \sim_{\K}  \\
   \PaB^{\le 4} / \sim_{\ti{\K}}  &   \PaB^{\le 4} / \sim_{\ti{\K}}  \\ };
\path[->, font=\scriptsize]
(m-1-1) edge node[above] {$\cT^{\isom}_{\K}$} (m-1-2) 
edge node[left] {$\cP_{\K, \ti{\K}}$} (m-2-1) 
(m-2-1) edge node[above] {$\cT^{\isom}_{\ti{\K}}$} (m-2-2) 
(m-1-2) edge node[right] {$\cP_{\K, \ti{\K}}$} (m-2-2) ;
\end{tikzpicture}
\end{equation}
commutes. 

Due to commutativity of \eqref{K-ti-K}, the formula 
\begin{equation}
\label{hat-T}
\hat{T} (\ga)(\K) : = \cT^{\isom}_{\K}(\ga(\K)) 
\end{equation}
defines a morphism of truncated operads in groupoids 
$\hat{T} : \wt{\PaB}^{\le 4} \to \wt{\PaB}^{\le 4}$.

To prove that $\hat{T}$ is continuous, we need to show that the composition 
$$
\hcP_{\K} \circ \hat{T} :  \wt{\PaB}^{\le 4} \to  \PaB^{\le 4} / \sim_{\K} 
$$
is continuous for every $\K \in \NFI^{isolated}_{\PB_4}(\B_4)$. 
 
By definition of $\hat{T}$ \eqref{hat-T}, 
\begin{equation}
\label{hcPT-T-hat}
\hcP_{\K} \circ \hat{T} =  \cT^{\isom}_{\K} \circ \hcP_{\K}
\end{equation}
for every $\K \in \NFI^{isolated}_{\PB_4}(\B_4)$. 

Since $\cT^{\isom}_{\K}$ is an automorphism of the (finite) groupoid $\PaB^{\le 4} / \sim_{\K}$
equipped with the discrete topology and  $\hcP_{\K}$ is continuous, identity \eqref{hcPT-T-hat}
implies that the composition $\hcP_{\K} \circ \hat{T}$ is indeed continuous. 

Thus equation \eqref{hat-T} defines a continuous endomorphism of the operad $\wt{\PaB}^{\le 4}$. 

To find the inverse of $\hat{T}$, we denote by $\cS^{\isom}_{\K}$ the inverse of $\cT^{\isom}_{\K}$ 
for every $\K \in \NFI^{isolated}_{\PB_4}(\B_4)$. Then it is easy to see that the formula 
$$
\hat{S} (\ga)(\K) : = \cS^{\isom}_{\K}(\ga(\K)) 
$$
defines the inverse of $\hat{T}$. 

The proof of surjectivity of \eqref{GTh-to-ML} is complete. 
\end{proof-of}

\bigskip

Let us consider $\K, \N \in \NFI_{\PB_4}(\B_4)$ with $\K \le \N$ and a pair $(m,f) \in \bbZ \times \F_2$
that represents a $\GT$-shadow in $\GT^{\hs}(\K)$. Clearly, the same pair $(m,f)$ also 
represents a $\GT$-shadow in $\GT^{\hs}(\N)$. In other words, if $\K \le \N$, then we have a natural map 
\begin{equation}
\label{K-to-N}
\GT^{\hs}(\K) \to \GT^{\hs}(\N). 
\end{equation}
It makes sense to consider this map even if neither $\K$ nor $\N$ are isolated.

\begin{defi}  
\label{dfn:survives}
We say that a $\GT$-shadow $[(m,f)] \in \GT^{\hs}(\N)$ \emph{survives into} $\K$ if 
$[(m,f)]$ belongs to the image of the map \eqref{K-to-N}. In other words, there exists 
$(m_1, f_1) \in \bbZ \times \F_2$ such that $[(m_1, f_1)] \in \GT^{\hs}(\K)$, 
$m_1 \cong m \mod N_{\ord}$ and $f_1 \N_{\F_2} = f \N_{\F_2}$. 
\end{defi}  

The following statement is a straightforward consequence of Proposition \ref{prop:isolated-practice}
and Theorem \ref{thm:GTh-ML}:
\begin{cor}  
\label{cor:survives}
Let $\N \in \NFI_{\PB_4}(\B_4)$ and $[(m,f)] \in \GT^{\hs}(\N)$. 
The $\GT$-shadow $[(m,f)]$ is genuine if and only if $[(m,f)]$ survives 
into $\K$ for every $\K \in \NFI_{\PB_4}(\B_4)$ such that $\K \le \N$. \qed
\end{cor}

\section{Selected results of computer experiments}
\label{sec:comp-exp}

In the computer implementation \cite{package-GT}, an element $\N$ of $\NFI_{\PB_4}(\B_4)$ 
is represented by a group homomorphism $\psi$ from $\PB_4$ to a symmetric group
such that $\N = \ker(\psi)$. Each homomorphism $\psi : \PB_4 \to S_d$ is, in turn, 
represented by a tuple of permutations
\begin{equation}
\label{g6}
(g_{12}, g_{23}, g_{13}, g_{14}, g_{24}, g_{34}) \in (S_d)^6
\end{equation}
satisfying the relations of $\PB_4$ (see \eqref{PB-n-rel}). 

It should be mentioned that, in \cite{package-GT}, we consider only practical 
$\GT$-shadows (see Remark \ref{rem:practical}). In particular, throughout this section, 
$\GT(\N)$ denotes the set of practical $\GT$-shadows with the target $\N$. Clearly, every 
charming $\GT$-shadow is practical.  

Table \ref{table:E-GT} presents basic information about $35$ selected elements 
\begin{equation}
\label{35-cool-guys}
\N^{(0)},~ \N^{(1)},~ \dots,~ \N^{(34)} ~\in~\NFI_{\PB_4}(\B_4).
\end{equation}
For every $\N^{(i)}$ in this list, the quotient $\F_2/\N^{(i)}_{\F_2}$ 
is non-Abelian. Table \ref{table:E-GT} also shows $N^{(i)}_{\ord} := |\PB_2:\N^{(i)}_{\PB_2}|$, 
the size of $\GT(\N^{(i)})$ (i.e. the total number of practical $\GT$-shadows
with the target $\N^{(i)}$) and the size of $\GT^{\hs}(\N^{(i)})$. 
The last column indicates whether $\N^{(i)}$ is isolated or not.

For every non-isolated element $\N$ in the list \eqref{35-cool-guys}, the connected 
component  $\GTSh^{\hs}_{\conn}(\N)$ has exactly two objects. More precisely,
\begin{itemize}

\item $\N^{(4)}$ is a conjugate of $\N^{(3)}$ and $\N^{(3)} \cap \N^{(4)} = \N^{(14)}$;

\item $\N^{(11)}$ is a conjugate of $\N^{(10)}$ and $\N^{(10)} \cap \N^{(11)} = \N^{(24)}$; 

\item $\N^{(17)}$ is a conjugate of $\N^{(16)}$ and $\N^{(16)} \cap \N^{(17)} = \N^{(30)}$;

\item  $\N^{(27)}$ is a conjugate of $\N^{(26)}$ and $\N^{(26)} \cap \N^{(27)} = \N^{(34)}$.

\end{itemize}

For $\N^{(31)}$, $\GT(\N^{(31)})$ has $588$ elements. To find the size of $\GT(\N^{(31)})$,
the computer had to look at $\approx 9 \cdot 10^6$ elements of the group $\F_2/\N^{(31)}_{\F_2}$. 
For the iMac with the processor 3.4 GHz, Intel Core i5, it took over 9 full days to complete this task.  

For $\N^{(32)}$, $\GT(\N^{(32)})$ has $800$ elements. To find the size of $\GT(\N^{(32)})$,
the computer had to look at over $9 \cdot 10^6$ elements of the group $\F_2/\N^{(32)}_{\F_2}$.
For the iMac with the processor 3.4 GHz, Intel Core i5, it took almost 10 full days to complete 
this task.

\begin{table}[h!]
\centering
{ \small
\begin{tabular}{|c|c|c|c|c|c|c|c|}
\hline
~& ~& ~& ~& ~& ~& ~ &~ \\
$i$ &  $|\PB_4: \N^{(i)}|$ &  $|\F_2: \N^{(i)}_{\F_2}|$ & 
$|[ \F_2/\N^{(i)}_{\F_2}, \F_2/\N^{(i)}_{\F_2}] | $ &  $ N^{(i)}_{\ord}$ &
$|\GT(\N^{(i)})|$ &  $|\GT^{\hs}(\N^{(i)})|$ & isolated? \\[0.3cm]  \hline
0 & 8 & 16 & 2 & 4 & 4 & 4 & True \\ \hline
1 & 8 & 16 & 2 &  4 &  8 & 4 & True    \\ \hline 
2 & 12 & 36 & 4  &  3 &18 & 6 & True  \\ \hline
3 & 21 & 63 & 7 &  3 &  36 & 12 & False  \\ \hline
4 & 21 & 63 & 7 &   3 & 36 & 12 & False   \\ \hline
5 & 24 & 288 & 8 &  6 & 72 & 12 & True  \\ \hline
6 & 24 & 144 & 4 &  6 & 72 & 12 & True  \\ \hline
7 & 48  & 144 & 4 &  6 & 72 & 12 & True   \\ \hline
8 & 60 & 1500 & 60 &  5 & 100 & 20 & True  \\ \hline
9 & 60 & 900 & 4  &  15 & 360 & 24 & True  \\ \hline
10 & 72 & 144 & 18  & 4 & 16 & 8 & False \\ \hline
11 & 72 &144 & 18  &  4 & 16 & 8 & False   \\ \hline
12 & 108 & 972 & 27 &  6 & 72 & 12 & True \\ \hline
13 & 120 & 6000 & 60 &  10 & 400 & 40 & True   \\ \hline
14 & 147 & 441 & 49  &  3 & 216 & 72 & True   \\ \hline
15 & 168 & 8232 & 168 &  7 & 294 & 42 & True  \\ \hline
16 & 168 & 1344 & 168  &   4 & 64 & 32 & False  \\ \hline
17 & 168 & 1344 & 168  &  4 &  64 & 32 & False  \\ \hline
18 & 180 & 13500 & 60  &  15 &600 & 40 & True  \\ \hline
19 & 216 & 7776 & 216  &  6 & 72 & 12 & True  \\ \hline
20 & 240 & 6000 & 60   &  10 & 400 & 40 & True  \\ \hline
21 & 324 & 8748 & 108 &  9 & 486 & 54 & True  \\ \hline
22 & 504 & 40824 & 504  &  9 &486 & 54 & True \\ \hline
23 & 504 & 24696 & 504 &  7 &294 & 42 & True   \\ \hline
24 & 648 & 1296 & 162  &  4  & 32 & 16 & True  \\ \hline
25 & 720 & 54000 & 240 &  15 & 1800 & 120 & True \\ \hline
26 & 1512 & 40824 & 504 & 9 & 486 & 54 & False  \\ \hline
27 & 1512 & 40824 & 504 &  9 & 486 & 54 & False   \\ \hline
28 & 2520 & 63000 & 2520 & 5 & 200 & 40 & True  \\ \hline
29 & 2520 & 45360 & 2520  &  6 &144 & 48 & True   \\ \hline
30 & 28224 & 225792 & 28224  & 4 & 512 & 256 & True \\ \hline
31 & 181440 & 8890560 & 181440 & 7 & {\bf 588} & 84 & True  \\ \hline
32 & 181440 & 9072000 & 181440  &  10 & {\bf 800} & 160 & True  \\ \hline
33 & 181440 & 40824000 & 181440  & 15 & {\bf $\ge$ 1800} & 120 & True  \\ \hline
34 & 762048 & 20575296 & 254016 &  9 & {\bf $\ge$ 4374} & 486 & True \\ \hline
\end{tabular}
}
\caption{The basic information about selected $35$ compatible equivalence relations}
\label{table:E-GT}
\end{table}

\begin{remark}  
\label{rem:isolated}
Recall that the definition of an isolated element of $\NFI_{\PB_4}(\B_4)$ 
(see Definition \ref{dfn:settled}) is based on charming $\GT$-shadows. In principal, 
it is possible that there exists an isolated element $\N \in \NFI_{\PB_4}(\B_4)$
for which $\GT(\N)$ has a non-settled element. We \emph{did not} encounter such 
examples in our experiments.  
\end{remark}

\subsection{Selected remarkable examples}
\label{sec:selected}
For the $19$-th example $\N^{(19)}$ in table \ref{table:E-GT}, the quotient $\F_2 / \N^{(19)}_{\F_2}$
has order $7776=2^5 \cdot 3^5$. Due to the similarity between this order and the historic year 1776, 
we decided to call the subgroup $\N^{(19)}$ the {\it Philadelphia subgroup} of $\PB_4$.
This subgroup is the kernel of the homomorphism from $\PB_4$ to $S_9$ that sends
the standard generators of $\PB_4$ to the permutations
\begin{equation}
\label{Philly_subgroup}
\begin{array}{ccc}
g_{12}  := (1, 3, 2)(4, 6, 5), & g_{23} := (1, 4, 9)(2, 7, 6), & g_{13} := (1, 7, 5)(3, 6, 9), \\
g_{14} := (2, 6, 7)(3, 8, 5), & g_{24} := (1, 8, 6)(3, 4, 7), & g_{34} := (1, 2, 3)(7, 9, 8),
\end{array}
\end{equation}
respectively.

Since $\N^{(19)}$ is isolated, $\GT^{\hs}(\N^{(19)})$ is a group. We showed that $\GT^{\hs}(\N^{(19)})$ is isomorphic 
to the dihedral group $D_6=\lan r,s ~|~ r^6, s^2, rsrs \ran$ of order $12$. We also showed that the kernel of the restriction of
the virtual cyclotomic character to $\GT^{\hs}(\N^{(19)})$ coincides with the cyclic subgroup $\lan r \ran$ of 
order $6$. 

The last element $\N^{(34)}$ in \eqref{35-cool-guys} has the biggest index 
$762,048 = 2^ 6 \cdot 3^5 \cdot 7^2$ in $\PB_4$. This subgroup is the kernel of the homomorphism 
from $\PB_4$ to $S_{18}$ that sends the standard generators of $\PB_4$ to 
\begin{equation}
\label{m_dandy}
\begin{array}{c}
g_{12} := (1, 3, 5, 7, 9, 2, 4, 6, 8)(10, 12, 14, 16, 18, 11, 13, 15, 17), \\
g_{23} := (1, 3, 7, 8, 2, 4, 9, 6, 5) (10, 15, 17, 11, 12, 16, 18, 14, 13)\\
g_{13} := (1, 3, 8, 5, 4, 9, 2, 6, 7) (10, 11, 15, 17, 13, 12, 18, 14, 16) \\
g_{14} := (1, 3, 7, 8, 2, 4, 9, 6, 5) (10, 15, 17, 11, 12, 16, 18, 14, 13)\\
g_{24} :=(1, 7, 6, 2, 4, 8, 9, 3, 5) (10, 15, 14, 11, 16, 18, 12, 13, 17)\\
g_{34} := (1, 3, 5, 7, 9, 2, 4, 6, 8) (10, 12, 14, 16, 18, 11, 13, 15, 17)
\end{array}
\end{equation}
respectively. We call this subgroup the \emph{Mighty Dandy}. 
 
Due to Proposition \ref{prop:isolated-practice}, the Mighty Dandy is an isolated element and 
hence $\GT^{\hs}(\N^{(34)})$ is a group. This is what we showed about this group:
\begin{itemize}
\item  $\GT^{\hs}(\N^{(34)})$ has order $486=2 \cdot 3^5$; 

\item the kernel $Ker_{34}$ of the restriction of the virtual cyclotomic character 
to $\GT^{\hs}(\N^{(34)})$ is an Abelian subgroup of order $81=3^4$;
in fact, $Ker_{34}$ is isomorphic to $\cZ_9 \times \cZ_9$;

\item  $\GT^{\hs}(\N^{(34)})$ is isomorphic to the semi-direct product
$$
(\cZ_2 \times \cZ_3) \ltimes (\cZ_9 \times \cZ_9);
$$

\item the Sylow $3$-subgroup $Syl$ of $\GT^{\hs}(\N^{(34)})$ is a non-Abelian
group of order $3^5 = 243$; $Syl$ is a normal subgroup of  $\GT^{\hs}(\N^{(34)})$
and it is isomorphic to the semi-direct product 
$$
\cZ_3 \ltimes (\cZ_9 \times \cZ_9).
$$ 
\end{itemize}

Although every element $\N$ in the list \eqref{35-cool-guys} has the property
$|\F_2: \N_{\F_2}| > |\PB_4 : \N| $, there are examples $\N \in  \NFI_{\PB_4}(\B_4)$
for which $|\PB_4 : \N|$ is significantly bigger than the index 
$|\F_2: \N_{\F_2}|$. 

One such example was suggested to us by Leila Schneps. 
\emph{Leila's subgroup} $\N^{\cL}$ of $\PB_4$ is the kernel of a homomorphism 
from $\PB_4$ to $S_{130}$ and it can be retrieved from one of the storage files 
in \cite{package-GT}. Here is what we know about $\N^{\cL}$: 
\begin{itemize}

\item the index of $\N^{\cL}$ in $\PB_4$ is $2^{29} \cdot 3^{12} = 285315214344192$;

\item the index of $\N^{\cL}_{\PB_3}$ in $\PB_3$ is $ 2^{12} \cdot 3^6 = 2985984$;

\item the index of $\N^{\cL}_{\F_2}$ in $\F_2$ is $ 2^{10} \cdot 3^5 = 248832$; 

\item $N^{\cL}_{\ord} = 12$;

\item the order of the commutator subgroup of $\F_2/\N^{\cL}_{\F_2}$ is $2^6 \cdot 3^3 = 1728$;
 
\item there are only $48= 2^4 \cdot 3$ charming $\GT$-shadows for $\N^{\cL}$;  

\item $\N^{\cL}$ is an \emph{isolated} element of $\NFI_{\PB_4}(\B_4)$ and hence 
$\GT^{\hs}(\N^{\cL})$ is a group. 

\end{itemize}
 
We found that the group $\GT^{\hs}(\N^{\cL})$ is isomorphic to the semi-direct product 
\begin{equation}
\label{Leila-group}
\cZ_2 \ltimes (\cZ_2 \times \cZ_2 \times \cZ_2 \times \cZ_3),
\end{equation}
where the non-trivial element of $\cZ_2$ acts on 
\begin{equation}
\label{Abelian-part}
\cZ_2 \times \cZ_2 \times \cZ_2 \times \cZ_3 = \lan a | a^2 \ran \times  \lan b | b^2 \ran
\times  \lan c | c^2 \ran \times  \lan d | d^3 \ran
\end{equation}
by the automorphism 
$$
a \mapsto b, \qquad b \mapsto a, \qquad c \mapsto c, \qquad d \mapsto d^{-1}\,.
$$ 

The restriction of the virtual cyclotomic character to  $\GT^{\hs}(\N^{\cL})$ gives 
us the group homomorphism
$$
\GT^{\hs}(\N^{\cL}) \to \big(\bbZ/12 \bbZ \big)^{\times}
$$
and the kernel of this homomorphism is the subgroup of \eqref{Abelian-part}
generated by $ab$, $c$ and $d$. 

\subsection{Is there a charming $\GT$-shadow that is also fake?}
\label{sec:charming_fake}

Table \ref{table:E-GT} shows that the set $\GT^{\hs}(\N)$ of charming $\GT$-shadows  
corresponding to a given $\N \in  \NFI_{\PB_4}(\B_4)$ is typically a {\it proper subset } of $\GT(\N)$. 
For example, for the Philadelphia subgroup $\N^{(19)}$, we have $72$ $\GT$-shadows 
and only 12 of them are charming.   

Due to Proposition \ref{prop:gen-charm}, every non-charming $\GT$-shadow is fake.  
Thus, for a typical $\N$ from our list of $35$ elements of $\NFI_{\PB_4}(\B_4)$, 
we have many examples of a fake $\GT$-shadows. 
For instance, $\GT(\N^{(19)})$ contains at least $60$ fake $\GT$-shadows. 

It is more challenging to find examples of charming $\GT$-shadows that are fake. 
At the time of writing, we did \emph{not} find a single example of a charming $\GT$-shadow
that is also fake. 

Here is what we did. In the list \eqref{35-cool-guys}, there are exactly 24 pairs $(\N^{(i)}, \N^{(j)})$
with $i \neq j$ such that 
$$
\N^{(j)} \le \N^{(i)}\,.
$$
For each such pair, we showed that every $\GT$-shadow in $\GT^{\hs}(\N^{(i)})$ survives 
into $\N^{(j)}$, i.e. the natural map  $\GT^{\hs}(\N^{(j)}) \to \GT^{\hs}(\N^{(i)})$ is onto.  
We also looked at other selected examples of elements $\K \le \N$ in  $\NFI_{\PB_4}(\B_4)$
in which $\N$ belongs to the list \eqref{35-cool-guys} and $\K$ is obtained by intersecting 
$\N$ with another element of \eqref{35-cool-guys}. In all examples we have considered so far, the natural 
map $\GT^{\hs}(\K) \to \GT^{\hs}(\N)$ is onto.  

\subsection{Versions of the Furusho property and selected open questions}
\label{sec:Furusho-open}

Two versions of the Furusho property are motivated by a remarkable theorem which 
says roughly that, in the prounipotent setting, the pentagon relation implies the hexagon 
relations. For a precise statement, we 
refer the reader to \cite[Theorem 3.1]{BN-Furusho} and \cite[Theorem 1]{Furusho}. 

We say that an element $\N \in \NFI_{\PB_4}(\B_4)$ satisfies {\it the strong Furusho property} if
\begin{pty}  
\label{P:Furusho-strong}
For every $f \N_{\F_2} \in \F_2/\N_{\F_2}$ satisfying
pentagon relation \eqref{GT-penta} modulo $\N$,  
there exists $m \in \bbZ$ such that 
\begin{itemize}

\item $2m+1$ represents a unit in $\bbZ/N_{\ord} \bbZ$ and 

\item the pair $(m,f)$ satisfies hexagon relations \eqref{hexa1}, \eqref{hexa11}. 

\end{itemize}
\end{pty}   

Furthermore, we say that an element $\N \in \NFI_{\PB_4}(\B_4)$ satisfies 
{\it the weak Furusho property} if
\begin{pty}  
\label{P:Furusho-weak}
For every $f \N_{\F_2} ~\in~ [\F_2/\N_{\F_2}, \F_2/\N_{\F_2}]$ satisfying  
pentagon relation \eqref{GT-penta} modulo $\N$,  
there exists $m \in \bbZ$ such that 
\begin{itemize}

\item $2m+1$ represents a unit in $\bbZ/N_{\ord} \bbZ$ and 

\item the pair $(m,f)$ satisfies hexagon relations \eqref{hexa1}, \eqref{hexa11}. 

\end{itemize}
\end{pty}   

Using \cite{package-GT}, we showed that the following $11$ elements of 
the list \eqref{35-cool-guys}
\begin{equation}
\label{Furusho-strong-list}
\N^{(1)}, ~\N^{(2)},~ \N^{(3)},~ \N^{(4)},~ \N^{(6)},~ \N^{(7)},~ \N^{(9)},~ \N^{(10)}, ~
\N^{(11)},~ \N^{(14)},~ \N^{(24)}
\end{equation}
satisfy Property \ref{P:Furusho-strong} and the remaining $24$ elements of \eqref{35-cool-guys}
do \emph{not} satisfy Property \ref{P:Furusho-strong}. 

For instance, for the Philadelphia subgroup $\N^{(19)}$, $N^{(19)}_{\ord}=6$ and there 
are $216$ elements $f \N^{(19)}_{\F_2}$ in 
$\F_2 /\N^{(19)}_{\F_2}$ that satisfy the pentagon relation modulo $\N^{(19)}$. 
However, for only $36$ of these $216$ elements, 
there exists $m \in \{ 0,1,\dots, 5\}$ such that $2m+1$ represents 
a unit in $\bbZ/ 6 \bbZ$ and the pair $(m,f)$ satisfies hexagon 
relations  \eqref{hexa1}, \eqref{hexa11} (modulo $\N^{(19)}_{\PB_3}$).

Using \cite{package-GT}, we also showed that the following $13$ elements of 
the list  \eqref{35-cool-guys}
\begin{equation}
\label{Furusho-weak-list}
\N^{(0)}, ~\N^{(1)},~ \N^{(2)},~ \N^{(3)},~ \N^{(4)},~ \N^{(5)},~ \N^{(6)},~ \N^{(7)},~ \N^{(9)},~ \N^{(10)},~
\N^{(11)},~ \N^{(14)},~ \N^{(24)}
\end{equation}
satisfy Property \ref{P:Furusho-weak} and the remaining $22$ elements of \eqref{35-cool-guys}
do \emph{not} satisfy Property \ref{P:Furusho-weak}. 

For instance, for the Mighty Dandy $\N^{(34)}$, $N^{(34)}_{\ord}=9$ and there 
are $4096$ elements\footnote{For the iMac with the processor 3.4 GHz, Intel Core i5, it 
took over 52 hours to find all these elements.} 
in $[\F_2 /\N^{(34)}_{\F_2}, \F_2 /\N^{(34)}_{\F_2}]$ that satisfy 
the pentagon relation modulo $\N^{(34)}$. However, for only $243$ of them, 
there exists $m \in \{ 0,1,\dots, 8\}$ such that $2m+1$ represents 
a unit in $\bbZ/ 9 \bbZ$ and the pair $(m,f)$ satisfies hexagon 
relations  \eqref{hexa1}, \eqref{hexa11} (modulo $\N^{(34)}_{\PB_3}$).

\bigskip

We conclude this section with selected open questions. 
Most of these questions are motivated by our experiments \cite{package-GT}.
\begin{quest}  
\label{quest:onto}
Let $\N \in \NFI_{\PB_4}(\B_4)$ and $(m,f) \in \bbZ \times \F_2$ be a pair satisfying 
\eqref{hexa1}, \eqref{hexa11}, \eqref{GT-penta} (relative to $\sim_{\N}$). 
Recall that, due Proposition \ref{prop:onto}, if the group homomorphisms 
$T^{\PB_2}_{m,f}$ and  $T^{\PB_3}_{m,f}$ are onto then so is the group 
homomorphism 
$$
T^{\PB_4}_{m,f}: \PB_4 \to \PB_4/\N.
$$
Using \cite{package-GT}, the authors could not find an example of a pair
$(m,f) \in \bbZ \times \F_2$ for which $T^{\PB_4}_{m,f}$ is onto but 
$T^{\PB_2}_{m,f}$ is not onto or $T^{\PB_3}_{m,f}$ is not onto.
Can one prove that, if  $T^{\PB_4}_{m,f}$ is onto, then so are 
the group homomorphisms $T^{\PB_2}_{m,f}$ and $T^{\PB_3}_{m,f}$?
\end{quest}
\begin{quest}  
\label{quest:more-conjugates}
Is it possible to find an example of a non-isolated $\N \in \NFI_{\PB_4}(\B_4)$ for which the connected 
component $\GTSh^{\hs}_{\conn}(\N)$ has more than $2$ objects? In other words, 
is it possible to find $\N \in \NFI_{\PB_4}(\B_4)$ that has $> 2$ distinct conjugates?
\end{quest}
\begin{quest}  
\label{quest:charming-fake}
Is it possible to find $\K, \N \in  \NFI_{\PB_4}(\B_4)$ such that 
$\K \le \N$ and the natural map 
$$
\GT^{\hs}(\K) \to \GT^{\hs}(\N)
$$
is \emph{not} onto? In other words, can one produce an example 
of a charming $\GT$-shadow that is also fake? 
\end{quest}  
\begin{quest}  
\label{quest:genuine}
Is it possible to find $\N \in  \NFI_{\PB_4}(\B_4)$ for which
$\F_2/\N_{\F_2}$ is \emph{non-Abelian} and we can identify all genuine $\GT$-shadows 
in the set $\GT^{\hs}(\N)$?
\end{quest}  
Note that, if $\F_2/\N_{\F_2}$ is Abelian, all charming $\GT$-shadows
can be described completely and they are {\it all genuine}. 
(See Theorem \ref{thm:Abelian} in Appendix \ref{app:Abelian}.)

\appendix

\section{The operad $\PaB$ and its profinite completion}
\label{app:PaB}

The operad $\PaB$ of parenthesized braids is an operad in the category of groupoids 
and it was introduced\footnote{A very similar construction appeared in beautiful paper \cite{BNGT} by D. Bar-Natan.} 
by D. Tamarkin in \cite{Tamarkin}.

In this appendix, we give a brief reminder of the operad $\PaB$ and its profinite completion. 
For a more detailed exposition, we refer the reader to \cite[Chapter 6]{Fresse1}.

\subsection{The groups $\B_n$ and $\PB_n$}
\label{app:braids}

The Artin braid group $\B_n$ on $n$ strands is, by definition, the fundamental group 
of the orbifold 
$$
\Conf(n,\bbC)/S_n\,,
$$
where $\Conf(n, \bbC)$ denotes the configuration space of $n$ (labeled) points on $\bbC$: 
$\Conf(n, \bbC) : = \{ (z_1, \dots, z_n) \in \bbC^n ~|~ z_i \neq z_j \textrm{ if } i \neq j\}$.   

It is known \cite[Chapter 1]{Braids} that $\B_n$ has the following presentation 
$$
\big\langle\, \si_1, \si_2, \dots, \si_{n-1} ~|~ 
$$
\begin{equation}
\label{Bn}
\si_i \si_j \si^{-1}_i \si^{-1}_j~~ \textrm{if} ~|i-j| \ge 2, 
~~ \si_i \si_{i+1} \si_i   \si^{-1}_{i+1} \si^{-1}_{i} \si^{-1}_{i+1} ~~\textrm{for}~~ 1\le i \le n-2 \,\big\rangle,
\end{equation}
where $\si_i$ is the element depicted in figure \ref{fig:si-i}.
\begin{figure}[htp] 
\centering 
\begin{tikzpicture}[scale=1.5, > = stealth]
\tikzstyle{v} = [circle, draw, fill, minimum size=0, inner sep=1]
\node[v] (v1) at (-2, 0) {};
\draw (-2,-0.2) node[anchor=center] {{\small $1$}};
\node[v] (vv1) at (-2, 1) {};
\draw [->] (v1) -- (vv1);
\draw (-1.2,0.5) node[anchor=center] {{$\dots$}};
\node[v] (v1i) at (-0.5, 0) {};
\draw (-0.5,-0.2) node[anchor=center] {{\small $i-1$}};
\node[v] (vv1i) at (-0.5, 1) {};
\draw [->] (v1i) -- (vv1i);
\node[v] (vi) at (0, 0) {};
\draw (0,-0.2) node[anchor=center] {{\small $i$}};
\node[v] (vi1) at (1, 0) {};
\draw (1,-0.2) node[anchor=center] {{\small $i+1$}};
\node[v] (vvi) at (0, 1) {};
\node[v] (vvi1) at (1, 1) {};
\draw [->] (vi) -- (vvi1);
\draw (vi1) -- (0.6, 0.4); 
\draw [->] (0.4, 0.6) -- (vvi);
\node[v] (vi2) at (1.7, 0) {};
\draw (1.7,-0.2) node[anchor=center] {{\small $i+2$}};
\node[v] (vvi2) at (1.7, 1) {};
\draw [->] (vi2) -- (vvi2);
\draw (2.4,0.5) node[anchor=center] {{$\dots$}};
\node[v] (vn) at (3, 0) {};
\draw (3,-0.2) node[anchor=center] {{\small $n$}};
\node[v] (vvn) at (3, 1) {};
\draw [->] (vn) -- (vvn);
\end{tikzpicture}
\caption{The generator $\si_i$} \label{fig:si-i}
\end{figure}

Recall that the \emph{pure braid group} $\PB_n$ on $n$ strands is 
the kernel of the standard group homomorphism $\rho: \B_n \to S_n$. 
This homomorphism sends the generator $\si_i$ to the transposition $(i, i+1)$. 

We denote by $x_{ij}$ (for $1 \le i < j \le n$) the following elements of $\PB_n$
\begin{equation}
\label{x-ij}
x_{ij} : =  \si_{j-1} \dots \si_{i+1} \si^2_i \si^{-1}_{i+1} \dots \si^{-1}_{j-1}
\end{equation}
and recall \cite[Section 1.3]{Braids} that $\PB_n$ has the following presentation: 
$$
\PB_n \cong \lan  \{ x_{ij} \}_{1 \le i < j \le n} ~|~ \textrm{the relations}\ran
$$
with the relations
\begin{equation}
\label{PB-n-rel}
x^{-1}_{rs} x_{ij} x_{rs} = 
\begin{cases}
x_{ij} ~~~\textrm{if} ~~ s < i \textrm{ or } i < r < s < j,\\[0.18cm]
x_{r j} x_{ij} x^{-1}_{r j} ~~~\textrm{if} ~~~ s = i, \\[0.18cm]
x_{r j} x_{s j} x_{ij} x^{-1}_{s j} x^{-1}_{r j} ~~~\textrm{if} ~~~ r = i < s < j, \\[0.18cm]
x_{r j} x_{sj}  x^{-1}_{rj} x^{-1}_{sj} \, x_{ij} \, x_{sj}  x_{rj}  x^{-1}_{sj}  x^{-1}_{rj}  ~~~\textrm{if} ~~~ r < i < s < j. 
\end{cases}
\end{equation}

For example, the standard generators of $\PB_3$ are
\begin{equation}
\label{PB3-gener}
x_{12} : = \si_1^2, \qquad x_{23} : = \si_2^2, \qquad 
x_{13} : = \si_2 \si_1^2 \si_2^{-1}\,.  
\end{equation}
The element 
\begin{equation}
\label{c-PB3-B3}
c : = x_{23} x_{12} x_{13} = x_{12} x_{13} x_{23} = (\si_1 \si_2)^3 =   (\si_2 \si_1)^3 
\end{equation}
has an infinite order; it generates the center of $\PB_3$ and the center of $\B_3$. 

The elements $x_{12}$ and $x_{23}$ generate a free subgroup in $\PB_3$. 
Thus $\PB_3$ is isomorphic to $\F_2 \times \bbZ$.  

A direction calculation shows that 
\begin{equation}
\label{conj-xxx-PB3}
\si_1^{-1} x_{23} \si_1 = x_{13} \,, \qquad
\si_2^{-1}x_{12}\si_2 = x_{23}^{-1}x_{12}^{-1}c \,,
\qquad
\si_2^{-1} x_{13} \si_2 = x_{12}.
\end{equation}

\subsection{The groupoid $\PaB(n)$}
\label{app:groupoid-PaB}
Objects of $\PaB(n)$ are parenthesizations of sequences $(\tau(1), \tau(2), \dots, \tau(n))$ where 
$\tau$ is a permutation $S_n$. For example, $\PaB(2)$ has exactly two objects $(1\, 2)$ and $(2\, 1)$ and  
$\PaB(3)$ has 12 objects: 
$$
(12)3,~ (21)3,~ (23)1,~ (32)1,~ (31)2,~ (13)2,~ 
1(23),~ 2(13),~ 2(31),~ 3(21),~ 3(12),~ 1(32).
$$ 

To define morphisms in $\PaB(n)$, we denote by $\mmp$ the obvious projection 
from the set of objects of $\PaB(n)$ onto $S_n$. For example, 
$$
\mmp( \,(23)1\,) : = 
\left(
\begin{array}{ccc}
1 & 2  & 3  \\
2  & 3  & 1  
\end{array}
\right).
$$

For two objects $\tau_1, \tau_2$ of $\PaB(n)$ we set 
\begin{equation}
\label{Hom-PaB}
\Hom_{\PaB}(\tau_1, \tau_2) : = \rho^{-1} (\mmp(\tau_2)^{-1} \circ  \mmp(\tau_1)) \,\subset\, \B_n\,,
\end{equation}
where $\rho$ is the standard homomorphism $\B_n$ to $S_n$.  

For instance, $\Hom_{\PaB}(2(31), (31)2)$ consist of elements $g \in \B_n$ such that 
$$
\rho(g) = 
\left(
\begin{array}{ccc}
1 & 2  & 3  \\
3 & 1  & 2  
\end{array}
\right).
$$
An example of an isomorphism from $2(31)$ to $(31)2$ is shown in figure \ref{fig:hom-PaB}
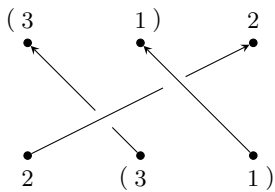
\begin{figure}[htp] 
\centering 
\begin{tikzpicture}[scale=1.5, > = stealth]
\tikzstyle{v} = [circle, draw, fill, minimum size=0, inner sep=1]
\node[v] (v2) at (0, 0) {};
\draw (0,-0.2) node[anchor=center] {{\small $2$}};
\node[v] (v3) at (1, 0) {};
\draw (0.85,-0.2) node[anchor=center] {{\small $($}};
\draw (1,-0.2) node[anchor=center] {{\small $3$}};
\node[v] (v1) at (2, 0) {};
\draw (2,-0.2) node[anchor=center] {{\small $1$}};
\draw (2.15,-0.2) node[anchor=center] {{\small $)$}};
\node[v] (vv3) at (0, 1) {};
\draw (-0.15,1.2) node[anchor=center] {{\small $($}};
\draw (0,1.2) node[anchor=center] {{\small $3$}};
\node[v] (vv1) at (1, 1) {};
\draw (1,1.2) node[anchor=center] {{\small $1$}};
\draw (1.15,1.2) node[anchor=center] {{\small $)$}};
\node[v] (vv2) at (2, 1) {};
\draw (2,1.2) node[anchor=center] {{\small $2$}};
\draw [->] (v1) -- (vv1);
\draw (v2) -- (1.2,0.6) [->] (1.4,0.7) -- (vv2);
\draw (v3) -- (0.75,0.25) [->] (0.6,0.4) -- (vv3) ;
\end{tikzpicture}
\caption{An example of an isomorphism from  $2(31)$ to $(31)2$ in $\PaB(3)$} \label{fig:hom-PaB}
\end{figure}

The composition of morphisms in $\PaB(n)$ comes from the multiplication in $\B_n$. 
For example, if $\eta$ is the element of $\Hom_{\PaB}(\tau_1, \tau_2)$ corresponding to $h \in B_n$ and 
$\ga$ is the element of $\Hom_{\PaB}(\tau_2, \tau_3)$ corresponding to $g \in B_n$ then their composition 
$\ga \cdot \eta$ is the element of $\Hom_{\PaB}(\tau_1, \tau_3)$ corresponding to $g \cdot h$.
Note that we use $\cdot$ for the composition of morphisms in $\PaB$ and the multiplication of elements in braid groups. 

By definition of morphisms, we have a natural forgetful map 
\begin{equation}
\label{ou}
\ou : \PaB(n) \to \B_n\,.
\end{equation}
This map assigns to a morphism $\ga\in \PaB(n)$ the corresponding element of the braid group $\B_n$. 
Moreover, since the composition of morphisms in $\PaB(n)$ comes from the multiplication in $\B_n$, we have
$$
\ou(\ga \cdot \eta) = \ou(\ga) \cdot \ou(\eta) 
$$
for every pair $\ga,\eta$ of composable morphisms. 

The isomorphisms $\al \in \PaB(3)$ and $\beta\in \PaB(2)$ shown in figure \ref{fig:beta-alpha}
play a very important role. We call $\beta$ {\it the braiding} and $\al$ {\it the associator}.
Note that, although $\al$ corresponds to the identity element in $\B_3$, it is not an 
identity morphism in $\PaB(3)$ because $(12)3 \neq 1(23)$. 
\begin{figure}[htp] 
\centering 
\begin{tikzpicture}[scale=1.5, > = stealth]
\tikzstyle{v} = [circle, draw, fill, minimum size=0, inner sep=1]
\draw (-0.7,0.5) node[anchor=center] {{$\beta ~ : = $}};
\node[v] (v1) at (0, 0) {};
\draw (0,-0.2) node[anchor=center] {{\small $1$}};
\draw (0,1.2) node[anchor=center] {{\small $2$}};
\node[v] (v2) at (1, 0) {};
\draw (1,-0.2) node[anchor=center] {{\small $2$}};
\draw (1,1.2) node[anchor=center] {{\small $1$}};
\node[v] (vv2) at (0, 1) {};
\node[v] (vv1) at (1, 1) {};
\draw [->] (v1) -- (vv1);
\draw (v2) -- (0.6, 0.4); 
\draw [->] (0.4, 0.6) -- (vv2);
\begin{scope}[shift={(3.5,0)}]
\draw (-0.7,0.5) node[anchor=center] {{$\al ~ : = $}};
\node[v] (v1) at (0, 0) {};
\draw (-0.15,-0.2) node[anchor=center] {{\small $($}};
\draw (0,-0.2) node[anchor=center] {{\small $1$}};
\node[v] (v2) at (0.5, 0) {};
\draw (0.5,-0.2) node[anchor=center] {{\small $2$}};
\draw (0.65,-0.2) node[anchor=center] {{\small $)$}};
\node[v] (v3) at (1.5, 0) {};
\draw (1.5,-0.2) node[anchor=center] {{\small $3$}};
\node[v] (vv1) at (0, 1) {};
\draw (0,1.2) node[anchor=center] {{\small $1$}};
\node[v] (vv2) at (1, 1) {};
\draw (0.85,1.2) node[anchor=center] {{\small $($}};
\draw (1,1.2) node[anchor=center] {{\small $2$}};
\node[v] (vv3) at (1.5, 1) {};
\draw (1.5,1.2) node[anchor=center] {{\small $3$}};
\draw (1.65,1.2) node[anchor=center] {{\small $)$}};
\draw [->] (v1) -- (vv1);  \draw [->] (v2) -- (vv2);  \draw [->] (v3) -- (vv3); 
\end{scope}
\end{tikzpicture}
\caption{The isomorphisms $\al$ and $\beta$} \label{fig:beta-alpha}
\end{figure}

The symmetric group $S_n$ acts on $\Ob(\PaB(n))$ in the obvious way. Moreover, 
for every $\te \in S_n$ and $\ga \in \Hom_{\PaB(n)}(\tau_1, \tau_2)$, we denote by 
$\te(\ga)$ the morphism from $\te(\tau_1)$ to $\te(\tau_2)$ that corresponds to the 
same element of the braid group $B_n$, i.e. 
\begin{equation}
\label{ou-S-n}
\ou (\te(\ga)) = \ou (\ga).
\end{equation}
 
For example, if $\te = (1,2) \in S_3$ then 
$$
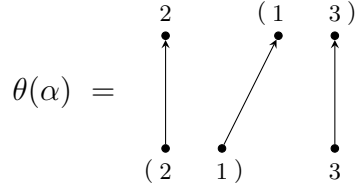
\begin{tikzpicture}[scale=1.5, > = stealth]
\tikzstyle{v} = [circle, draw, fill, minimum size=0, inner sep=1]
\draw (-0.9,0.5) node[anchor=center] {{$\te(\al) ~ = $}};
\node[v] (v1) at (0, 0) {};
\draw (-0.15,-0.2) node[anchor=center] {{\small $($}};
\draw (0,-0.2) node[anchor=center] {{\small $2$}};
\node[v] (v2) at (0.5, 0) {};
\draw (0.5,-0.2) node[anchor=center] {{\small $1$}};
\draw (0.65,-0.2) node[anchor=center] {{\small $)$}};
\node[v] (v3) at (1.5, 0) {};
\draw (1.5,-0.2) node[anchor=center] {{\small $3$}};
\node[v] (vv1) at (0, 1) {};
\draw (0,1.2) node[anchor=center] {{\small $2$}};
\node[v] (vv2) at (1, 1) {};
\draw (0.85,1.2) node[anchor=center] {{\small $($}};
\draw (1,1.2) node[anchor=center] {{\small $1$}};
\node[v] (vv3) at (1.5, 1) {};
\draw (1.5,1.2) node[anchor=center] {{\small $3$}};
\draw (1.65,1.2) node[anchor=center] {{\small $)$}};
\draw [->] (v1) -- (vv1);  \draw [->] (v2) -- (vv2);  \draw [->] (v3) -- (vv3); 
\end{tikzpicture}
$$

For our purposes, it is convenient to assign to every element $g \in \B_n$ the corresponding morphism 
$\mm(g)\in \PaB(n)$  from $\paren 1,2)3) \dots n)$ to $\paren i_1 , i_2) i_3)\dots i_n)$, 
where $i_k : = \rho(g)^{-1}(k)$. It is easy to see that the map 
\begin{equation}
\label{mm}
\mm : \B_n \to \PaB(n)
\end{equation}
defined in this way is a right inverse of $\ou$ (see \eqref{ou}).  

It is also easy to see that, for every pair $g_1, g_2 \in \B_n$, we have 
\begin{equation}
\label{mm-comp}
\mm(g_1 \cdot g_2)  = \rho(g_2)^{-1}\big(\mm(g_1)\big) \cdot \mm(g_2).
\end{equation}

For example, for $\si_1, \si_2 \in \B_3$, $\mm(\si_1) = \id_{12} \circ_1 \beta$ and 
$$
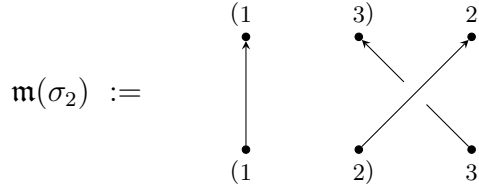
\begin{tikzpicture}[scale=1.5, > = stealth]
\tikzstyle{v} = [circle, draw, fill, minimum size=0, inner sep=1]
\draw (-1.5,0.5) node[anchor=center] {{$\mm(\si_2) ~: =$}};
\node[v] (v1) at (0, 0) {};
\draw (-0.1,-0.2) node[anchor=center] {{\small $($}};
\draw (0,-0.2) node[anchor=center] {{\small $1$}};
\node[v] (v2) at (1, 0) {};
\draw (1,-0.2) node[anchor=center] {{\small $2$}};
\draw (1.1,-0.2) node[anchor=center] {{\small $)$}};
\node[v] (v3) at (2, 0) {};
\draw (2,-0.2) node[anchor=center] {{\small $3$}};
\node[v] (vv1) at (0, 1) {};
\draw (-0.1,1.2) node[anchor=center] {{\small $($}};
\draw (0,1.2) node[anchor=center] {{\small $1$}};
\node[v] (vv3) at (1, 1) {};
\draw (1,1.2) node[anchor=center] {{\small $3$}};
\draw (1.1,1.2) node[anchor=center] {{\small $)$}};
\node[v] (vv2) at (2, 1) {};
\draw (2,1.2) node[anchor=center] {{\small $2$}};
\draw [->] (v1) -- (vv1);
\draw [->] (v2) -- (vv2);
\draw (v3) -- (1.6,0.4) [->] (1.4,0.6) -- (vv3);
\end{tikzpicture}
$$

The composition $\mm(\si_2) \cdot \mm(\si_1)$ is not defined because the source of $\mm(\si_2)$ does not 
coincide with the target of $\mm(\si_1)$. On the other hand, the source of $(1,2) (\mm(\si_2))$ coincides 
with the target of $\mm(\si_1)$ and  $(1,2) (\mm(\si_2))\, \cdot\, \mm(\si_1) = \mm(\si_2 \cdot \si_1)$. 

\subsection{The operad structure on $\PaB$}
\label{app:oper-PaB} 

We already explained how the symmetric group $S_n$ acts on the groupoid $\PaB(n)$. 
Furthermore, it is easy to see that $\{\Ob(\PaB(n))\}_{n \ge 1}$ is the underlying collection of the free operad 
(in the category of sets) generated by the collection $\sT$ with 
$$
\sT(n) : = 
\begin{cases}
\{~1\,2, 2\,1 ~\}  \qquad \textrm{if} ~~ n = 2 \,, \\
~~~\emptyset \qquad \qquad \textrm{otherwise}.
\end{cases}
$$ 

Thus the functors 
\begin{equation}
\label{circ-i}
\circ_i ~:~ \PaB(n) \times \PaB(m) ~\to~ \PaB(n+m-1)
\end{equation}
act on the level of objects in the obvious way. 

For example, 
$$
(23)1 \,\circ_2\, 12 := (\gray{(23)}4)1, 
\quad
21 \,\circ_1\, (23)1 : = 4 \gray{ ((2 3)1)}, 
\quad 
2( 3 (14)) \,\circ_3\, 1(32) := 2 ( \gray{(3(54))}(16) ),
$$
where we use the gray color to indicate what happens with the inserted sequence. 
For instance, in the third example,  $1(32) \mapsto \gray{(3(54))}$.  

To define the action of the functor $\circ_i$ on the level of morphisms, we proceed 
as follows: given $\ga \in \PaB(n)$, $\ti{\ga} \in \PaB(m)$ and $1 \le i \le n$, we 
set $g : = \ou(\ga)$ and $\ti{g} : = \ou(\ti{\ga})$;  we 
compute the source and the target of $\ga \circ_i \ti{\ga}$ using the 
rules of operad $\{\Ob(\PaB(k))\}_{k \ge 1}$. Finally, to get the element of $\B_{n+m -1}$
corresponding to $\ga \circ_i \ti{\ga}$, we replace the strand of $g$ that originates at 
the position labeled by $i$ by a ``thin'' version of $\ti{g}$. For example,  
$$
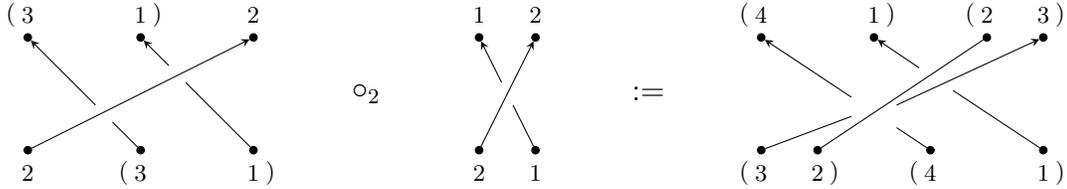
\begin{tikzpicture}[scale=1.5, > = stealth]
\tikzstyle{v} = [circle, draw, fill, minimum size=0, inner sep=1]
\node[v] (v2) at (0, 0) {};
\draw (0,-0.2) node[anchor=center] {{\small $2$}};
\node[v] (v3) at (1, 0) {};
\draw (0.85,-0.2) node[anchor=center] {{\small $($}};
\draw (1,-0.2) node[anchor=center] {{\small $3$}};
\node[v] (v1) at (2, 0) {};
\draw (2,-0.2) node[anchor=center] {{\small $1$}};
\draw (2.15,-0.2) node[anchor=center] {{\small $)$}};
\node[v] (vv3) at (0, 1) {};
\draw (-0.15,1.2) node[anchor=center] {{\small $($}};
\draw (0,1.2) node[anchor=center] {{\small $3$}};
\node[v] (vv1) at (1, 1) {};
\draw (1,1.2) node[anchor=center] {{\small $1$}};
\draw (1.15,1.2) node[anchor=center] {{\small $)$}};
\node[v] (vv2) at (2, 1) {};
\draw (2,1.2) node[anchor=center] {{\small $2$}};
\draw (v1) -- (1.4, 0.6)  [->] (1.25,0.75) -- (vv1);
\draw [->] (v2) --  (vv2);
\draw (v3) -- (0.75,0.25) [->] (0.6,0.4) -- (vv3) ;
\draw (3,0.5) node[anchor=center] {{$\circ_2$}};
\begin{scope}[shift={(4,0)}]
\node[v] (v2) at (0, 0) {};
\draw (0,-0.2) node[anchor=center] {{\small $2$}};
\draw (0,1.2) node[anchor=center] {{\small $1$}};
\node[v] (v1) at (0.5, 0) {};
\draw (0.5,-0.2) node[anchor=center] {{\small $1$}};
\draw (0.5,1.2) node[anchor=center] {{\small $2$}};
\node[v] (vv1) at (0, 1) {};
\node[v] (vv2) at (0.5, 1) {};
\draw [->] (v2) -- (vv2);
\draw (v1) -- (0.3, 0.4); 
\draw [->] (0.2, 0.6) -- (vv1);
\draw (1.5,0.5) node[anchor=center] {{$ : = $}};
\end{scope}
\begin{scope}[shift={(6.5,0)}]
\node[v] (v3) at (0, 0) {};
\draw (-0.15,-0.2) node[anchor=center] {{\small $($}};
\draw (0,-0.2) node[anchor=center] {{\small $3$}};
\node[v] (v2) at (0.5, 0) {};
\draw (0.5,-0.2) node[anchor=center] {{\small $2$}};
\draw (0.65,-0.2) node[anchor=center] {{\small $)$}};
\node[v] (v4) at (1.5, 0) {};
\draw (1.35,-0.2) node[anchor=center] {{\small $($}};
\draw (1.5,-0.2) node[anchor=center] {{\small $4$}};
\node[v] (v1) at (2.5, 0) {};
\draw (2.5,-0.2) node[anchor=center] {{\small $1$}};
\draw (2.65,-0.2) node[anchor=center] {{\small $)$}};
\node[v] (vv4) at (0, 1) {};
\draw (-0.15,1.2) node[anchor=center] {{\small $($}};
\draw (0,1.2) node[anchor=center] {{\small $4$}};
\node[v] (vv1) at (1, 1) {};
\draw (1,1.2) node[anchor=center] {{\small $1$}};
\draw (1.15,1.2) node[anchor=center] {{\small $)$}};
\node[v] (vv2) at (2, 1) {};
\draw (1.85,1.2) node[anchor=center] {{\small $($}};
\draw (2,1.2) node[anchor=center] {{\small $2$}};
\node[v] (vv3) at (2.5, 1) {};
\draw (2.5,1.2) node[anchor=center] {{\small $3$}};
\draw (2.65,1.2) node[anchor=center] {{\small $)$}};
\draw  (v1) --  (1.7, 0.53) [->] (1.4,0.73) -- (vv1);
\draw (v2) -- (vv2);
\draw (v3) -- (0.8,0.3) [->] (1.2,0.4) -- (vv3);
\draw (v4) -- (1.2,0.2) [->] (0.8,0.47) -- (vv4);
\end{scope}
\end{tikzpicture}
$$

For a more precise definition of operadic multiplications on $\PaB$ we refer 
the reader to \cite[Chapter 6]{Fresse1}.

The (iso)morphisms $\al$ and $\beta$ satisfy the following {\it pentagon relation}
\begin{equation}
\label{pentagon}
\begin{tikzpicture}
\matrix (m) [matrix of math nodes, row sep=1.5em, column sep=1.5em]
{~&  (1(23))4  &  ~ & 1((23)4) &  ~ \\
((12)3)4 & ~ & ~ & ~  & 1(2(34)) \\ 
 ~ & ~~ & (12)(34)  & ~ &  ~ \\};
\path[->, font=\scriptsize]
(m-2-1) edge node[above] {$\id_{12} \circ_1 \al~~~~~~~~$} (m-1-2)
(m-1-2) edge node[above] {$\al \circ_2 \id_{12}$} (m-1-4)
(m-1-4) edge node[above] {$~~~~~~~~\id_{12} \circ_2 \al$} (m-2-5)
(m-2-1) edge node[above] {$~~~~\al \circ_1 \id_{12}$} (m-3-3)
 (m-3-3)  edge node[above] {$\al \circ_3 \id_{12}~~~~$} (m-2-5);
\end{tikzpicture}
\end{equation}
and the two {\it hexagon relations}:
 \begin{equation}
\label{hexagon1}
\begin{tikzpicture}
\matrix (m) [matrix of math nodes, row sep=1.8em, column sep=3.2em]
{(12) 3 &  3(12)  &  (31)2  \\
1(23)  & 1(32) &   (13)2  \\ };
\path[->, font=\scriptsize]
(m-1-1) edge node[above] {$\beta \circ_1 \id_{12}$} (m-1-2)  edge node[left] {$\al~$} (m-2-1)
(m-2-1) edge node[above] {$\id_{12} \circ_2 \beta $} (m-2-2)  (m-2-2)  edge node[above] {$ (2,3)\, \al^{-1}$} (m-2-3) 
(m-2-3)  edge node[right] {$(2,3)\, (\id_{12} \circ_1 \beta)$} (m-1-3) (m-1-3)  edge node[above] {$(1,3,2)\, \al$} (m-1-2) ;
\end{tikzpicture}
\end{equation}
\begin{equation}
\label{hexagon11}
\begin{tikzpicture}
\matrix (m) [matrix of math nodes, row sep=1.8em, column sep=3.2em]
{1(23) &  (23)1  &  2(31)  \\
(12)3  & (21)3 &   2(13)  \\ };
\path[->, font=\scriptsize]
(m-1-1) edge node[above] {$\beta \circ_2 \id_{12}$} (m-1-2)  (m-1-2) edge node[above] {$(1,2,3)\, \al$}  (m-1-3)
(m-1-1) edge node[left] {$\al^{-1}$} (m-2-1)  (m-2-1) edge node[above] {$\id_{12} \circ_1 \beta$} (m-2-2) 
(m-2-2)  edge node[above] {$(1,2)\, \al $} (m-2-3)  (m-2-3) edge node[right] {$(1,2) \, (\id_{12} \circ_2 \beta)$}  (m-1-3);
\end{tikzpicture}
\end{equation}

It is known \cite[Theorem 6.2.4]{Fresse1} that\footnote{A very similar statement is proved in 
\cite{BNGT}. See Claim 2.6 in {\it loc. cit.} It goes without saying that Theorem \ref{thm:PaB-gener} can 
be thought of as a version of MacLane's coherence theorem for braided monoidal categories.} 
\begin{thm}
\label{thm:PaB-gener}
As the operad in the category of groupoids, $\PaB$ is generated by 
morphisms $\al$ and $\beta$ shown in figure \ref{fig:beta-alpha}.
Moreover, any relation on $\al$ and $\beta$ in $\PaB$ is a consequence 
of \eqref{pentagon}, \eqref{hexagon1} and  \eqref{hexagon11}.
\end{thm}

\subsection{The cosimplicial homomorphisms for pure braid groups in arities $2,3,4$}
\label{app:vfs}

The collection $\{\PB_n\}_{n \ge 1}$ of pure braid groups can be equipped 
with the structure of a cosimplicial group. For our purposes we will need the cofaces 
of this cosimplicial structure only in arities $2,3$ and $4$. 

Let $\tau_1$ and $\tau_2$ be objects of $\PaB(n)$ which differ only by parenthesizations, 
i.e. $\mmp(\tau_1) = \mmp(\tau_2)$. For such objects, we denote by $\al_{\tau_1}^{\tau_2}$ the 
isomorphism from $\tau_1$ to $\tau_2$ given by the identity element of $\B_n$. 
For example, the associator $\al$ is precisely $\al_{(12)3}^{1(23)}$ and $\al^{-1}$
is precisely $\al_{1(23)}^{(12)3}$.  
  
Using the identity morphism $\id_{12} \in \PaB(2)$, the maps $\ou$, $\mm$ (see \eqref{ou}, \eqref{mm})
and the operadic insertions, we define 
the following maps from $\PB_3$ to $\PB_4$ and the maps from $\PB_2$ to $\PB_3$: 
$$
\vf_{123} (h) : =  \ou ( \id_{12} \circ_1 \mm(h) ), \qquad \vf_{12,3,4}(h) : = \ou( \mm(h) \circ_1 \id_{12} ), 
$$
\begin{equation}
\label{vfs-PB}
\vf_{1,23,4} (h) : = \ou ( \mm(h) \circ_2 \id_{12}  ),  
\end{equation}
$$
\vf_{1,2,34} (h) : = \ou ( \mm(h) \circ_3 \id_{12} ) ,
\qquad
\vf_{234} (h) : =  \ou ( \id_{12} \circ_2 \mm(h) ), 
$$
\bigskip
$$
\vf_{12} (h) : = \ou (\id_{12} \circ_1 \mm(h)), \qquad \vf_{23}(h) : = \ou(\id_{12} \circ_{2} \mm(h) ),
$$
\begin{equation}
\label{vfs-PB-2-3}
\vf_{12,3}(h) := \ou( \mm(h) \circ_1 \id_{12} ), \qquad 
\vf_{1,23}(h) := \ou(\mm(h) \circ_2 \id_{12}).
\end{equation}

We claim that 
\begin{prop}
\label{prop:vf-h-sm}
The equations in \eqref{vfs-PB} (resp. in  \eqref{vfs-PB-2-3}) define group homomorphisms 
from $\PB_3$ (resp. $\PB_2$) to $\PB_4$ (resp. $\PB_3$). 
\end{prop}
\begin{proof}
Let us consider the map $\vf_{1,23,4} : \PB_3 \to \PB_4$.
For elements $h$, $\ti{h} \in \PB_3$, we set 
$$
\ga : = \mm(h), \qquad \ti{\ga} : = \mm(\ti{h}).
$$

Since $\PaB$ is an operad in the category of groupoids,  we have 
$$
(\ga \cdot \ti{\ga}) \circ_2 \id_{12} = (\ga \circ_2 \id_{12}) \cdot (\ti{\ga} \circ_2 \id_{12}).
$$

Hence 
$$
\vf_{1,23,4} (h) \cdot \vf_{1,23,4} (\ti{h}) = \ou(\ga \circ_2 \id_{12}) \cdot \ou(\ti{\ga} \circ_2 \id_{12}) 
 = \ou \big(\, (\ga \circ_2 \id_{12}) \cdot (\ti{\ga} \circ_2 \id_{12}) \, \big) = 
$$
$$
\ou \big(\, (\ga \cdot \ti{\ga}) \circ_2 \id_{12} \, \big) = \vf_{1,23,4} (h\cdot \ti{h}), 
$$
where the last identity is a consequence of $\ga \cdot \ti{\ga} =  \mm(h\cdot \ti{h})$.

The proofs for the remaining $8$ maps are very similar and we leave it to the reader.
\end{proof}

Since all $9$ maps in \eqref{vfs-PB} and \eqref{vfs-PB-2-3} are group homomorphisms, 
they are uniquely determined by their values on generators of $\PB_3$ and $\PB_4$, respectively. 
It is easy to see that 
$$
\vf_{123}(x_{12}) = x_{12}, \quad
\vf_{123}(x_{23}) = x_{23}, \quad
\vf_{123}(x_{13}) = x_{13}, 
$$
$$
\vf_{234}(x_{12}) = x_{23}, \quad
\vf_{234}(x_{23}) = x_{34}, \quad
\vf_{234}(x_{13}) = x_{24},
$$
\begin{equation}
\label{vfs-3-4-on-gen}
\vf_{12,3,4}(x_{12}) = x_{13} x_{23}, \quad
\vf_{12,3,4}(x_{23}) =  x_{34}, \quad
\vf_{12,3,4}(x_{13}) =  x_{14} x_{24},
\end{equation}
$$
\vf_{1,23,4}(x_{12}) = x_{12} x_{13}, \quad
\vf_{1,23,4}(x_{23}) =  x_{24} x_{34}, \quad
\vf_{1,23,4}(x_{13}) =  x_{14},
$$
$$
\vf_{1, 2, 34}(x_{12}) = x_{12}, \quad
\vf_{1, 2, 34}(x_{23}) = x_{23} x_{24}, \quad
\vf_{1, 2, 34}(x_{13}) = x_{13} x_{14}.
$$
\bigskip
\begin{equation}
\label{vfs-2-3-on-gen}
\vf_{12}(x_{12}) = x_{12}, \quad \vf_{23}(x_{12}) = x_{23}, 
\quad 
\vf_{12,3}(x_{12}) = x_{13} x_{23}, 
\quad 
\vf_{1,23}(x_{12}) = x_{12} x_{13},
\end{equation}

\subsection{The profinite completion $\wh{\PaB}$ of $\PaB$}
\label{sec:PaB-hat}

Let $\cG$ be a connected groupoid with finitely many objects and $G$ be the group that
represents the isomorphism class of $\Aut(a)$ for some object $a$ of $\cG$. We tacitly assume that 
the group $G$ is residually finite.  
Following \cite{ProfiniteGrpd}, an equivalence relation $\sim$ on $\cG$ is called \emph{compatible}, if  

\begin{enumerate}

\item $\ga_1 \sim \ga_2 \quad \Rightarrow$ the source (resp. the target) 
of $\ga_1$ coincides with the source (resp. the target) of $\ga_2$;  

\item  $\ga_1 \sim \ga_2 \quad \Rightarrow \quad \ga_1 \cdot \ga \sim \ga_2  \cdot \ga
 ~\textrm{and}~ \tau \cdot \ga_1 \sim  \tau \cdot \ga_2 $ (if the compositions are defined);

\item the set $\cG / \sim$ of equivalence classes is finite. 

\end{enumerate}
It is clear that, for every compatible equivalence relation $\sim$ on 
$\cG$, the quotient $\cG/\sim$ is naturally a finite groupoid (with the same set of objects).

Compatible equivalence relations on $\cG$ form a directed poset and the assignment $\sim \mapsto \cG/\sim$
gives us a functor from this poset to the category of finite groupoids.  
In \cite{ProfiniteGrpd}, the profinite completion $\wh{\cG}$ of the groupoid $\cG$ is defined as the limit 
of this functor. 

In \cite{ProfiniteGrpd}, it was also shown that compatible equivalence relations on $\cG$ are in bijection 
with finite index normal subgroups $\N$ of $G$. This gives us the following  ``pedestrian'' way of thinking 
about morphisms in $\wh{\cG}(a,b)$:  choose\footnote{$\cG(a,b)$ is non-empty because $\cG$ is connected.} 
$\la \in \cG(a,b)$, then every morphism in $\ga \in \wh{\cG}(a,b)$ can be uniquely written as 
$$
\ga = \la \cdot h,
$$ 
where $h \in \wh{G}$.

In \cite{ProfiniteGrpd}, we also proved that the assignment $\cG  \mapsto \wh{\cG}$ upgrades to a functor 
from the category of groupoids to the category of topological groupoids. 
Moreover, this is a symmetric monoidal functor.
 
Thus ``putting hats'' over $\PaB(n)$ for every $n \ge 0$ gives us an operad $\wh{\PaB}$ in the category 
of topological groupoids.  

\section{Charming $\GT$-shadows in the Abelian setting. Examples of genuine $\GT$-shadows}
\label{app:Abelian}

Let us prove the following statement: 
\begin{prop}  
\label{prop:Abelian}
For $\N \in \NFI_{\PB_4}(\B_4)$, the following conditions are equivalent: 
\begin{itemize}

\item [{\bf a)}] the quotient group $\PB_4/\N$ is Abelian;
\item [{\bf b)}] the quotient group $\PB_3/\N_{\PB_3}$ is Abelian;
\item [{\bf c)}] the quotient group $\F_2/\N_{\F_2}$ is Abelian. 

\end{itemize}
\end{prop}  
\begin{proof} Implications {\bf a)} $\Rightarrow$ {\bf b)} and 
{\bf b)} $\Rightarrow$ {\bf c)} are straightforward so we leave them to the reader.  

Let us assume that the quotient group $\F_2/\N_{\F_2}$ is Abelian. 
Then the images of $x_{12}$ and $x_{23}$ in $\PB_3/\N_{\PB_3}$ commute. 
Furthermore, since the image of $c$ in $\PB_3/\N_{\PB_3}$
is obviously in the center of $\PB_3/\N_{\PB_3}$ and $\PB_3 =\lan x_{12}, x_{23}, c\ran$,
we conclude that the quotient group $\PB_3/\N_{\PB_3}$ is also Abelian.  

To show that the generators $\bar{x}_{ij} := x_{ij} \N$ ($1 \le i < j \le 4$) of $\PB_4/\N$ commute
with each other, we consider the group homomorphisms from $\PB_3$ to $\PB_4$
given by formulas \eqref{vfs-3-4-on-gen}. 

Note that, for every homomorphism $\vf : \PB_3 \to \PB_4$ in the set
\begin{equation}
\label{set-vfs}
\{\vf_{234}, \vf_{12,3,4}, \vf_{1,23,4}, \vf_{1,2,34}, \vf_{234}\},
\end{equation}
we have $\N_{\PB_3} \le \vf^{-1}(\N) \le \PB_3$. Therefore, since the quotient $\PB_3/\N_{\PB_3}$
is Abelian, the quotient  $\PB_3/\vf^{-1}(\N)$ is also Abelian. 

Applying these observations to every $\vf$ in \eqref{set-vfs}, we deduce that 
\begin{itemize}

\item the elements  $\bar{x}_{12}$,  $\bar{x}_{23}$, $\bar{x}_{13}$ commute with each other;

\item the elements  $\bar{x}_{23}$,  $\bar{x}_{34}$, $\bar{x}_{24}$ commute with each other;

\item  the elements 
$\bar{x}_{13} \bar{x}_{23}$, $\bar{x}_{34}$ and $\bar{x}_{14} \bar{x}_{24}$ 
commute with each other;

\item the elements $\bar{x}_{12}$, $\bar{x}_{23} \bar{x}_{24}$ and $\bar{x}_{13} \bar{x}_{14}$
commute with each other;

\item  the elements $\bar{x}_{14}$, $\bar{x}_{12} \bar{x}_{13}$ and $\bar{x}_{24} \bar{x}_{34}$
commute with each other.

\end{itemize}

Using these observations one can show that $[\bar{x}_{ij}, \bar{x}_{kl}] = 1_{\PB_4/\N}$ for 
every pair in the set 
$$
\{ \{(i,j), (k,l)\} \,|\, 1 \le i < j \le 4, \, 1 \le k < l \le 4,\}  -  \{\{(1,2),(3,4)\}, \{ (1,3),(2,4) \}, \{ (2,3),(1,4) \} \}.
$$
 
Luckily, due to \eqref{PB-n-rel}, we have 
$$
x_{12} x_{34} = x_{34} x_{12}\,,
\qquad 
x_{23} x_{14} = x_{14} x_{23}\,, \qquad
x_{13}^{-1} x_{24} x_{13} =  [x_{14} , x_{34}] x_{24}  [x_{14} , x_{34}]^{-1}\,.
$$

Thus all generators $\bar{x}_{ij}$ of $\PB_4/\N$ commute with each other. 
\end{proof}

If one of the three equivalent conditions of Proposition \ref{prop:Abelian}
is satisfied then we say that we are in {\it the Abelian setting}.  

\bigskip

We can now prove the following analog of the Kronecker-Weber theorem:
\begin{thm}  
\label{thm:Abelian}
Let $\N \in \NFI_{\PB_4}(\B_4)$.  If the quotient group $\PB_4/\N $ is Abelian then
\begin{equation}
\label{GT-Abelian}
\GT^{\hs}(\N) = \{(m + N_{\ord} \bbZ, \bar{1}) ~|~ 0 \le m \le N_{\ord}-1, ~ \gcd(2m+1,  N_{\ord}) =1\},
\end{equation}
where $\bar{1}$ is the identity element of $\F_2/\N_{\F_2}$. Furthermore, 
every $\GT$-shadow in \eqref{GT-Abelian} is genuine. 
\end{thm}  
\begin{proof} Since $\bar{1}$ can be represented by the identity element of $\F_2$, every
element of the set  
\begin{equation}
\label{X-Abelian}
X_{\N} : =  \{(m + N_{\ord} \bbZ, \bar{1}) ~|~ 0 \le m \le N_{\ord}-1, ~ \gcd(2m+1,  N_{\ord}) =1\}
\end{equation}
satisfies the pentagon relation \eqref{GT-penta}. 

For every element of $X_{\N}$, the hexagon relations \eqref{hexa1} and \eqref{hexa11}
boil down to 
\begin{equation}
\label{hexa1-Ab}
\si_1 x_{12}^m  \si_2 x_{23}^m  \, \N_{\PB_3} ~ = ~ 
 \si_1 \si_2 (x_{13} x_{23})^m \, \N_{\PB_3}
\end{equation}
and
\begin{equation}
\label{hexa11-Ab}
 \si_2 x_{23}^m  \, \si_1 x_{12}^m \,  \N_{\PB_3}  
~=~ \si_2 \si_1 (x_{12} x_{13})^m \,    \N_{\PB_3}\,.
\end{equation}

Equation \eqref{hexa1-Ab} follows easily from the identity 
$$
\si^{-1}_2 x_{12} \si_2 = x_{23}^{-1} x_{13} x_{23}
$$
and the fact that the quotient $\PB_3/\N_{\PB_3}$ is Abelian. 

Similarly,  equation \eqref{hexa11-Ab} follows easily from the identity 
$$
\si^{-1}_1 x_{23} \si_1 = x_{13}
$$
and the fact that the quotient $\PB_3/\N_{\PB_3}$ is Abelian. 

We proved that every element of $X_{\N}$ is a $\GT$-pair for $\N$.
Moreover, since $2m+1$ represents a unit in the ring  $\bbZ/N_{\ord} \bbZ$, 
every $\GT$-pair in $X_{\N}$ is friendly, i.e. the group homomorphism 
$T^{\PB_2}_{m,1} : \PB_2 \to \PB_2 / \N_{\PB_2}$ is onto.

Due to \eqref{act-on-xy} and the second identity in \eqref{act-on-x-c},
we have 
$$
T^{\PB_3}_{m,1} (x_{12})= x_{12}^{2m+1} \N_{\PB_3}\,, 
\quad
T^{\PB_3}_{m,1} (x_{23})= x_{23}^{2m+1} \N_{\PB_3}\,,
\quad
T^{\PB_3}_{m,1} (c) = c^{2m+1} \N_{\PB_3}
$$ 
for every $m \in \bbZ$. 

Since the orders of the elements $x_{12} \N_{\PB_3}$, $x_{23} \N_{\PB_3}$ and $c \N_{\PB_3}$ 
divide $N_{\ord}$ and $2m+1$ represents a unit in $\bbZ/N_{\ord} \bbZ$, 
all three cosets $x_{12} \N_{\PB_3}$,  $x_{23} \N_{\PB_3}$ and $c \N_{\PB_3}$ belong to 
the image of $T^{\PB_3}_{m,1}$. Thus, due to Proposition \ref{prop:onto}, every element 
of $X_{\N}$ is a $\GT$-shadow. 

Furthermore, every $\GT$-shadow in $X_{\N}$ is charming. The first condition of 
Definition \ref{dfn:charm} is clearly satisfied and the second one follows from the 
fact that $2m+1$ represents a unit in $\bbZ/N_{\ord} \bbZ$ and the orders of the elements $x_{12} \N_{\F_2}$,
$x_{23} \N_{\F_2}$ divide $N_{\ord}$.  

Since the inclusion $\GT^{\hs}(\N) \subset X_{\N}$ is obvious, the first statement of 
Theorem \ref{thm:Abelian} is proved.

Let us now show that every $\GT$-shadow in $\GT^{\hs}(\N)$ is genuine. 

Due to Remark \ref{rem:virt-cyclotomic} and the surjectivity of the cyclotomic character, 
we know that, for every $\bar{\la} \in \big( \bbZ/ N_{\ord} \bbZ \big)^{\times}$ there should exist 
at least one genuine $\GT$-shadow $[(m,f)] \in \GT^{\hs}(\N)$ such that 
\begin{equation}
\label{m-lambda}
2 \bar{m} + \bar{1} = \bar{\la}.
\end{equation}
 
Let us assume that $N_{\ord}$ is odd. In this case $\bar{2} \in \big( \bbZ/ N_{\ord} \bbZ \big)^{\times}$ 
and hence, for every fixed $\bar{\la} \in  \big( \bbZ/ N_{\ord} \bbZ \big)^{\times}$, equation
\eqref{m-lambda} has exactly one solution  $\bar{m} \in  \bbZ/ N_{\ord} \bbZ$. 

Since, for every $\bar{\la} \in  \big( \bbZ/ N_{\ord} \bbZ \big)^{\times}$, we have exactly 
one $\GT$-shadow $(\bar{m}, \bar{1})$ in $\GT^{\hs}(\N)$ such that $2\bar{m} +1 = \bar{\la}$, 
the surjectivity of the cyclotomic character implies that every $\GT$-shadow 
in $\GT^{\hs}(\N)$ is genuine.

The case when $N_{\ord} = 2k$ (for $k \in \bbZ_{\ge 1}$) requires more work.
In this case, equation \eqref{m-lambda} has exactly two solutions 
for every $\bar{\la} \in  \big( \bbZ/ 2 k \bbZ \big)^{\times}$.  More precisely, 
if $2\bar{m}+\bar{1} = \bar{\la}$ then the solution set for  \eqref{m-lambda}
is $\{\bar{m}, \bar{m} + \bar{k}\}$.

The proof of the desired statement about $\GT^{\hs}(\N)$ is based on the fact that
the integers $2 m+1$ and $2m+2k+1$ represent two distinct units in the ring  $\bbZ/ 4 k \bbZ$. 

Let $\K$ be an element of $\NFI_{\PB_4}(\B_4)$ satisfying these three properties:
\begin{itemize}

\item $\K \le \N$;

\item $\PB_4/\K$ is Abelian;

\item $4k$ divides $K_0: = |\PB_2 : \K_{\PB_2}|$.

\end{itemize}

One possible way to construct such $\K$ is to define a group homomorphism 
$\psi : \PB_4 \to S_{4k}$ by the formulas
\begin{equation}
\label{psi-4k}
\psi(x_{ij}) := (1,2,\dots, 4k), \qquad \forall ~1 \le i < j \le 4
\end{equation}
and set $\K := \N \cap \ker(\psi)$. 

Since the natural group homomorphism 
$$
\big( \bbZ/K_0 \bbZ \big)^{\times} \to \big( \bbZ/ 4k \bbZ \big)^{\times} 
$$
is onto, there exist $\bar{\la}_1 \neq \bar{\la}_2$ in $\big( \bbZ/K_0 \bbZ \big)^{\times} $ 
whose images in $ \big( \bbZ/ 4k \bbZ \big)^{\times} $ are the two distinct units 
represented by $2m+1$ and $2m+2k+1$, respectively. 

Therefore there exist genuine $\GT$-shadows $[(m_1,1)]$ and $[(m_2,1)]$ in 
$\GT^{\hs}(\K)$ such that 
$$
2 m_1 +1 \equiv \la_1 \mod K_0 ~~~\textrm{and}~~~
2 m_2 +1 \equiv \la_2 \mod K_0\,.
$$
Consequently, $m_1$ and $m_2$ satisfy these congruences mod $4k$: 
$$
2 m_1 +1 \equiv 2m+1 \mod 4k ~~~\textrm{and}~~~
2 m_2 +1 \equiv 2m+2k+1 \mod 4k. 
$$

Thus the images of the genuine $\GT$-shadows  $[(m_1,1)]$ and $[(m_2,1)]$
in $\GT(\N)$ are  $[(m,1)]$ and $[(m+k,1)]$.

\end{proof}
\begin{remark}
\label{rem:Abelian}
Note that, in the Abelian setting, every charming $\GT$-shadow comes from an element of $G_{\bbQ}$. 
The authors \emph{do not know} whether there is a genuine $\GT$-shadow (in the non-Abelian setting) 
that does not come from an element of $G_{\bbQ}$. Of course, if such a $\GT$-shadow exists then 
the homomorphism \eqref{G-Q-to-GTh} is not 
onto\footnote{Some mathematicians believe that, in modern mathematics,  
there are no tools for tackling this question.}. 
\end{remark}

\bigskip
\bigskip

\noindent\textsc{Department of Mathematics,
Temple University, \\
Wachman Hall Rm. 638\\
1805 N. Broad St.,\\
Philadelphia PA, 19122 USA \\
\emph{E-mail addresses:} {\bf vald@temple.edu}, {\bf khanh.q.le@temple.edu}}

\bigskip
\bigskip

\noindent\textsc{Department of Mathematics,
Vanderbilt University, \\
1326 Stevenson Center Ln, Office 1227D \\
Nashville, TN 37240 USA\\
\emph{E-mail addresses:} {\bf aidan.lorenz@vanderbilt.edu}}

\end{document}